\documentclass[noinfoline, preprint]{imsart}
\arxiv{1708.02107v4}
\RequirePackage{natbib}
\usepackage[utf8]{inputenc}
\usepackage[T1]{fontenc}
\usepackage[english]{babel}
\usepackage{dsfont}
\usepackage{hyperref}
\usepackage{color}
\usepackage{graphicx}
\usepackage{epstopdf}
\usepackage{amsmath,amsthm,amssymb, color}
\usepackage{enumitem}
\usepackage{enumerate}
\usepackage{tikz,pgfplots}
\usepackage{a4wide}
\usepackage{lscape}

\newtheorem{theorem}{Theorem}
\newtheorem{lemma}[theorem]{Lemma}
\newtheorem{corollary}[theorem]{Corollary}
\newtheorem{proposition}[theorem]{Proposition}
\newtheorem{remark}{Remark}
\newtheorem*{definition}{Definition}

\def\beq{\begin{equation}}
\def\eeq{\end{equation}}
\def\beqn{\begin{eqnarray*}}
\def\eeqn{\end{eqnarray*}}
\def\bitem{\begin{itemize}}
\def\eitem{\end{itemize}}
\def\benum{\begin{enumerate}}
\def\eenum{\end{enumerate}}
\def\bmult{\begin{multline*}}
\def\emult{\end{multline*}}
\def\bcenter{\begin{center}}
\def\ecenter{\end{center}}

\renewcommand{\corref}[1]{Corollary~\ref{cor:#1}}
\newcommand{\lemref}[1]{Lemma~\ref{lem:#1}}

\DeclareMathOperator*{\argmin}{arg\, min}

\def\bA{\boldsymbol{A}}

\def\bD{\boldsymbol{D}}

\def\bG{\boldsymbol{G}}

\def\bL{\boldsymbol{L}}

\def\bS{\boldsymbol{S}}
\def\bT{\boldsymbol{T}}
\def\bU{\boldsymbol{U}}

\def\bb{\mathbf{b}}

\def\bd{\mathbf{d}}

\def\bp{\mathbf{p}}

\def\bw{\mathbf{w}}

\def\bTheta{\boldsymbol{\Theta}}

\newcommand\bsigma{{\boldsymbol\sigma}}

\def\bbC{\mathbb{C}}

\def\bbF{\mathbb{F}}

\def\bbH{\mathbb{H}}

\def\bbN{\mathbb{N}}
\def\bbO{\mathbb{O}}
\def\bbP{\mathbb{P}}

\def\bbR{\mathbb{R}}
\def\bbS{\mathbb{S}}
\def\bbT{\mathbb{T}}
\def\bb\bU{\mathbb{\bU}}

\newcommand{\E}{\operatorname{\mathbb{E}}}
\renewcommand{\P}{\operatorname{\mathbb{P}}}

\newcommand{\expect}[1]{\mathbb{E}\left(#1\right)}
\newcommand{\pr}[1]{\mathbb{P}\left\{#1\right\}}

\def\\bUnif{\text{\bUnif}}

\newcommand{\1}{\mathds{1}}
\newcommand{\R}{\mathcal{R}}

\usepackage[textwidth=2.8cm, textsize=scriptsize]{todonotes} 
\date{\today}

\begin{document}

\begin{frontmatter}

\title{Adaptive Estimation of Nonparametric Geometric Graphs}
\runtitle{Estimating Structured Graphons}

\begin{aug}
\author{\fnms{Yohann} \snm{De Castro}\ead[label=e1]{yohann.decastro},\ead[label=e5]{yohann.de-castro@ec-lyon.fr}}
\author{\fnms{Claire} \snm{Lacour}\ead[label=e2]{claire.lacour}\ead[label=e4]{claire.lacour@u-pem.fr}}
\and
\author{\fnms{Thanh Mai} \snm{Pham Ngoc}\ead[label=e3]{thanh.pham\_ngoc@math.u-psud.fr}}

\address{Laboratoire de Math\'ematiques d'Orsay,  Univ. Paris-Sud, CNRS,  Universit\'e Paris-Saclay, 91405 Orsay, France \printead{e1,e2,e3}\\
LAMA, Univ Gustave Eiffel, UPEM, Univ Paris Est Creteil, CNRS, F-77447 Marne-la-Vallée, France\\\printead{e4}\\
Institut Camille Jordan, École Centrale de Lyon, 69134 Écully \\\printead{e5}
}  
 \runauthor{De Castro, Lacour and Pham Ngoc}

\end{aug}

\begin{abstract}
{{\it :}\ }
This article studies the recovery of graphons when they are convolution kernels on compact (symmetric) metric spaces. This case is of particular interest since it covers the situation where the probability of an edge depends only on some unknown nonparametric function of the distance between latent points, referred to as Nonparametric Geometric Graphs (NGG).\\
In this setting, adaptive estimation of NGG is possible using a spectral procedure combined with a Goldenshluger-Lepski adaptation method. The latent spaces covered by our framework encompass (among others) compact symmetric spaces of rank one, namely real spheres and projective spaces. For these latter, explicit computations of the eigen-basis and of the model complexity can be achieved, leading to quantitative non-asymptotic results. The time complexity of our method scales cubicly in the size of the graph and exponentially in the regularity of the graphon. Hence, this paper offers an algorithmically and theoretically efficient procedure to estimate smooth NGG.\\
As a by product, this paper shows a non-asymptotic concentration result on the spectrum of integral operators defined by symmetric kernels (not necessarily positive).
\end{abstract}




\end{frontmatter}

\maketitle

\section{Introduction}

Over the recent years, the study of networks has become prevailing in many fields. Through the advent of  social networks, biological neural networks, food webs, protein interaction in genomics and World wide web for instance, large scale data have become available. Extracting information from those repositories of data is a true challenge. Random graphs prove to be particularly relevant to model real-world networks. They are capable to capture complex interactions between actors of a system. Vertices of a random graph usually represent entities of a system and the edges stand for the presence of a specified relation between those entities.  An important statistical problem is seeking better and more informative representations of random graphs. 

Following the seminal work of \cite{Erdos} various random graphs models have been suggested, see \cite{Bollobas,Newman,Kolac, hoff2002,matias2014modeling} and references therein. Aside from classical random graphs, random geometric graphs, see \cite{Penrose, Liben, Parthasarathy} have emerged as an interesting alternative to model real networks having spatial content. Examples include the Internet (where the nodes are the routers) and other physical communication networks such as road networks or neural networks in the brain. Recall that a random geometric graph is an undirected graph  in which each vertex is assigned a latent (unobservable) random label in some metric spaces $\mathbf S$. Two vertices are connected by an edge if the distance between them is smaller than some threshold. Assuming that the underlying metric is the unit sphere~$\mathds S^{\bd-1}$ and latent variables drawn from the uniform distribution on $\mathds S^{\bd-1}$, the paper \cite{bubeck2016testing} considered the problem of testing if the observed graph is an Erd\"os-R\'enyi one (no geometric structure) or a geometric graph on the sphere where points are connected if their distance is smaller than some threshold. 

More generally, random graphs with latent space can be characterized by the so-called graphon. In fact, graphons can be seen as kernel functions for latent position graphs. For more insight about the theory of graphon, we refer to the excellent monograph of \cite{lovasz2012large}.  In the case of graphons defining positive definite kernels, the paper~\cite{tang2013universally} proved that the eigen-decomposition of the adjacency matrix yields consistent estimator of the graphon feature maps involving the latent variables. Besides, nonparametric representations of graphons has gained  attention. Statistical approaches on estimating graphons have been developed using Least-Squares estimation~\cite{klopp2017oracle} or Maximum Likelihood estimation \cite{wolfe2013nonparametric}. Dealing with estimation of (sparse) graphons from the observation of the adjacency matrix, the paper \cite{klopp2017oracle} derives sharp rates of convergence for the $\mathbf L^2$ loss for the Stochastic Block Model. We mention also the general methodology, referred to as USVT algorithm, of \cite{chatterjee2015matrix} that can be invoked to control the $\mathbf L^2$ loss between the probability matrix and a eigenvalue-tresholded version of the adjacency matrix. However, note that the present paper is more concerned (see Theorem~\ref{thm:KG_revisited}) by controlling the distance between the probability matrix and its integral operator. The USVT point of view has been further investigated in \cite{xu2017rates} 
that gives rates of convergence for the so-called ‘‘{\it probability matrix}\,'' estimation problem (see Section~\ref{sec:proba}), under smoothness assumptions. Motivated by sharp control of individual eigenvalues behavior (invoking Weyl's perturbation Theorem, see \cite[page 63]{bhatia2013matrix}), we choose to control the difference between the adjacency matrix and the probability matrix in operator norm, see Proposition~\ref{prop:BvH}. 


\subsection{A Statistical Pledge for Structured Latent Spaces}
The graphons are limiting objects that describe large dense graphs. The graphon model~\cite{lovasz2012large} is standardly and without loss of generality formulated choosing~$[0,1]$ as latent space. In this model, given latent points $x_1,\ldots,x_n\in[0,1]$, the probability to draw an edge between $i$ and~$j$ is $W(x_i,x_j)$ where $W$ is a function from $[0,1]^2$ onto $[0,1]$, referred to as a graphon. This model is general and well referenced in the literature\textemdash as mentioned earlier, the reader may consult the book \cite{lovasz2012large} for further details. 

However, this model may underneath intrinsic features of a random graph. For instance, recall the prefix attachment graph model \cite[page 190]{lovasz2012large} where the nodes are added one at a time and each new node connects to a random previous node and all its predecessors. In this model, the graph sequence converges in cut distance \cite[Proposition 11.42]{lovasz2012large} to the graphon $W_{\mathrm{pref}}$ defined as, for all $(x_1,y_1),(x_2,y_2)\in[0,1]^2$,
\begin{equation}
\label{eq:pref}
 W_{\mathrm{pref}}((x_1,y_1),(x_2,y_2))=\mathds{1}(x_1<x_2y_2)+\mathds{1}(x_2<x_1y_1)\,,
\end{equation}
up to a measure preserving homomorphism of the latent space $[0,1]^2$. From a statistical point of view, the estimation of the function $((x_1,y_1),(x_2,y_2))\mapsto\mathds{1}(x_1<x_2y_2)$ from sample points~$((x_k,y_k))_k$ uniformly distributed on~$[0,1]^2$ is a well understood standard task. 

Yet one may also represent this graphon on the standard latent space $[0,1]$. And, in this case, one cannot represent this graphon using the indicator function of two symmetric convex sets with piecewise smooth border as done in~\eqref{eq:pref}. 
Actually, in this case, a fractal-like structure appear and the statistical estimation of this function seems more difficult than in~\eqref{eq:pref}. 
Our statement may be loose here but one may emphasize that there may exist better latent spaces than~$[0,1]$ on which the graphon may present a simple and better estimable formulation. 

An other important statistical issue is that, by construction, graphons are defined on an equivalent class ‘‘{\it up to a measure preserving homomorphism}'' and it can be challenging to have a simple description from an observation given by sampled graph\textemdash since one has to deal with all possible composition of a bivariate function by any measure preserving homomorphism. In this paper, we circumvent this disappointing statistical issue restraining our attention to graph models for which the probability of appearance of an edge depends as a nonparametric function of  the distance between latent points.

\subsection{Main results}

In this paper, we focus on latent metric spaces for which the distance is invariant by translation (or conjugation) of pairs of points. This natural assumption leads to consider that the latent space~$\mathbf S$ has some group structure, namely it is a compact Lie group or some compact symmetric space (‘‘{\it Intuitively, a symmetric space is a Riemannian manifold where geodesics are ‘‘nicely'' arranged in a symmetric way around any point}''\cite[Chapter 3]{plmHDR}). Hence, consider graphons defined as functions $\mathbf p$ of (the cosine of) the distance~$\gamma$ (normalized so that the range of~$\gamma$ equals $[0,\pi]$) of some compact Lie group~$\mathbf S$, or more generally of some compact symmetric space, see Section~\ref{sec:LatentSpaces}. In this case, the graphon is given by 
\[
\forall x,y\in\mathbf S,\quad W(x,y)=\mathbf p(\cos \gamma(x,y))=\mathbf p(\cos \gamma(z,e))\quad\mathrm{and}\quad z=xy^{-1}
\,,
\] 
where $y^{-1}$ is the inverse of $y$, $e$ denotes the identity element of $\bS$ and $\mathbf p$ is a function from $[-1,1]$ onto~$[0,1]$ referred to as the ‘‘{\it envelope}''. In the case when~$\bS$ is the Euclidean sphere, we consider graphons that are a function~$\bp$ of cosine of the distance, namely $\cos \gamma(x,y)=\langle x,y\rangle$, between latent points $x,y\in\bS$. 
In this case $e$ is the north pole and $\mathbf p(\langle x,y \rangle )= \mathbf p(\langle z,e \rangle)$ where $z$ is the image of~$x$ by the rotation which maps $y$ to $e$.

\pagebreak[3]

First, note that $W$, viewed as an integral operator on square-integrable functions, is a compact convolution (on the left) operator:
\[
f\in \mathbf L^2(\mathbf S) \mapsto \int_{\mathbf S} \mathbf p(\cos \gamma(x,.)) f(x) dx 
\in \mathbf L^2(\mathbf S)
\]
Then the main point is that there exists an $\mathbf L^2$-decomposition of the Hilbert space of square integrable functions such that the eigenfunctions basis of the convolution kernel (and the graphon $W$ viewed as a integral operator) depends only on the latent space $\bS$ and does not depend on the function~$\mathbf p$. This basis is the irreducible characters in the (Lie) Group case and the zonal spherical functions in the non-Group case, see Cartan’s Classification of ‘‘sscc'' Lie Groups and ‘‘ssccss'' in Section~\ref{sec:LatentSpaces} for further details. This decomposition can be pushed on~$[-1,1]$ and one gets an $\mathbf L^2$-decomposition of the envelope function~$\mathbf p$ such that the orthonormal basis~$(Z_\ell)$ depends only on the latent space $\bS$ and does not depend on $\mathbf p$, see~\eqref{eq:decompoP}. Furthermore, the eigenvalues $\lambda^{\star}=(\lambda^{\star}_k)_{k\geq0}$ of the kernel~$W$ are exactly (up to some known multiplicities and up to  some known multiplicative constants) the coefficients of $\mathbf p$ onto the orthonormal basis $(Z_\ell)_{\ell\geq0}$. Hence, the graphon~$W$ is entirely described by the univariate function $\mathbf p$ defined on $[-1,1]$. It follows that this subclass of graphons may be well suited for estimation since it reduces to estimate a simple univariate function on~$[-1,1]$. 

Now, consider the case when $\bS$ is one of the compact symmetric space of rank one\textemdash namely real spheres or real/complex/quaternionic/octonionic projective spaces. In this case, one can explicitly give the decomposition of the envelope function $\mathbf p$. One can prove that the orthonormal polynomials~$(Z_\ell)_{\ell\geq0}$ are the orthonormal polynomials (more precisely, normalized Jacobi polynomials) of some Beta law with known shape parameters $(\alpha,\beta)$, see Table~\ref{t:CSR1} in the Appendix for the explicit values. This decomposition is given by
\beq
\label{eq:decompoP}
\bp=\sum_{\ell\geq0}  \sqrt{d_\ell}\bp_\ell^\star\, Z_\ell\quad\mathrm{and}\quad{\bp}_\ell^\star=\frac1{\sqrt{d_\ell}}\langle \bp,Z_\ell\rangle_{\mathbf L^2([-1,1],\bw)}\,,
\eeq
in $\mathbf L^2([-1,1],\bw)$ where $\bw$ denotes the density function of the Beta distribution. We further assume that there exists $s>0$, a (Sobolev) regularity parameter, such that
\[
\forall R\geq1\,,\quad
\sum_{\ell >R} d_\ell (\bp_\ell^\star)^2 \leq C(\bp,s,\bS)\, R^{-2s }.
\]
for some constant $ C(\bp,s,\bS)>0$ and for some known dimensions $(d_\ell)_{\ell\in\bbN}$ (given by the representation of the group/quotient $\bS$) that depend only on $\bS$, see Table~\ref{t:CSR1} in the Appendix. This assumption governs the regularity of the kernel $W$ and it can be understood that the derivative of order~$s$ (in the Laplacian on $\bS$ sense) of $W$ is square-integrable. In this case, one can build an estimator $\widehat \lambda^{\widehat R}$ (from the spectrum of the adjacency matrix of the graph) of the spectrum $\lambda^{\star}$ of $W$ (viewed as an integral operator) such that 
$$\E\left[\delta_2^2\left(\widehat \lambda^{\widehat R}, \lambda^{\star}\right)\right]
=\mathcal{O}\left[\Big(\frac{n}{\log n }\Big)^{-\frac{2 s }{2s +(\bd-1)}}\right],$$
where $n$ is the size of the graph, $\mathbf d$ is the dimension of the latent space (actually, $\bS$ is a ($\bd -1$)-manifold) and $\delta_2$ is the $\ell_2$ distance between spectra, see~\eqref{eq:Hardy_Littlewood_L2} for a definition. We uncover for the minimax risk, the rate of estimating a $s$-regular function on a space of (Riemannian) dimension $\bd-1$ up to a multiplicative log factor. This result is stated in Theorem~\ref{thm:New_Main_Sphere} without adaptation to the smoothness parameter, Theorem~\ref{thm:adap} and Corollary~\ref{vitesse} with smoothness adaptation, and Theorem~\ref{cor:variancepr} and Corollary~\ref{u} for adaptive estimation of the envelope function $\mathbf p$ at rate $\mathcal{O}(\log n/n)$ when $\mathbf p$ is a polynomial. The general statement for compact symmetric spaces is given by Theorem~\ref{thm:SSCCSSversion}.


Note that our results hold for general convolution kernels and not necessarily semidefinite positive kernels. Indeed, it is often assumed in the literature, see for instance \cite{ferreira2008integral,rosasco2010learning,tang2013universally,tang2017nonparametric}, that the graphon~$W$ is a semidefinite positive kernel. {If one assumes that the graphon is a positive definite kernel (leading to a RKHS representation) then  the probability matrix (see Section~\ref{sec:proba}) of the random graph is almost surely semidefinite positive (by definition of positive definite kernels). In this case, the empirical eigenvalues (the eigenvalues of the adjacency matrix) might be negative but these negative empirical eigenvalues converge to nonnegative limiting eigenvalues as the graph size grows, which is a strong requirement. See also Remark~\ref{rem:RKHS} on this point. To bypass this limitation, our approach does not use any RKHS representation but a new non-asymptotic concentration result on the integral operator, see Theorem~\ref{thm:KG_revisited} and Corollary~\ref{cor:Canonical_Kernels}. In particular, this framework is consistent with negative empirical eigenvalues clustering around negative limiting eigenvalues.} The rates uncovered by these results allow us to introduce an adaptive estimation procedure of the spectrum of the graphon.

\pagebreak[3]

From a computational point of view, Theorem~\ref{thm:TrueComplexity} enlightens on the time complexity of our estimator. Remarkably, the time complexity is $n^3+(R_{\max}+2)!$, that is cubic in the graph size $n$ (as any spectral method) and exponential in the number of coefficients $\bp_\ell^\star$ one has to estimate. 
The spatial complexity is quadratic in $n$ as one has to store the adjacency matrix of the graph.

%

\subsection{Outline}
The convergence of the spectrum of the ‘‘{\it matrix of probabilities}'' towards the spectrum of the integral operator in a non-asymptotic frame is given in Section~\ref{sec:graphon_Model}. 

Then, we begin our study by a comprehensive example on the $\mathbf d$-dimensional sphere in Section~\ref{sec:Sphere_Example}. Interestingly, we uncover that the spectrum of the graphon (viewed as a kernel operator) presents a structure: the eigenvalues have prescribed multiplicities and the eigenvectors are fixed\textemdash they are the spherical harmonics. 

Adaptive estimation of the spectrum of the graphon $W$ (viewed as an integral operator) is proved and computational complexities are discussed. 

Extensions to compact symmetric spaces is done in Section~\ref{sec:LatentSpaces}. Numerical experiments 
are presented in Section~\ref{sec:numeric}. 

The proofs are given in the appendix.

\section{Spectral Convergence of the Sampled graphons}
\label{sec:graphon_Model}


\subsection{Estimating the Matrix of Probabilities}
\label{sec:proba}
We denote $[n]:=\{1,\ldots,n\}$ for all $n\geq1$. Consider a random undirected graph $\bG$ with $n$ nodes and assume that we observe its ${n\times n}$ adjacency matrix $\bA$ given by entries $\bA_{ij}\in \{0,1\}$ where $\bA_{ij}=1$ if the nodes $i$ and $j$ are connected and $\bA_{ij}=0$ otherwise. 

We set $\bA_{ii}=0$ on its diagonal entries for all $i\in[n]$ and we assume that~$\bA_{ij}$ are independent Bernoulli random variables with $(\bTheta_0)_{ij}:=\P\{\bA_{ij}=1\}$ for $1\le i<j\le n$. 

We denote by $\bTheta_0$ the $n\times n$ symmetric matrix with entries~$(\bTheta_0)_{ij}$ for $1\le i<j\le n$ and zero diagonal entries. This is a matrix of probabilities associated to the random graph~$\bG$. 

Throughout this paper, we denote by 
\beq
\label{def:Tn}
\widehat{\bT}_n:=(1/n)\,\bA\quad\mathrm{and} \quad \bT_n:=(1/n)\,\bTheta_0\,.
\eeq
Our analysis leverages the operator norm~$\|\cdot\|$ loss to account for the distance between the observation $\widehat{\bT}_{n}$ and the target parameter ${\bT}_{n}$. 

Furthermore, a near optimal error bound can be derived for the operator norm~$\|\cdot\|$ loss as shown in~\cite{bandeira2014sharp}.

\pagebreak[3]

\begin{proposition}[
\cite{bandeira2014sharp}]
\label{prop:BvH}
There exists a universal constant $C_0>0$ such that for all $\alpha\in(0,1)$, it holds
\beq
\label{eq:BvH}
\pr{\|\widehat{\bT}_{n}-{\bT}_{n}\|\geq3\frac{\sqrt{2\bD_0}}n+C_0\frac{\sqrt{\log(n/\alpha)}}n}\leq\alpha
\eeq
where $\displaystyle\bD_0=\max_{i\in[n]}\Big[\sum_{j\in[n]} (\bTheta_0)_{ij}(1-(\bTheta_0)_{ij})\Big]\leq n/4$. 
\end{proposition}


\noindent
A proof is recalled in Appendix~\ref{proof:BvH}. Proposition \ref{prop:BvH} is of particular interest giving an error bound on each eigenvalue~$\lambda_k({\bT}_{n})$ of ${\bT}_{n}$, where $\lambda_k(M)$ denotes the $k$-th largest eigenvalue of the symmetric matrix $M$. Indeed, it holds, with probability greater that $1-n\exp(-n)$,
\beq
\label{eq:Weyl}
\forall k\in[n],\quad
|\lambda_k(\widehat{\bT}_{n})-\lambda_k({\bT}_{n})|\leq\|\widehat{\bT}_{n}-{\bT}_{n}\|=\mathcal{O}(1/\sqrt{n})\,,
\eeq
by Weyl's perturbation Theorem, see \cite[page 63]{bhatia2013matrix} for instance.


\subsection{On the Kernel Spectrum}
We understand that the spectrum  of $\widehat{\bT}_{n}$ can be a good approximation of the spectrum of ${\bT}_{n}$ in the sense of~\eqref{eq:Weyl}. Assuming a graphon $W$ model we can link the spectrum of ${\bT}_{n}$ (sampled graphon onto the latent points $X_1,\ldots,X_n$ see below) to the spectrum of an integral operator $\bbT_W$ defined by the graphon~$W$ viewed as a symmetric kernel. More precisely, we consider $J:=(\bS,\mathcal A,\bsigma)$ a probability space on~$\bS$ endowed with measure $\bsigma$ on the $\sigma$-algebra~$\mathcal A$\!, and $W:\bS\times\bS\to[0,1]$ a symmetric $\bsigma$-measurable function. The couple~$(J,W)$ is referred to as a graphon, see for instance \cite[Chapter 13]{lovasz2012large}. We then define a probabilistic model on $\bTheta_0$ setting 
\[
(\bTheta_0)_{i,j}=W(X_i,X_j) \mbox{ for } i\neq j \mbox{ and } 0 \mbox{ otherwise } 
\]
where $X_1,\ldots,X_n$ are i.i.d. drawn w.r.t. $\bsigma$. Assume that the kernel satisifies $W\in\bL^2(\bS\times\bS, \bsigma\otimes\bsigma)$, so that
\begin{equation}
\nonumber\label{eq:defoperateurintegral}
\forall x\in \bS,\ \forall g\in\mathbf L^2(\bS,\bsigma),\quad (\bbT_Wg)(x)=\int_{\bS} W( x,y)g(y)\mathrm d\bsigma(y)\,,
\end{equation}
defines a symmetric Hilbert-Schmidt operator $\bbT_W$ on $\mathbf L^2(\bS,\bsigma)$ and we can invoke the spectral theorem. Hence, it holds that, in the $\bL^2(\bS\times\bS,\bsigma\otimes\bsigma)$-sense,
\beq
\label{eq:Spectral_Decomposition}
\mbox{for almost every } x,y\in\bS,\quad
W(x,y)=\sum_{k\geq1}\lambda_k^\star \phi_k(x)\phi_k(y)\,,
\eeq
for an $\bL^2(\bS,\bsigma)$-orthonormal basis $(\phi_i)_{i\geq1}$. This operator has a discrete spectrum, i.e. a countable multiset~$\lambda^\star $ of nonzero (real) labeled eigenvalues $(\lambda_k^\star)_{k\geq1}$ such that $\lambda_k^\star \to0$. In particular, every nonzero eigenvalue has finite multiplicity. We are free to choose any labeling of the target eigenvalues $(\lambda_k^\star)_{k\geq1}$ and observe that our results are valid for any choice of labeling. For instance, we can standardly label the eigenvalues in decreasing order with respect to their absolute values such that $|\lambda^\star_1|\geq|\lambda^\star_2|\geq\cdots$ and this gives results whose error rates (typically~$\|W-W_R\|_2$ see below) are in terms of the best $\bL^2$-approximation of rank $R$ of the kernel $W$. An other choice may result in labeling the eigenvalues in increasing order of ‘‘frequencies''. This labeling is natural for instance when we have a representation by spherical harmonics of the kernel as in Section~\ref{sec:Sphere_Example}. This gives results whose error rates are in terms of the best approximation by low frequency (i.e. the $R$ first frequencies) kernels.

\subsection{The relatively sparse model}
\label{sec:relativelySparse}
Note that the average degree of node $i$ is $\sum_{j\in[n]} (\bTheta_0)_{ij}$ which is of the order of the graph size $n$ in the graphon model for which $(\bTheta_0)_{ij}=W(X_i,X_j)$. To gain in realism, one may consider a model where 
\[
(\bTheta_0)_{ij}=\zeta_nW(X_i,X_j)
\]
where $\zeta_n$ is a sequence of positive real numbers that may converge to zero. In this model, the average degree of one node is of the order of $n\zeta_n$. One standard interpretation is that edges are drawn independently with probability $W(X_i,X_j)$ and we independently suppress these edges with probability $1-\zeta_n$. The {\it relatively sparse model} \citep{wolfe2013nonparametric} is given by sequences $\zeta_n$ such that
\begin{equation}
\label{eq:relativelysparse}
\liminf_n \frac{n\zeta_n}{\log n}\geq Z\,,
\end{equation}
where $Z>0$ is a universal positive constant. In this model, the average degree of one node is at least $\mathcal{O}(\log n)$. This latter rate is a standard threshold on connectedness in random graphs \citep{Bollobas}. 
More precisely, it is known that in the Erd\"os-R\'enyi model, the sharp threshold on connectedness is $\zeta_n=\log(n)/n$. The interested reader may also consult further works on percolation on graphons in the {\it sparse} regime (where $\zeta_n=c/n$ for some constant $c>0$), see \citep{bollobas2010percolation} for instance.
 Note that if $\zeta_n=1$ then we recover the previous model, referred to as the ‘‘{\it \!dense\,}'' regime.

Note that 
\[
\bbT_{\zeta_n W}=\zeta_n \bbT_{W}\quad\mathrm{and}\quad \bT_n:=(\zeta_n W(X_i,X_j)/n)_{i,j}=\zeta_n(W(X_i,X_j)/n)_{i,j}\,.
\]
By homogeneity, we understand that one may consider that $\zeta_n=1$ when studying the convergence of~$\bT_n$ towards $\bbT_{\zeta_n W}$. 

However, the situation is more intricate for the convergence of $\widehat \bT_n$ towards~$\bT_n$. Given a fixed graphon model $W$, one has the bound 
\[
\|\bT_n\|= \zeta_n\|(W(X_i,X_j)/n)_{i,j}\|=\mathcal{O}_{\mathds P}(\zeta_n)\,,
\]
where $\mathcal{O}_{\mathds P}$ denotes stochastic boundedness and since the operator norm of $(W(X_i,X_j)/n)_{i,j}$ converges to the largest absolute eigenvalue of $\bbT_W$, see Section~\ref{sec:EstimatingKernel}. On the other hand, the control~\eqref{eq:BvH} is given by: {\it There exists a universal constant $C_0>0$ such that for all $\alpha\in(0,1)$, it holds
\beq
\label{eq:BvH2}
\pr{\|\widehat{\bT}_{n}-{\bT}_{n}\|\geq3\sqrt{\frac{2\zeta_n}n}+C_0\frac{\sqrt{\log(n/\alpha)}}n}\leq\alpha
\eeq
using that $\displaystyle\bD_0=\max_{i\in[n]}\Big[\sum_{j\in[n]} (\bTheta_0)_{ij}(1-(\bTheta_0)_{ij})\Big]\leq n\zeta_n$}. It gives
\[
\|\widehat \bT_n-\bT_n\|=\mathcal{O}_{\mathds P}\Big(\sqrt\frac{{\zeta_n}}{n}+\frac{\sqrt{\log n}}{n}\Big)\,.
\]
Under the relatively sparse model assumption~\eqref{eq:relativelysparse}, one has
\[
\|\widehat \bT_n-\bT_n\|=\mathcal{O}_{\mathds P}\left(\sqrt\frac{{\zeta_n}}{n}\right)\quad\mathrm{and}\quad\|\bT_n\|=\mathcal{O}_{\mathds P}(\zeta_n)
\]
entailing that $\widehat \bT_n$ is a better approximation of $\bT_n$ than the null matrix. While, for faster rate, namely $\zeta_n=o(\sqrt{\log n}/n)$ one has
\[
\|\widehat \bT_n-\bT_n\|=\mathcal{O}_{\mathds P}\Big(\frac{\sqrt{\log n}}{n}\Big)\quad\mathrm{and}\quad\|\bT_n\|=\mathcal{O}_{\mathds P}(\zeta_n)=o_{\mathds P}\Big(\frac{\sqrt{\log n}}{n}\Big)
\]
entailing that the null matrix is a better approximation of $\bT_n$ than the observation $\widehat \bT_n$. This short argumentation shows that the relatively sparse model~\eqref{eq:relativelysparse} ensures the observation $\widehat \bT_n$ is at least more informative than the null matrix for the operator norm topology.

To conclude, we will adopt two conventions. First, we will consider that $\zeta_n=1$ when studying the convergence of $\bT_n$ towards $\bbT_{\zeta_n W}$. Second, every results based on statistics of $\widehat \bT_n$ will be presented in the relatively sparse model~\eqref{eq:relativelysparse} in a joint remark, see Section~\ref{sec:NEKS}.

\subsection{Non-Asymptotic Error Bounds  in $\delta_2$-metric}
\label{sec:EstimatingKernel}

Given two sequences $x$ and $y$ of real numbers\textemdash completing finite sequences by zeros\textemdash  such that it holds $\sum x_i^2+y_i^2<\infty$, we standardly define the $\ell_2$-rearrangement distance $\delta_2(x,y)$ as
\[
\delta_2(x,y):=\inf_{\pi\in\mathcal P}\Big[\sum(x_i-y_{\pi(i)})^2\Big]^{\frac12}\,,
\]
where the infimum is taken over $\mathcal P$ the set of permutations with finite support. Using Hardy-Littlewood rearrangement inequality \cite[Theorem 368]{hardy1952inequalities}, it is standard to observe that
\beq
\label{eq:Hardy_Littlewood_L2}
\delta_2(x,y)=\lim_{N\to\infty}\left[\sum_{k=-N}^N(x_k-y_k)^2\right]^{\frac12}\,,
\eeq
with the convenient notation $x_{-1}\leq x_{-2}\leq\ldots\leq0\leq\ldots\leq x_2\leq x_1\leq x_0$ (respectively $y_{-1}\leq y_{-2}\leq\ldots\leq0\leq\ldots\leq y_2\leq y_1\leq y_0$) where we denote $x=(x_k)_{k\in\mathbb Z}$ (respectively $y=(y_k)_{k\in\mathbb Z}$) completing with zeros if necessary. Using this metric we can compare the (finite) spectrum $\lambda({\bT}_{n})$ of ${\bT}_{n}$ to the (infinite) spectrum~$\lambda^\star $ of $\bbT_W$. 

{
\begin{remark}
\label{rem:RKHS}
To the best of our knowledge, existing results on this issue assume that $W$ is a {\sf positive kernel} and use a RKHS representation and/or Mercer theorem. This assumption might seem meaningless for a graphon. Indeed, it implies that $\bbT_W$ is semi-definite $($contrary to the present paper$)$ and if $W=W_H$ is a ‘‘step-function'' kernel representing a finite graph $H$, it implies that the adjacency matrix of $H$ is semi-definite which might be seen as restrictive. In this article, we bypass this limitation with the next result based on the analysis developed in \cite{koltchinskii2000random} and some recent development in random matrix concentration, see \cite{tropp2012user} for instance.
\end{remark}
}

\begin{theorem}
\label{thm:KG_revisited}
Let $W\in\bL^2(\bS\times\bS,\bsigma\otimes\bsigma)$ be a symmetric kernel and let~$(\phi_k)_{k\geq1}$ be an orthonormal eigenbasis as in~\eqref{eq:Spectral_Decomposition}. Let $R\geq1$ and $\alpha\in(0,1/3)$. Set
\[
\rho(R):=\max\Bigg[1,\Big\|\sum_{r=1}^R\phi_r^2\Big\|_\infty-1\Bigg]\quad\mathrm{and}\quad 
W_R(x,y):=\sum_{i=1}^R\lambda_i^\star\phi_i(x)\phi_i(y).
\]
Then, for all $n^3\geq\rho(R)\log(2R/\alpha)$, it holds 
\begin{align*}
\delta_2(\lambda({\bT}_{n}),\lambda^\star )\leq&
2\|W-W_R\|_2+\|W-W_R\|_\infty\Big[{\frac{2\log(2/\alpha)}n}\Big]^{\frac14}\\
&+\|W_R\|_2
\Bigg[
\Big[\frac{\rho(R)\log(2R/\alpha)}n\Big]^{\frac12}\\
&+
\Big[\frac{2\rho(R)}{n}\Big(1+\max_{1\leq r\leq R}||\phi^2_r||_\infty \sqrt{{\frac{\log(R/\alpha)}{{2n}}}}\Big)\Big]^{\frac12}
\Bigg],
\end{align*}
with probability at least $1-3\alpha$.
\end{theorem}

\noindent
A proof of Theorem \ref{thm:KG_revisited} can be found in Appendix~\ref{proof:KG_revisited}. This result shows that for all $n\geq n_0(R)$, it holds that $\delta_2(\lambda({\bT}_{n}),\lambda^\star)\leq 2\|W-W_R\|_2+{C_0(R)}\,{n^{-\frac14}}$ with probability at least $1-3\alpha$, where the constants $n_0(R)\geq1$ and $C_0(R)>0$ may depend on $R$, the orthogonal basis $(\phi_k)_{k\in[R]}$, $\alpha$ and the graphon $W$. 

We have the following improvement for canonical kernels, see \cite[Definition 3.5.1]{de2012decoupling} for a definition.

\pagebreak[3]

\begin{remark}
Under the relatively sparse model~\eqref{eq:relativelysparse}, Theorem \ref{thm:KG_revisited} becomes: for all $n^3\geq\rho(R)\log(2R/\alpha)$, it holds 
\begin{align*}
\delta_2(\zeta_n\lambda({\bT}_{n}),\zeta_n\,\lambda^\star )\leq&
2\,\zeta_n\,\|W-W_R\|_2+\zeta_n\,\|W-W_R\|_\infty\Big[{\frac{2\log(2/\alpha)}n}\Big]^{\frac14}\\
&+\zeta_n\,\|W_R\|_2
\Bigg[
\Big[\frac{\rho(R)\log(2R/\alpha)}n\Big]^{\frac12}\\
&+
\Big[\frac{2\rho(R)}{n}\Big(1+\max_{1\leq r\leq R}||\phi^2_r||_\infty \sqrt{{\frac{\log(R/\alpha)}{{2n}}}}\Big)\Big]^{\frac12}
\Bigg],
\end{align*}
with probability at least $1-3\alpha$. In the aforementioned setting, we have denoted the eigenvalues of~$\bbT_W$ by $\lambda^\star$ (as before) so that $\zeta_n\,\lambda^\star$ are the eigenvalues of  $\bbT_{\zeta_n W}$, and their empirical counterpart (based on the probability matrix) by $\zeta_n\lambda({\bT}_{n})$, namely the eigenvalues of $\zeta_n {\bT}_{n}$. 
\end{remark}

\pagebreak[3]

\begin{corollary}
\label{cor:Canonical_Kernels}
Assume further that the kernel $(W-W_R)^2(x,y)-\expect{(W-W_R)^2}$ is canonical, namely
\[
\mbox{For almost every } x\in\bS,\quad\expect{(W-W_R)^2(x,X_1)}=\expect{(W-W_R)^2(X_1,X_2)}\,,
\]
then there exist universal constants $C_1,C_2>0$ such that for all $n^3\geq\rho(R)\log(2R/\alpha)$, it holds 
\begin{align*}
\delta_2(\lambda({\bT}_{n}),\lambda^\star )\leq&
2\|W-W_R\|_2+\|W-W_R\|_\infty\Big[{\frac{C_1\log(C_2/\alpha)}n}\Big]^{\frac12}\\
&+\|W_R\|_2
\Bigg[
\Big[\frac{\rho(R)\log(2R/\alpha)}n\Big]^{\frac12}\\
&
+
\Big[\frac{2\rho(R)}{n}\Big(1+\max_{1\leq r\leq R}||\phi^2_r||_\infty \sqrt{{\frac{\log(R/\alpha)}{{2n}}}}\Big)\Big]^{\frac12}
\Bigg],
\end{align*}
with probability at least $1-3\alpha$.
\end{corollary}
\noindent
A proof of this corollary can be found in Appendix~\ref{proof:CanonicalKernels}. 

\begin{remark}
Under the relatively sparse model~\eqref{eq:relativelysparse}, Corollary \ref{cor:Canonical_Kernels} becomes: for all $n^3\geq\rho(R)\log(2R/\alpha)$, it holds 
\begin{align*}
\delta_2(\zeta_n\lambda({\bT}_{n}),\zeta_n\lambda^\star )\leq&
2\zeta_n\,\|W-W_R\|_2+\zeta_n\,\|W-W_R\|_\infty\Big[{\frac{C_1\log(C_2/\alpha)}n}\Big]^{\frac12}\\
&+\zeta_n\,\|W_R\|_2
\Bigg[
\Big[\frac{\rho(R)\log(2R/\alpha)}n\Big]^{\frac12}\\
&
+
\Big[\frac{2\rho(R)}{n}\Big(1+\max_{1\leq r\leq R}||\phi^2_r||_\infty \sqrt{{\frac{\log(R/\alpha)}{{2n}}}}\Big)\Big]^{\frac12}
\Bigg],
\end{align*}
with probability at least $1-3\alpha$. In the aforementioned setting, we have denoted the eigenvalues of~$\bbT_W$ by $\lambda^\star$ (as before) so that $\zeta_n\,\lambda^\star$ are the eigenvalues of  $\bbT_{\zeta_n W}$, and their empirical counterpart (based on the probability matrix) by $\zeta_n\lambda({\bT}_{n})$, namely the eigenvalues of $\zeta_n {\bT}_{n}$. 
\end{remark}


\section{The Sphere Example, Prelude of Symmetric Compact Spaces}
\label{sec:Sphere_Example}
From a general point of view, this article focuses on the case where the value $W(x,y)$ depends on a nonparametric function $\bp$ of the distance between the points $x$ and~$y$ of a latent space $\bS$ assumed a compact Lie group or a compact symmetric space, see Section~\ref{sec:LatentSpaces} for further details. Such assumptions on the graphon $W$  allow to lead the spectral analysis a step further. In this section, we restrict our analysis to the pleasant case of $\bS:=\mathds S^{\bd-1}$  the unit sphere of $\mathbb R^{\bd}$ with $\bd \geq 3$ equipped with the uniform probability measure $\bsigma$ and the usual scalar product $\langle\cdot,\cdot\rangle$. In the literature, a popular model is given by the Random Geometric Graph for which the value $W(x,y)$ depends on the distance between the points~$x$ and $y$ of the latent space $ \mathds S^{\bd -1}$ and $W(x,y)=\mathds 1_{\langle x,y\rangle\geq\tau}$ for some threshold~$\tau\in(-1,1)$ as in~\cite{devroye2011high,bubeck2016testing}. 
From now on, assume that $W$ only depends  on the distance between latent points, namely
\[
\forall x,y\in \mathds S^{\bd-1},\quad W(x,y)=\bp(\langle x,y\rangle)
\]
where $\bp:[-1,1]\to[0,1]$ is an unknown function that is to be estimated. We refer to $\bp$ as the  ‘‘{\it envelope}'' function.

\subsection{Harmonic Analysis on $ \mathds S^{\bd -1}$ }
Let us start by providing a brief overview on Fourier analysis on $\mathds  S^{\bd -1}$. As pointed out above, in this case the operator~$\bbT_W$ is a  convolution (on the left) operator.  Its spectral decomposition~\eqref{eq:Spectral_Decomposition} satisfies that the orthonormal basis~$(\phi_k)_k$ does not depend on~$\bp$ and the spectrum~$\lambda(\bbT_W)$ is exactly described by the Fourier coefficients~$(\bp^\star_\ell)_\ell$ of~$\bp$, see \cite[Lemma 1.2.3]{dai2013approximation}. This remark remains true when the latent space~$\bS$ is assumed a compact Lie group or a compact symmetric space, see Section~\ref{sec:LatentSpaces} for further details.

In the spherical case, the orthonormal  basis of eigenfunctions consists of  the real spherical harmonics. The following material can be found in \cite{dai2013approximation}.
Let us denote $\mathcal H_\ell$ the space of real spherical harmonics of degree $\ell$ with orthonormal basis $(Y_{\ell j})_{j\in[d_\ell]}$ where
\beq
\label{dl}
 d_\ell:=\textrm{dim}( \mathcal{H}_\ell) = \binom {\ell+\bd -1}{\ell} -\binom{\ell +\bd -3}{\ell -2}
 \eeq
 for $\ell\geq 2$ and $d_0=1$, $d_1=\bd$.
 Note that the $d_\ell$'s are all distinct and of order $\ell^{\bd-2}$. We discard $\mathds S^{1}$  from our analysis because in that case, the $d_\ell$'s are constant equal to $2$.
In the sequel we identify $(\phi_k)_{k\geq1}=(Y_{\ell j})_{\ell\geq0,\,j\in[d_\ell]}$ so that the spectral decomposition~\eqref{eq:Spectral_Decomposition} reads
\beq
\label{eq:Decomposition_Kernel_Sphere}
\forall x,y\in\mathds S^{\bd-1},\quad
W(x,y)=\bp(\langle x,y\rangle)=\sum_{\ell\geq0}\bp^\star_\ell
\,\Big[\underbrace{\sum_{j=1}^{d_\ell}Y_{\ell j}(x)Y_{\ell j}(y)}_{\mathrm{Zonal\ Harmonic}}\Big]\,,
\eeq
where $\lambda^\star =\{\bp_0^\star,\bp_1^\star,\ldots, \bp_1^\star, \ldots,\bp_\ell^\star,\ldots, \bp_\ell^\star,\ldots\}$ and  $ \sum_{j=1}^{d_\ell}Y_{\ell j}(x)Y_{\ell j}(y)$ is a zonal harmonic of degree $\ell$. The eigenvalue $\bp_\ell^\star$ has multiplicity $d_\ell$ if the eigenvalues are all distinct. Furthermore, it holds that 
\[
\bp_\ell^\star :=\Big(\frac{c_\ell b_\bd}{d_\ell}\Big) \int_{-1}^1 \bp(t) G_\ell^\beta(t)\bw_\beta(t) dt,
\]
where $G_\ell^{\beta}$ denotes the {\it Gegenbauer} polynomial of degree $\ell$ defined for a parameter $\beta=({\bd-2})/{2}$

\[
\bw_\beta(x):=(1-x^2)^{\beta-\frac 1 2}\,,\quad c_\ell:=\frac{2\ell+\bd-2}{\bd-2}\quad\mathrm{and}\quad b_\bd:=\frac{\Gamma(\frac\bd2)}{\Gamma(\frac 12)\Gamma(\frac\bd2 -\frac 12)}\,,
\]
with $\Gamma$ the Gamma function. We recall that the Gegenbauer polynomials are orthogonal polynomials on the interval $[-1,1]$ with respect to the weight function $\bw_\beta$. Besides, one can recover $\bp\in\bL^2([-1,1],\bw_\beta)$ thanks to the identity
\beq
\label{eq:Coefficients_Polynomial}
\bp=\sum_{\ell\geq0}\Big[\sqrt{d_\ell}\bp_\ell^\star \Big]\,
\Big[\underbrace{G_\ell^\beta/\|G_\ell^\beta\|_{\bL^2([-1,1],\bw_\beta)}}_{Z_\ell}\Big]
=\sum_{\ell\geq0}\bp_\ell^\star c_\ell G_\ell^\beta\,.
\eeq

\begin{remark}
Note that $\bp^\star_\ell$ is the eigenvalue of the operator $\bbT_W$  associated to the eigenspace $\mathcal H_\ell$, $(\sqrt{d_\ell}\bp_\ell^\star)_{\ell\geq0}$ are the coordinates of $\bp\in\bL^2([-1,1],\bw_\beta)$ in the orthonormal basis $(Z_\ell)_{\ell\geq0}$, where $Z_\ell:=G_\ell^\beta/\|G_\ell^\beta\|_{\bL^2([-1,1],\bw_\beta)}$. Note that requiring $W\in\bL^2(\mathds  S^{\bd -1}\times\mathds  S^{\bd -1},\bsigma\otimes\bsigma)$ is equivalent to $\bp\in\bL^2([-1,1],\bw_\beta)$.
\end{remark}

\noindent
Let $R\geq0$ and define
\beq
\label{eq:Definition_widetilde_R}
\widetilde R:=\sum_{\ell=0}^Rd_\ell={R+\bd-1\choose R}+{R+\bd-2\choose R-1}\,,
\eeq
where the last equality is obtained with the telescoping sum using (\ref{dl}). Furthermore, we get that 
\beq \nonumber
\widetilde R
\leq \frac{2(R+\bd -1)^{\bd -1}}{(\bd -1)!}
=\mathcal{O}_{\mathds P}(R^{\bd-1}),
\eeq
and this quantity is the dimension of Spherical Harmonics of degree less than $R$.

\subsection{A Glimpse into Weighted Sobolev Spaces}
\label{sec:WSobolevSphere}
Some of our result concern ‘‘{\it smooth graphons}'' for which a regularity assumption is required.
Following~\cite{Nicaise}, we can define our approximation space defining the {\it Weighted Sobolev space} with the eigenvalues of the Laplacian on the Sphere. More precisely, let $s>0$ a regularity parameter and $f \in\bL^2((-1,1),\bw_\beta) $ such that 
$f =\sum_{\ell \geq 0} f^\star _\ell c_\ell G_\ell^\beta $ in $\bL^2$,
 we define
$$\| f \|^*_{Z^s_{\bw_\beta}((-1,1))}=  \left [\sum_{\ell=0}^{\infty} d_\ell  | f_\ell^\star |^2 (1+(\ell(\ell + 2\beta))^{s}) \right]^{\frac 1 2 } $$
and
$$ Z^s_{\bw_\beta}((-1,1))= \big \{ f \in\bL^2((-1,1),\bw_\beta) \ :\quad  \| f \|^*_{Z^s_{\bw_\beta}((-1,1))} < \infty \big \}.$$
Then, if $\bp$ belongs  to the Weighted Sobolev $ Z^s_{\bw_\beta}((-1,1))$ with smoothness $s>0$, 
it holds
\beq
\label{eq:l2decreasingSobolov}
\sum_{\ell >R} d_\ell (\bp_\ell^\star)^2 = \sum_{\ell>R} d_\ell (\bp_\ell^\star)^2 \frac{1+(\ell(\ell + 2 \beta ))^s}{1+(\ell(\ell + 2 \beta ))^s} \leq C(\bp,s,\bd) R^{-2s }\,,
\eeq
where $C(\bp,s,\bd)>0$ is a constant that may depend on $\bp$, $s$ or $\bd$.

\subsection{Spectrum Consistency of the Matrix of Probabilities}
Under this framework, \corref{Canonical_Kernels} can be written as follows.
\begin{proposition}
\label{thm:KG_Sphere}
There exists a universal constant $C>0$ such that for all $\alpha\in(0,1/3)$ and for all $n^3\geq\widetilde R\log(2\widetilde R/\alpha)$, it holds
\begin{align*}
%
\delta_2(\lambda({\bT}_{n}),\lambda^\star)\leq  
2\Big[\sum_{\ell>R}{d_\ell}(\bp_\ell^\star)^2 \Big]^{\frac12}+
C\sqrt{{\widetilde R}(1+{\log(\widetilde R/\alpha)})/n}
\end{align*}
with probability at least $1-3\alpha$. Moreover, if $\bp$ belongs to the Weighted Sobolev space $ Z^s_{\bw_\beta}((-1,1))$, then for~$n$ large enough
$$\E[\delta_2^2(\lambda({\bT}_{n}),\lambda^\star)]\leq  
C'\left[\frac{n}{\log n }\right]^{-\frac{2 s }{2s +(\bd-1)}} $$
where $C'$ only depends on $s$, $\bd$ and $\| \bp \|^*_{Z^s_{\bw_\beta}((-1,1))}$.
\end{proposition}
A proof can be found in Appendix~\ref{proof:KG_Sphere}. These theoretical results show that the eigenvalues of~${\bT}_{n}$ converge towards  the unknown spectrum $\lambda^\star$. 
\subsection{Nonparametric Estimation of the Kernel Spectrum}
\label{sec:NEKS}

Let us now define our estimation procedure. Recall that we observe a graph and then its
${n\times n}$ adjacency matrix $\bA$, where $\bA_{ij}$ are independent Bernoulli random variables. Our model is that $$\P\{\bA_{ij}=1\}=(\bTheta_0)_{ij}=W(X_i,X_j)=\bp(\langle X_i,X_j\rangle),\qquad 1\le i<j\le n,$$
where $X_1,\ldots,X_n$ are i.i.d. uniform variables on $\mathds S^{\bd-1}$. 
Our aim is to recover the envelope function $\bp$ using only observations $\bA$, the variables $X_i$ being unobserved. The idea is to estimate the coefficients $\bp_\ell^\star$ of $\bp$ in the Gegenbauer polynomial basis, using that 
\[
\lambda^\star:=\big\{\bp_0^\star,\bp_1^\star,\ldots, \bp_1^\star, \ldots,\bp_\ell^\star,\ldots, \bp_\ell^\star,\ldots\big\}
\] is close to $\lambda(\bT_n)$ and this latter is close to the spectrum 
\[
\lambda:=\lambda({\widehat{\bT}_{n}})
\] of our observable $\widehat{\bT}_{n}=(1/n)\bA$.
 Let us fix $R\geq 0$ some 
resolution level, and denote
$$\lambda^{\star R}:=\big(\underbrace{\bp_0^\star}_{d_0},\underbrace{\bp_1^\star,\ldots, \bp_1^\star}_{d_1}, \ldots,\underbrace{\bp_R^\star,\ldots, \bp_R^\star}_{d_R}\big)$$ the first coefficients of $\bp$, {\it i.e.,} the first eigenvalues of $\bbT_W$\textemdash not necessarily the largest. In view of~\eqref{eq:Decomposition_Kernel_Sphere} and defining $\widetilde R$ as in~\eqref{eq:Definition_widetilde_R}, we understand that the $\widetilde R$ first eigenvalues of $\bbT_W$ belong to the convex set

\beq
\label{eq:MR}
\mathcal M_R:=
\Big\{
(\underbrace{u^\star_0}_{d_0},\underbrace{u^\star_1,\ldots, u^\star_1}_{d_1}, \ldots,\underbrace{u^\star_R,\ldots, u^\star_R}_{d_R})\in\bbR^{\widetilde R}
\Big\}.
\eeq

\begin{remark}
One can consider the convex set $\mathcal M_R^{[0,1]}$ of admissible coefficients 
\[
(\underbrace{u^\star_0}_{d_0},\underbrace{u^\star_1,\ldots, u^\star_1}_{d_1}, \ldots,\underbrace{u^\star_R,\ldots, u^\star_R}_{d_R})
\] 
corresponding to a function between $0$ and $1$, namely
\begin{align*}
\mathcal M_R^{[0,1]}:=\Big\{
(u^\star_0,&\, u^\star_1,\ldots, u^\star_1, \ldots,u^\star_R,\ldots, u^\star_R)\in\bbR^{\widetilde R}\text{ s.t.}\\
&\text{there exists an extension } (u^\star_\ell)_{\ell>R} \text{ s.t.}\\
&\text{for a.e. }t\in[-1,1],\quad
0\leq \sum_{\ell=0}^\infty u_\ell^\star c_\ell G_\ell^\beta(t)\leq 1
\Big\}.
\end{align*}
Then, note that $\lambda^{\star R}\in\mathcal M_R^{[0,1]}$ and that for all $x\in\mathcal M_R$
\[
\delta_2(\mathcal P_{\mathcal M_R^{[0,1]}}(x),\lambda^{\star R})\leq\delta_2(x,\lambda^{\star R})
\]
where $\mathcal P_{\mathcal M_R^{[0,1]}}$ denotes the ${\mathbf L}^2$-projection onto $\mathcal M_R^{[0,1]}$. It follows that all the results presented applies if we substitute $\mathcal M_R$ by $\mathcal M_R^{[0,1]}$. But, since we do not use the fact that the coefficients $(u_0^\star,u_1^\star,\ldots, u_1^\star, \ldots,u_R^\star,\ldots, u_R^\star)$ correspond to a function between $0$ and~$1$ in our proofs and our numerical study, we choose to alleviate presentation using $\mathcal M_R$ instead of $\mathcal M_R^{[0,1]}$.
\end{remark}

We assume that $n\geq \widetilde R$ and we denote $\mathfrak S_n$  the set of all permutation of~$[n]$.
    We define the estimator 
$\widehat\lambda^R$ as the closest sequence to $\lambda$ which belongs to the set of ‘‘{\it admissible}'' spectra $\mathcal M_R$ as follows:
\beq
\label{eq:Min_Admissible}
\widehat\lambda^R
\in\argmin_{u\in\mathcal M_R}\min_{\sigma\in\mathfrak S_n}
\Big\{
\sum_{k=1}^{\widetilde R}(u_k-\lambda_{\sigma(k)})^2+\sum_{k=\widetilde R+1}^n\lambda_{\sigma(k)}^2
\Big\}\,.
\eeq
where we recall that $\lambda$ denotes the spectrum of $\widehat{\bT}_n$.
%
We denote $\widehat\bp_\ell^R$ the stage values of $\widehat\lambda^R$, such that 
$$
\widehat\lambda^R=(\widehat\lambda_1^R,\ldots,\widehat\lambda_{\widetilde R}^R)=(\widehat\bp_0^R,\widehat\bp_1^R,\ldots, \widehat\bp_1^R, \ldots,\widehat\bp_R^R,\ldots, \widehat\bp_R^R).
$$ 
One can check that 
\[\displaystyle\widehat\bp_\ell^R=\frac1{d_\ell}\sum_{k=\widetilde{\ell-1}}^{\widetilde \ell}\lambda_{\sigma(k)}
\]
where $\sigma$ (that depends on $R$) is a permutation achieving the minimum in~\eqref{eq:Min_Admissible} and we use the notation~\eqref{eq:Definition_widetilde_R} with the convention $\widetilde{-1}=1$. Furthermore, the true complexity of this estimator is not $n!$ which matches the complexity of~$\mathfrak S_n$. The true computation complexity of our estimator is at most $(R+2)!$ as shown by the next theorem.

\begin{theorem}[Computational Complexity]
\label{thm:TrueComplexity}
Let $R\geq0$ such that $\widetilde R\leq n$. For any sequence of real numbers~$(\lambda_k)_{k=1}^n$ such that $\lambda_1\geq\lambda_2\geq\ldots\geq\lambda_n$ it holds that
\begin{align*}
\exists \mathfrak H_R\subseteq\mathfrak S_n\ \mathrm{s.t.}\ \forall {u\in\mathcal M_R},\ 
\min_{\sigma\in\mathfrak S_n}&
\Big\{
\sum_{k=1}^{\widetilde R}(u_k-\lambda_{\sigma(k)})^2+\sum_{k=\widetilde R+1}^n\lambda_{\sigma(k)}^2
\Big\}\\
&
=
\min_{\sigma\in\mathfrak H_R}
\Big\{
\sum_{k=1}^{\widetilde R}(u_k-\lambda_{\sigma(k)})^2+\sum_{k=\widetilde R+1}^n\lambda_{\sigma(k)}^2
\Big\}
\end{align*}
where the set $\mathfrak H_R$ depends only on $R$ and has size at most $(R+2)!$.
\end{theorem}
A proof can be found in Appendix~\ref{proof:TrueComplexity}. This proof is constructive and it gives the expression of $\mathfrak H_R$. 

\begin{remark}
Remark that the hypothesis $\lambda_1\geq\lambda_2\geq\ldots\geq\lambda_n$ is not necessary and can be removed. Indeed, if $\tau\in\mathfrak S_n$ a permutation such that $\lambda_{\tau(1)}\geq\lambda_{\tau(2)}\geq\ldots\geq\lambda_{\tau(n)}$ then it holds that 
\begin{align*}
\exists \mathfrak H_R\subseteq\mathfrak S_n\ \mathrm{s.t.}\ \forall {u\in\mathcal M_R},\ 
\min_{\sigma\in\mathfrak S_n}&
\Big\{
\sum_{k=1}^{\widetilde R}(u_k-\lambda_{\sigma(k)})^2+\sum_{k=\widetilde R+1}^n\lambda_{\sigma(k)}^2
\Big\}\\
&=
\min_{\sigma\in\mathfrak H_R}
\Big\{
\sum_{k=1}^{\widetilde R}(u_k-\lambda_{\sigma\circ\tau(k)})^2+\sum_{k=\widetilde R+1}^n\lambda_{\sigma\circ\tau(k)}^2
\Big\}
\end{align*}
where the set $\mathfrak H_R$ depends only on $R$ and has size at most $(R+2)!$.
\end{remark}

\begin{remark}
Interestingly the computational complexity of our spectral estimator depends cubicly on the sample size $n$ which is important when observing large networks. The presented algorithm (see Section~\ref{simus}) has exponential complexity in the dimension~$R$ of the model. Hence, it is relevant only for low degree $R$ kernels {\it stricto sensu}. However, it can be accelerated:
\begin{itemize}
\item If the experimenter knows that the eigenvalues are monotone (when sorting the eigenvalues so that the corresponding eigen-spaces have increasing dimensions) then the complexity is linear in $R$;
\item If not, different clustering strategies might be used in place of {\bf Steps 2-5} in Section~\ref{simus}. One may consult Section~\ref{speed} for further details, where we argue that hierarchical agglomerative clustering (HAC) with single linkage, of time complexity $\mathcal O(n^3)$, can be invoked here.
\end{itemize}
\end{remark}

Using Proposition~\ref{prop:BvH} and Theorem~\ref{thm:KG_Sphere} we can prove that $\widehat\lambda^R$ is a relevant estimator of the true first eigenvalues~$\lambda^{\star R}$ as shown in the next theorem.
\begin{theorem}
\label{thm:New_Main_Sphere}
There exists a universal constant $\kappa_0>0$ such that the following holds. For all $\alpha\in(0,1)$, if $n^3\geq(2\widetilde R)^3\vee\widetilde R\log(2\widetilde R/\alpha)$, with probability greater that $1-3\alpha$, it holds
 \begin{align*}
\delta_2(\widehat\lambda^R,\lambda^{\star R})\leq  4
\delta_2(\lambda^{\star R}, \lambda^{\star }) +
\kappa_0\sqrt{{\widetilde R}\left(1+\log\left({\widetilde R}/{\alpha}\right)\right)/n}.
\end{align*}
Moreover, there exists a universal constant $\kappa_1>0$ such that, if $n\geq 2\widetilde R$ then 
\begin{align*}
\E[\delta_2^2(\widehat \lambda^{R}, \lambda^{\star R})]&\leq \kappa_1\Big\{\delta_2^2(\lambda^{\star R}, \lambda^{\star }) + \frac{\widetilde R\log n}n\Big\}.
\end{align*}
\end{theorem}

\noindent
A proof can be found in Appendix~\ref{proof:New_Main_Sphere}.

\begin{remark}
\label{rem:Thm6}
Possibly considering larger numerical constants $\kappa_0,\kappa_1>0$, in the relatively sparse model~\eqref{eq:relativelysparse}, the previous result reads as follows: if $n^3\geq (2\widetilde R)^3\vee\widetilde R\log(2\widetilde R/\alpha)$ then
$$\delta_2(\widehat\lambda^R,\zeta_n\,\lambda^{\star R})\leq4\,\zeta_n\,\,\delta_2(\lambda^{\star R}, \lambda^{\star }) 
+
\kappa_0\sqrt{\frac{\zeta_n\widetilde R}n}
\Bigg[1+\sqrt{\zeta_n(1+{\log(\widetilde R/\alpha)})}+\sqrt{\frac{\log(n/\alpha)}{n\zeta_n}}\Bigg].
$$
with probability at least $1-3\alpha$. If $n\geq 2\widetilde R$ then 
\begin{align*}
\E[\delta_2^2(\widehat \lambda^{R}, \zeta_n\,\lambda^{\star R})]&\leq \kappa_1\Big\{\zeta_n^2\delta_2^2(\lambda^{\star R}, \lambda^{\star }) + {\zeta_n\frac{\widetilde R\,(1+\zeta_n\log n)}n}\Big\}.
\end{align*}
In the aforementioned setting, we have denoted the eigenvalues of $\bbT_W$ by $\lambda^\star$ (as before), and their estimation (based on the adjacency matrix) by $\widehat\lambda$. These latter are scaled by a factor $\mathcal O_{\bbP}(\zeta_n)$ in the relatively sparse model~\eqref{eq:relativelysparse}.
\end{remark}

\pagebreak[3]


To go further
we need to analyze the behavior of the bias term $\delta_2(\lambda^{\star R}, \lambda^{\star })$ as a function of $R$ under some regularity conditions on the envelope $\bp$.
Indeed we can write $$\delta_2(\lambda^{\star R}, \lambda^{\star})^2=\sum_{k>\widetilde R}|\lambda_k^{\star}|^2=\sum_{\ell>R}{d_\ell}(\bp_\ell^\star)^2
.$$

Assume that $\bp$ belongs to the weighted Sobolev space $Z^s_{\bw_\beta}((-1,1))$ of regularity $s>0$ defined in Section~\ref{sec:WSobolevSphere}. Thus, since $\widetilde R= \mathcal{O}(R^{\bd-1})$, 
using~\eqref{eq:l2decreasingSobolov} and setting $R_{\mathrm{opt}}=\lfloor ({n}/{\log n })^{\frac{1}{2s+\bd-1}} \rfloor$, we get
\begin{align*}
\E\left[\delta_2^2(\widehat\lambda^{R_{\mathrm{opt}}},\lambda^\star)\right]&\leq
2\delta_2^2(\lambda^{\star R_{\mathrm{opt}}},\lambda^\star)+2\E\delta_2^2(\widehat\lambda^{R_{\mathrm{opt}}},\lambda^{\star R_{\mathrm{opt}}})\\
&\lesssim R_{\mathrm{opt}}^{-2s }+\frac{\widetilde R_{\mathrm{opt}}\log n}n
\lesssim  \left[\frac{n}{\log n }\right]^{-\frac{2 s }{2s +(\bd-1)}}\,.
\end{align*}
Thus we recover a classical nonparametric rate of convergence for  estimating a function with smoothness $s$ in a space of dimension $\bd -1$, see  \cite{hasminskii1990} for instance. This is also the rate towards the probability matrix obtained by \cite{xu2017rates}. However, assuring that this is the optimal rate of convergence is beyond the scope of the paper. Note that the present setting to estimate a graphon nonparametrically differs from the  regression framework. First, the $\delta_2$ loss is defined up to the action of the permutation group. Moreover, despite the
 number~$n^2$ of  observations, the problem suffers from the presence of latent variables. Indeed the design points~$X_i$'s are unobserved. This all contributes to a non standard estimation problem.
 
 \pagebreak[3]
 
\begin{remark}
In the relatively sparse model~\eqref{eq:relativelysparse}, the same calculation leads to 
\[
R_{\mathrm{opt}}=\Big(\frac{n\zeta_n}{1+\zeta_n\,\log n }\Big)^{\frac{1}{2s+\bd-1}} 
\]
and
\begin{align*}
\E\left[\delta_2^2(\widehat\lambda^{R_{\mathrm{opt}}},\zeta_n\,\lambda^\star)\right]&
\lesssim  \zeta_n^2\left[\frac{n\zeta_n}{1+\zeta_n\log n }\right]^{-\frac{2 s }{2s +(\bd-1)}}\,.
\end{align*}
\end{remark}

We also face a classical issue of nonparametric statistics: {\it how to choose $R$, given that the best theoretical choice $R_{\mathrm{opt}}$ depends on the unknown smoothness~$s$?} This is the point of the next section.

\subsection{Adaptation to the Smoothness of $\bp$}
\label{sec:Adaptation}

Let us define
$\R=\{1, 2,\dots, R_{\max}\}$ the possible values for $R$, with $2\widetilde R_{\max}\leq n$.
Following the Goldenshluger-Lepski method \citep{GoldLepski2013}, set
\begin{align}\label{defGL1}
B(R):=\max_{R'\in\R} \Big\{\delta_2(\widehat \lambda^{R'},\widehat \lambda^{R'\wedge R})-\kappa\sqrt{\frac{\widetilde R'\log n}n}\Big\},
\end{align}
where $R\wedge R'=\min(R,R')$ and $\kappa>0$ is a constant to be specified later. This function can be seen as an estimation of the (unknown) bias $\delta_2(\lambda^{\star R}, \lambda^{\star })$. Then we define our final resolution level  $\widehat R$ as a minimizer of an approximation of the risk as 
\begin{align}\label{defGL2}
\widehat R\in\argmin_{R\in\R}\Big\{B(R)+\kappa\sqrt{\frac{\widetilde R\log n}n}\Big\}.
\end{align}
The estimator of $\lambda^\star$ is then $\widehat \lambda^{\widehat R}$, which depends on the choice of constant $\kappa$ in~\eqref{defGL1} and~\eqref{defGL2}. 
The following results show that this estimator is as good as the best one of the collection $( \widehat \lambda^{ R})_{R\in \R}$,
up to a constant $C$, provided that $\kappa$ is large enough.

\begin{theorem}\label{thm:adap} 
Let $\widehat \lambda^{\widehat R}$ the estimator defined by~\eqref{eq:Min_Admissible},~\eqref{defGL1} and~\eqref{defGL2}. 
There exist numerical constants $C>0$ and $\kappa_0>0$ $($as in Theorem~\ref{thm:New_Main_Sphere}$)$ such that ,
if $\kappa\geq \kappa_0\sqrt{11}$, 
with probability $1-3n^{-8}$
$$\delta_2(\widehat \lambda^{\widehat R}, \lambda^{\star})\leq
C\min_{R\in\R}\Big\{\delta_2(\lambda^{\star R}, \lambda^{\star })+\kappa\sqrt{\frac{\widetilde R\log n}n}\Big\}.$$
Moreover, for $\kappa\geq \kappa_0\sqrt5$, there exists a numerical  constant $C'>0$ such that 
$$\E[\delta_2^2(\widehat \lambda^{\widehat R}, \lambda^{\star})]\leq
C'\min_{R\in\R}\Big\{\delta_2^2(\lambda^{\star R}, \lambda^{\star })+\kappa^2\frac{\widetilde R\log n}n\Big\}.$$
\end{theorem}
A proof can be found in Appendix~\ref{proof:adap}. Thus we choose $\kappa\geq \kappa_0\sqrt5$  in~\eqref{defGL1} and~\eqref{defGL2}, the practical choice of the tuning constant $\kappa$ will be tackled in Section~\ref{sec:numeric}. Note also that the interesting choice of $\mathcal R$ is such that $R_{\mathrm opt}\in\mathcal R$ which is the case for $cn^{\frac{1}{2s+\mathbf d-1}}\leq R_{\max}$ where $c>0$ is a constant. A more simple choice of~$R_{\max}$ may be $cn\leq 2\widetilde R_{\max}\leq n$ where $0<c<1$ is a constant. In these cases, we get the following rate of convergence.

\begin{corollary}\label{vitesse} Assume that $\bp$ belongs to the Weighted Sobolev space $ Z^s_{\bw_\beta}((-1,1))$. Then
there exists a constant $C>0$ depending only on $\| \bp \|^*_{Z^s_{\bw_\beta}((-1,1))}$, $s$ and $\bd$ such that 
$$\E[\delta_2^2(\widehat \lambda^{\widehat R}, \lambda^{\star})]\leq C\left[\frac{n}{\log n }\right]^{-\frac{2 s }{2s +(\bd-1)}}.$$
\end{corollary}
This means that the algorithm automatically adapts $\widehat R$ to the unknown smoothness $s$ of $\bp$: it chooses a small resolution level for smooth functions
and a greater $\widehat R$ for irregular functions, that provides the best result in each case. 

The final step is to define the following estimator of envelope $\bp$ by
\begin{equation}\label{pestimator}
\forall t\in[-1,1],\quad\widehat\bp^{\widehat R}(t):=\sum_{\ell=0}^{\widehat R}\widehat\bp_\ell^{\widehat R} c_\ell G_\ell^\beta(t)\,.
\end{equation}

\subsection{Estimating the envelope function}

Inferring from the estimation of $\lambda^\star$ to the estimation of $\bp$, we face an identifiability problem. Indeed, consider for instance the case $\bd=3$, which implies $\beta=1/2$,
$d_\ell=2\ell+1$, $c_\ell=2\ell+1$. For $\mu>0$, let
\begin{align*}
\bp_a&=\frac12 c_0 G_0^\beta+\mu c_1 G_1^\beta+0 \times c_2 G_2^\beta+0\times  c_3 G_3^\beta+\mu c_4 G_4^\beta\,,\\
\bp_b&=\frac12 c_0 G_0^\beta+0\times c_1 G_1^\beta+\mu c_2 G_2^\beta+\mu c_3 G_3^\beta+0\times c_4 G_4^\beta
\end{align*}
Then the associated spectrum are
\begin{align*}
\lambda_a^\star&=(1/2,\; \underbrace{\mu,\mu, \mu}_3,\; \underbrace{0,0,0,0,0}_5,\; \underbrace{0,0,0,0,0,0,0}_7, \; \underbrace{\mu,\mu, \mu,\mu,\mu, \mu,\mu,\mu, \mu}_9)\\
\lambda_b^\star&=(1/2,\; \underbrace{0,0,0}_3,\; \underbrace{\mu,\mu, \mu,\mu,\mu}_5,\; \underbrace{\mu,\mu, \mu,\mu,\mu, \mu,\mu}_7,\;\underbrace{0,0,0, 0,0,0,0,0,0}_9)
\end{align*}
which are indistinguishable in $\delta_2$ metric, although $\|\bp_a-\bp_b\|_2=\mu\sqrt{24}$. Furthermore, note that, for $\mu\leq 1/24$, these functions have values in $[0,1]$.

\begin{remark}
A natural question is then: {\rm Can we recover the right eigenvalues labels from the empirical eigenvectors?}\newline
 Under stronger requirements (RKHS-type assumptions), convergence of the eigenvectors of~$\bA/n$ towards the eigenfunctions of the integral operator $\bbT_W$ may be proved as in~\cite{tang2013universally}. Essentially, it is possible to prove that the orthogonal projections $\Pi_\ell$ onto eigenspaces of $\bA/n$  are closed in operator norm to the $n\times n$ matrix with entries $\sum_{j=1}^{d_\ell}Y_{\ell j}(X_i)Y_{\ell j}(X_j)$ given by the Zonal Harmonics. Unfortunately, this statistics depends on the latent points and suffers from the ‘‘agnostic'' error as explained in~\cite{klopp2017oracle}. While possible theoretically, it seems difficult in practice to use the information of the observed eigenvectors to uncover the right labels of the eigenvalues.
\end{remark}

\pagebreak[3]

Nevertheless we can state a result in the case of  a finite spectrum of distinct eigenvalues. 

\begin{proposition} 
\label{cor:variancepr} 
 Assume that the envelope function $\bp$ is polynomial of degree~$D$, {\it i.e.,} $\bp_\ell^\star=0$ for any $\ell>D$ and $\bp_D^\star\neq0$. Assume also that all nonzeros $\bp_\ell^\star$ for $\ell\in\{0,\dots, D\}$ are distinct.  If $R\geq D$ and  $n$ is large enough then 
$$\|\widehat \bp^{R}-\bp\|_2^2\leq 11\kappa_0^2\frac{\,{\widetilde R}\log n}{n}\,, $$
with probability greater that $1-3n^{-8}$ where $\kappa_0>0$ is the constant defined in Theorem~\ref{thm:New_Main_Sphere}. Furthermore, it holds
$$\E[\|\widehat \bp^{R}-\bp\|_2^2]\leq (18+4\kappa_0^2) \frac{
\,\widetilde R\log n}n\,,$$ 
for $n$ large enough.
\end{proposition}

A proof can be found in Appendix~\ref{proof:variancepr}.  
We actually  prove that these upper bounds are true as soon as
\[
n\geq 2\widetilde R\quad\mathrm{and}
\quad2\kappa_0\sqrt{11{\widetilde R}\log(n)/n}<\min_{0\leq i\neq j\leq  D;\ \bp_i^\star\neq0}|\bp_i^\star-\bp_j^\star|\wedge|\bp_i^\star|
\]
Note that we uncover (up to a log factor) the parametric rate of estimation. 

{
\begin{remark}
Under the same assumptions than Proposition \ref{cor:variancepr}, in the relatively sparse model~\eqref{eq:relativelysparse},  the previous result reads as follows: 
 If $R\geq D$ and  $n$ is large enough then 
$$\|\widehat \bp^{R}-\bp\|_2^2\leq \kappa_0^2 \frac{\zeta_n \widetilde R}{n} \left(1 +11 \zeta_n \log n  \right)$$
with probability greater that $1-3n^{-8}$, for some constant $\kappa_0$ possibly large. Furthermore, we have 
$$
\E[\|\widehat \bp^{R}-\bp\|_2^2]\leq \kappa_1 \frac{\zeta_n \widetilde R}{n}\left (1+ 4\zeta_n \log n \right),
$$
for $n$ large enough and for some constant $\kappa_1$ possibly large.
\end{remark}
}

Let us now state what the adaptive procedure defined by~\eqref{defGL1} and~\eqref{defGL2} can do in this polynomial case. 

\begin{corollary}\label{u} 
 Assume that the envelope function $\bp$ is polynomial of degree~$D$, {\it i.e.,} $\bp_\ell^\star=0$ for any $\ell>D$ and $\bp_D^\star\neq0$. Assume also that all nonzeros $\bp_\ell^\star$ for $\ell\in\{0,\dots, D\}$ are distinct. If ${R_{\max}}\geq D$, there exists a numerical constant $C$ such that,
if $n$ large enough, then $\widehat R\geq D$ a.s. 
and $$\E[\|\widehat \bp^{\widehat R}-\bp\|_2^2]\leq C\widetilde{D}\left(\frac{\log n }{n}\right).$$
\end{corollary}

More precisely, the result is true as soon as $\sqrt{\widetilde D \log(n)/n}$ is smaller than a constant times $\min_{0\leq i\neq j\leq  D;\ \bp_i^\star\neq0}|\bp_i^\star-\bp_j^\star|\wedge|\bp_i^\star|$.
A proof can be found in Appendix~\ref{proof:adaptpolynomial}.  
Here again, the parametric rate of estimation is attained by the adaptive procedure. 


\section{Extensions to Compact Symmetric Spaces}
\label{sec:LatentSpaces}

The aim of this section is to extend the previous result on spheres to numerous spaces such as compact Lie groups and compact symmetric spaces. A useful reference might be the books~\cite{wolf2007harmonic,bump2013lie} or the nice survey written in \cite[Chapter 3]{plmHDR} (see also \cite{meliot2014cut} for a presentation of compact symmetric spaces) which has been useful to polish this section.

\subsection{Harmonic Analysis on Compact Symmetric Spaces}

In this section, we consider that $(\bS,\gamma)$ is a compact Lie group with an invariant Riemannian metric~$\gamma$, or more generally a compact symmetric space. The definitions will be given below when describing \textit{Cartan's Classification} and, to be specific, this section focuses on (semi)simple connected compact Lie groups (sscc in short) and simple simply connected compact symmetric spaces (ssccss in short). These structures encompass spheres, projective spaces, Grassmannians, and orthogonal or unitary groups; and one can handle explicit eigenvectors computations in this framework.

Consider again that the graphon $W(g,h)$ depends only on (the cosine of) the distance $\gamma(g,h)$ (normalized so that the range of $\gamma$ equals $[0,\pi]$) between points $g,h\in\bS$ such that
\beq
\notag
W(g,h)=\bp(\cos \gamma(g,h))=\bp(\cos \gamma(gh^{-1},e_{\bS}))=:p(gh^{-1})
\eeq
where $e_{\bS}$ denotes the identity element and $p(g)=\bp(\cos\gamma(g,e_{\bS}))$. Also we assume that $0\leq W\leq 1$ since~$W$ defines a probability matrix. In particular, $W$ is square-integrable on the compact $\bS\times \bS$. Observe that estimating $W$ reduces to estimate~$\bp$ that reduces to estimate~$p$ and \textit{vice versa}. By definition of the distance, note that
\begin{itemize}
\item When $\bS$ is a sscc Lie group, the function $p$ is invariant by conjugation, namely $p(hgh^{-1})=p(g)$ for any latent points $g,h\in\bS$. We denote by $\mathbf L^2(\bS)^{\bS}$ the space of square-integrable functions $p$ on~$\bS$ that are invariant by conjugation.
\item When $\bS=G/K$ is a ssccss, the function $p$ is bi-$K$-invariant, namely $p(k_1gk_2)=p(g)$ for any $k_1,k_2\in K$ and $g\in G$. We may denote by $\mathbf L^2(K\setminus\! G/K)$ the space of square-integrable functions on $G$ that are bi-$K$-invariants.
\end{itemize}
In particular, Peter–Weyl’s decomposition (presented below) gives an $\mathbf L^2$-decomposition of $p$ in these settings. The measure on $\bS$ is the Haar measure (normalized to be a probability measure), denoted~$\mathrm dg$, standardly defined for any compact topological group $\bS$.  The harmonic analysis on $\bS$ is based on the Fourier transform of the space $\mathbf L^2(\bS,\mathrm dg)$ of square integrable (complex valued) functions on $\bS$. This space $\mathbf L^2(\bS,\mathrm dg)$ is a Hilbert space for the scalar product
\[
\langle f_1,f_2\rangle=\int_{\bS}\overline{f_1}(g)f_2(g)\mathrm dg\,.
\]
We define also the convolution product
\[
(f_1*f_2)(g)=\int_{\bS}f_1(gh^{-1})f_2(h)\mathrm dh\,.
\]
Now, recall that $W$ defines a symmetric Hilbert-Schmidt operator $\bbT_W$ on $\mathbf L^2(\bS,\mathrm dg)$ and the spectral theorem~\eqref{eq:Spectral_Decomposition} gives 
\beq
\notag
W(g,h)=\sum_{k\geq1}\lambda_k^\star {\phi_k}(g){\phi_k}(h)\,,
\eeq
for an $\bL^2(\bS,\mathrm dg)$-orthonormal basis $(\phi_i)_{i\geq1}$. Remark also that
\begin{align*}
(\bbT_W(f))(g_1)&=\int_{\bS}W(g_1,g_2)f(g_2)\mathrm dg_2=\int_{\bS}W(g_1g_2^{-1},e_{\bS})f(g_2)\mathrm dg_2
\\&=\int_{\bS}p(g_1h^{-1})f(h)\mathrm dh=(p*f)(g_1)
\end{align*}
for all $f\in\bL^2(\bS,\mathrm dg)$. We deduce that $\bbT_W$ is the convolution on the left by $p$. We continue with a short reminder on harmonic analysis on compact groups and compact quotients.

\begin{description}
\item[Representation of Compact Groups and Irreducible Characters] The first ingredient is representations of any compact group $\bS$. It is defined by a finite dimensional complex vector space $V$ and by a continuous morphism of groups $\rho:\bS\to \mathrm{GL}(V)$ where $\mathrm{GL}(V)$ denotes the group of isomorphisms of $V$. A linear representation~$(V,\rho)$ is \textit{irreducible} if one cannot find a subspace $W$ such that $0\subsetneq W\subsetneq V$ and that is $\bS$-stable, \textit{i.e.}, for all~$w\in W$ and all $g\in\bS$, one has $\rho(g)(w)\in W$. If $V$ is a linear representation then one can always split it into irreducible components
\[
V=\bigoplus_{\mathbf{r}\in\widehat{\bS}}m_{\mathbf{r}} V^{\mathbf{r}}
\]
where $\widehat{\bS}$ is the countable set of isomorphism classes of irreducible representations ${\mathbf{r}}=(\rho^{\mathbf{r}},V^{\mathbf{r}})$ of~$\bS$ and $m_{\mathbf{r}}\geq1$. Furthermore, we denote by 
\[
\mathrm{ch}^{{\mathbf{r}}}(g)=\mathrm{tr}(\rho^{\mathbf{r}}(g))\,,
\]
the irreducible characters associated to the irreducible representation ${\mathbf{r}}=(\rho^{\mathbf{r}},V^{\mathbf{r}})$ of~$\bS$ where $\mathrm{tr}$ denotes the trace operator on $\mathrm{End}_{\bbC}(V^{\mathbf{r}})$ the set of (complex) endomorphisms of $V^{\mathbf{r}}$. In particular, since $\rho^{\mathbf{r}}(g)$ is unitary, it holds
\[
\forall g\in\bS,\quad|\mathrm{ch}^{{\mathbf{r}}}(g)|\leq d_{\mathbf{r}}=\mathrm{ch}^{{\mathbf{r}}}(e_{\bS})\,,
\]
where $d_{\mathbf{r}}$ is the dimension of $V^{\mathbf{r}}$. Also, note that
\[
\mathrm{ch}^{{\mathbf{r}}}*\mathrm{ch}^{{\mathbf{s}}}=\frac{\delta_{{\mathbf{r}}{\mathbf{s}}}}{d_{\mathbf{r}}}\mathrm{ch}^{{\mathbf{r}}}\,,
\]
where $\delta_{{\mathbf{r}}{\mathbf{s}}}$ denotes the Kroneker delta.
\item[Peter–Weyl’s Decomposition] 
The Peter–Weyl’s Decomposition shows that $(\mathrm{ch}^{{\mathbf{r}}})_{{\mathbf{r}}\in\widehat{\bS}}$ is an orthonormal basis of  $\mathbf L^2(\bS)^{\bS}$. It follows that 
\[
p=\sum_{{\mathbf{r}}\in\widehat{\bS}}\langle p,\mathrm{ch}^{{\mathbf{r}}}\rangle \mathrm{ch}^{{\mathbf{r}}}
\]
in $\mathbf L^2(\bS)^{\bS}$. Using that $\bbT_W$ is a left convolution operator by $p$, we find that $(\mathrm{ch}^{{\mathbf{r}}})_{\mathbf{r}\in\widehat{\bS}}$ is an eigenfunction basis of  $\bbT_W$ associated to the eigenvalues $(\lambda^\star_{\mathbf{r}})_{\mathbf{r}\in\widehat{\bS}}$ given by
\[
\lambda^\star_{\mathbf{r}}=\frac{\langle p,\mathrm{ch}^{{\mathbf{r}}}\rangle}{d_{\mathbf{r}}}\,,
\]
with multiplicity $d_{\mathbf{r}}^2=\mathrm{dim}(\mathrm{End}_{\bbC}(V^{\mathbf{r}}))$. 
\item[Compact Gelfand Pairs and Zonal Spherical Functions] There is an extension of this decomposition to quotients $\bS=G/K$ of a compact topological group $G$ by a closed subgroup~$K$. The most convenient setting for this extension is the one of compact Gelfand pairs defined as follows. 
\begin{definition}[Gelfand Pair]
We say that $(G,K)$ is a Gelfand pair if for any irreducible representation~$V^{\mathbf{r}}$ of~$G$, the space of $K$-fixed vectors 
\[
V^{{\mathbf{r}},K}=\Big\{v\in V^{\mathbf{r}}\ :\ \forall k\in K\,,\ \rho^{\mathbf{r}}(k)(v)=v\Big\}
\]
has dimension at most one.
\end{definition}
An irreducible representation $V^{\mathbf{r}}$ is called \textit{spherical} if $\mathrm{dim}_{\bbC}(V^{{\mathbf{r}},K})=1$. We denote by $\widehat G^{K}$ the set of spherical representations of the Gelfand pair $(G,K)$. If ${\mathbf{r}}\in\widehat G^{K}$ then we denote by $e^{\mathbf{r}}$ a unit vector vector in $V^{{\mathbf{r}},K}$ which is unique up to a multiplicative complex constant of modulus one. The \textit{zonal spherical functions} 
\[
\mathrm{zon}^{\mathbf{r}}(g)=\sqrt{d_{\mathbf{r}}}\big\langle e^{\mathbf{r}},\rho^{\mathbf{r}}(g)(e^{\mathbf{r}})\big\rangle_{V^{\mathbf{r}}}
\]
where $d_{\mathbf{r}}$ is the dimension of $V^{\mathbf{r}}$. In particular, since $\rho^{\mathbf{r}}(g)$ is unitary and~$e^{\mathbf{r}}$ normalized, it holds
\beq
\label{eq:BorneZon}
\forall g\in G,\quad|\mathrm{zon}^{\mathbf{r}}(g)|\leq \sqrt{d_{\mathbf{r}}}=\mathrm{zon}^{\mathbf{r}}(e_{G})\,,
\eeq
where $e_{G}$ is the identity element of $G$. Also, note that
\[
\mathrm{zon}^{\mathbf{r}}*\mathrm{zon}^{{\mathbf{s}}}=\frac{\delta_{{\mathbf{r}}{\mathbf{s}}}}{\sqrt{d_{\mathbf{r}}}}\mathrm{zon}^{{\mathbf{r}}}\,.
\]

\item[Cartan’s Extension of Peter–Weyl’s Decomposition] In the case of bi-$K$-invariant functions on $G$, an extension of Peter–Weyl’s decomposition theorem shows that $(\mathrm{zon}^{\mathbf{r}})_{{\mathbf{r}}\in\widehat G^{K}}$ is an orthonormal basis of $\mathbf L^2(K\setminus\! G/K)$. It follows that 
\[
p=\sum_{{\mathbf{r}}\in\widehat G^{K}}\langle p,\mathrm{zon}^{{\mathbf{r}}}\rangle \mathrm{zon}^{{\mathbf{r}}}
\]
in $\mathbf L^2(K\setminus\! G/K)$. Using that $\bbT_W$ is a left convolution operator by $p$, we find that $(\mathrm{zon}^{{\mathbf{r}}}))_{\mathbf{r}\in\widehat{G}^K}$ is an eigenfunction of  $\bbT_W$ associated to the eigenvalue~$(\lambda^\star_{\mathbf{r}})_{\mathbf{r}\in\widehat{G}^K}$ given by
\[
\lambda^\star_{\mathbf{r}}=\frac{\langle p,\mathrm{zon}^{{\mathbf{r}}}\rangle}{\sqrt{d_{\mathbf{r}}}}\,.
\]
with multiplicity $d_{\mathbf{r}}=\mathrm{dim}(V^{\mathbf{r}})$. The reader may recognize here the case of the sphere studied in the previous section.
\item[Cartan's Classification of sscc Lie Groups and ssccss] Now, a crucial question is how explicit are these decompositions. We begin with the notion of sscc Lie Groups that is based on Cartan’s criterion for semisimplicity. It implies that a simply connected compact Lie group can always be written as a direct product of simple simply connected compact Lie group (in short sscc Lie group). Here, by simple we mean a Lie group $\bS$ whose Lie algebra is simple, that is nonabelian and without non-trivial ideal. Interestingly, Cartan’s classification of sscc Lie groups shows that any sscc Lie group fall into one of the following infinite families:
\begin{description}
\item[Group type] 
{\ }
\begin{itemize}
\item Special unitary group $\mathrm{SU}(n + 1)$,
\item Odd spin group $\mathrm{Spin}(2n + 1)$,
\item Compact symplectic group $\mathrm{USp}(n)$,
\item Even spin group $\mathrm{Spin}(2n)$,
\end{itemize}
or, it is one of the five exceptional compact Lie groups. 
\end{description}
The sscc Lie groups belong to a larger class of compact Riemannian manifolds called symmetric spaces. Moreover, any simply connected compact symmetric space is isometric to a product of simple simply connected compact symmetric spaces (in short ssccss), which cannot be split further. A classification of all the ssccss $\bS$ has been proposed by Cartan which shows that $\bS$ is either of Group type (see above) or one of the following objects
\begin{description}
%
%
%
\item[non-Group type] In this case, $\bS$ falls into one of the following infinite families:
\begin{itemize}
\item Real Grassmannians $\mathrm{SO}(p + q)/(\mathrm{SO}(p)\times\mathrm{SO}(q))$,
\item Complex Grassmannians $\mathrm{SU}(p + q)/(\mathrm{SU}(p)\times\mathrm{SU}(q))$,
\item Quaternionic Grassmannians $\mathrm{USp}(p + q)/(\mathrm{USp}(p)\times\mathrm{USp}(q))$,
\item Space of real structures on a complex space $\mathrm{SU}(n)/\mathrm{SO}(n)$,
\item Space of quaternionic structures on an even complex space $\mathrm{SU}(2n)/\mathrm{USp}(n)$,
\item Space of complex structures on a quaternionic space $\mathrm{USp}(n)/\mathrm{U}(n)$,
\item Space of complex structures on an even real space $\mathrm{SO}(2n)/\mathrm{SU}(n)$,
\end{itemize}
or, it is one of the twelve exceptional sscc symmetric spaces.
\end{description}
Remark that, for all the ssccss examples, the eigenfunctions of the spectral decomposition of $\bbT_W$ do not depend on $\bbT_W$ and they are irreducible characters in the group case and zonal spherical functions in the non-group case.
\item[Weyl’s Highest Weight theorem and Cartan–Helgason's Extension] 
Given a ssccss, we can make explicit the set $\widehat{\bS}$ in the group case, and the set $\widehat G^K$ in the non-group case thanks to the Weyl’s highest weight theorem and Cartan–Helgason's extension, see \cite[Chapter 3]{plmHDR} for a short and well written introduction. The highest weight theorem is completed by a formula for the irreducible character~$\mathrm{ch}^{\mathbf r}$ of the module $V^{\mathbf r}$ with highest weight $\mathbf r$, see for instance \cite[Chapter~22]{bump2013lie} and Weyl’s integration formula \cite[Chapter~17]{bump2013lie}. 

The same analysis can be lead in the non-group case. The only additional difficulty is the manipulation of zonal spherical functions. This issue will be handled by considering compact symmetric spaces of rank 1 in the following.
\end{description}
Now, we are ready to extend the previous results on the sphere to other latent spaces $\bS$, namely the compact symmetric spaces of rank $1$.
\subsection{Compact Symmetric Spaces of Rank One}
We focus here on the interesting case of compact symmetric spaces of rank one for which the zonal spherical functions can be explicitly computed. Indeed, one has the following classification of the compact symmetric spaces of rank one and of the corresponding spherical representations, see \cite[Chapter 3]{plmHDR} and \cite{volchkov2009harmonic} for instance. We recall that  ‘‘{\it intuitively, a symmetric space is a Riemannian manifold where geodesics are ‘‘nicely'' arranged in a symmetric way around any point}''. More formally, a compact symmetric space of rank one is 
ssccss that is $2$-point homogeneous, namely
\begin{itemize}
\item {\bf [Compact Symmetric Spaces of Rank One]} {\it Given two pairs of points $(x_1,x_2)$ and $(y_1,y_2)$ such that $\gamma(x_1,x_2)=\gamma(y_1,y_2)$, there is an isometry of $\bS$ that maps $x_1$ $($resp. $x_2)$ onto $y_1$ $($resp. $y_2)$.}
\end{itemize}
The compact symmetric spaces of rank one are
\begin{itemize}
\item the real spheres $\bbS^{\mathbf d-1}=\mathrm{SO}(\mathbf d)/\mathrm{SO}(\mathbf d-1)$,
\item the real projective spaces $\bbR\bbP^{\mathbf d-1}=\mathrm{SO}(\mathbf d)/\mathrm{O}(\mathbf d-1)$,
\item the complex projective spaces $\bbC\bbP^{\mathbf d-1}=\mathrm{SU}(\mathbf d)/\mathrm{U}(\mathbf d-1)$,
\item the quaternionic projective spaces $\bbH\bbP^{\mathbf d-1}=\mathrm{USp}(\mathbf d)/(\mathrm{USp}(\mathbf d-1)\times\mathrm{USp}(1))$,
\item or the octonionic projective plane $\bbO\bbP^{2}=F_4/\mathrm{Spin}(9)$.
\end{itemize}
In the case of compact symmetric spaces of rank $1$, one can explicitly described the spherical representations $\widehat G^K$. The dimension $d_\ell:=\mathrm{dim}_\bbC(V^{\ell\omega_0})$ of the $\ell$-th spherical representation $V^{\ell\omega_0}$ are given in Table~\ref{t:CSR1} in the Appendix. One can even describe the zonal spherical functions of these spaces, and thus compute the eigenvalues~$\bp^\star_\ell$ (recall that their multiplicities $d_\ell$ are given by Table~\ref{t:CSR1} in the Appendix). 

For compact symmetric spaces of rank one, on can define 
\begin{itemize}
\item a probability density function $\bw(t)$ on $[-1,1]$ defined as the density of the pushforward measure of the Haar measure by the map $x\mapsto t=\cos(\gamma(x,e))$,
\item the pushforwards  $Z_\ell$ on $[-1,1]$ of the zonal spherical functions, normalized so that they are an orthonormal basis of $\mathbf L^2([-1,1],\bw)$, the space of square-integrable functions with respect to the weight function~$\bw$ on~$[-1,1]$.
\end{itemize}


In Table~\ref{t:CSR1}, one hase the following standard parameterizations of latent space $\bS$:
\begin{itemize}
\item the real sphere~$\bbS^{\mathbf d-1}$ is endowed with the coordinates $x=(x_1,\ldots,x_{\mathbf d})$ such that $|| x||_2=1$ and the ‘‘{\it north pole}'' is given by  $e=(0,\ldots,0,1)$. We denote the ‘‘{\it weight function}'' by $\bw(x)$, it is the density of the push forward measure of the Haar measure by the map $x\mapsto\cos(\gamma(x,e))$ where we recall that $\gamma(x,e)=\arccos x_{\mathbf d}$.
\item the projective space $\bbF\bbP^{\mathbf d-1}$ (where $\bbF=\bbR,\bbC,\bbH$ or $\bbO$) is endowed with projective coordinates $[x_1 : x_2 : \cdots : x_{\mathbf d}]$ with the $x_i$'s in $\bbF$, and the ‘‘north pole'' is given by  $e=[0  : \cdots : 0 : 1]$. We denote the ‘‘weight function'' by $\bw(x)$, it is the density of the push forward measure of the Haar measure by the map $x\mapsto\cos(\gamma(x,e))$ where we recall that $\gamma(x,e)=2\arccos(|x_{\mathbf d}|/||x||_2)$.

\end{itemize}
One can show that the Jacobi polynomials (resp. beta distributions on $[-1,1]$) are the pushforward zonal spherical functions $Z_\ell$ (resp. the Haar measure) with shape parameters $(\alpha,\beta)$ depending on the base field and the dimension, see Table~\ref{t:CSR1} in the Appendix. In the case of real spheres, these Jacobi polynomials are the Legendre/Gegenbauer polynomials seen in Section~\ref{sec:Sphere_Example}. We recall that for shape parameters $(\alpha,\beta)$ the beta density distribution $\mathbf w$ is given by
\begin{align}
\label{eq:Beta}
\mathbf w(t)&=\frac{\Gamma(\alpha+\beta)}{2^{\alpha+\beta-1}\Gamma(\alpha)\Gamma(\beta)}(1-t)^{\alpha-1}(1+t)^{\beta-1}\mathds1_{[-1,1]}(t)\,,
\end{align}
where $\Gamma$ is the Gamma function. In particular, recall that one has
\[
\bp=\sum_\ell \sqrt{d_\ell}\bp_\ell^\star Z_\ell\quad\mathrm{and}\quad\bp_\ell^\star=\frac1{\sqrt{d_\ell}}\langle \bp,Z_\ell\rangle_{\mathbf L^2([-1,1],\bw)}\,,
\]
in $\mathbf L^2([-1,1],\bw)$. We further assume that there exists $s>0$, a (Sobolev) regularity parameter, such that
\[
\forall R\geq1\,,\quad
\sum_{\ell >R} d_\ell (\bp_\ell^\star)^2 \leq C(\bp,s,\bS)\, R^{-2s }.
\]
for some constant $ C(\bp,s,\bS)>0$ and for dimensions $(d_\ell)_{\ell\geq0}$ that depends only on $\bS$. 

\pagebreak[3]

Now, recall the definition of the set of models $\mathcal M_R$ in~\eqref{eq:MR} (the dimensions~$(d_\ell)_{\ell\geq0}$ are given by Table~\ref{t:CSR1}), of the estimator~$\widehat\lambda^R$ in~\eqref{eq:Min_Admissible}, of the adaptation~$\widehat R$ in~\eqref{defGL2}, of $\widehat\bp^{\widehat R}$ in~\eqref{pestimator}, and of $\bT_n$ in~\eqref{def:Tn}. Our estimation procedure is the same as in the sphere example the only difference is that the dimensions~$(d_\ell)_{\ell\geq0}$, $\widetilde R$ and the zonal spherical function $Z_\ell$ depend on the latent space under consideration, see Table~\ref{t:CSR1} in the Appendix. 

\begin{theorem}
\label{thm:SSCCSSversion}
Let $\bS$ be a compact symmetric space of rank one with Riemanian dimension $\mathbf d-1$. There exist constants $C_0,C_1,C_2,\kappa_0,\kappa_1>0$ such that the following holds. Let $\alpha\in(0,1/3)$ and $n,\,R\geq0$ such that $n\geq2\widetilde R$ and $n^3\geq\widetilde R\log(2\widetilde R/\alpha)$ where~$\widetilde R$ is given in Table~\ref{t:CSR1} in the Appendix. Then it holds,
\begin{itemize}
\item {\bf [Convergence of the matrix of probabilities]} 
\begin{align*}
%
\delta_2(\lambda({\bT}_{n}),\lambda^\star)\leq  
2\Big[\sum_{\ell>R}{d_\ell}(\bp_\ell^\star)^2 \Big]^{\frac12}+
C_0\sqrt{{\widetilde R}(1+{\log(\widetilde R/\alpha)})/n}
\end{align*}
with probability at least $1-3\alpha$ and 
$$
\E[\delta_2^2(\lambda({\bT}_{n}),\lambda^\star)]=\mathcal{O}
\left[\Big(\frac{n}{\log n }\Big)^{-\frac{2 s }{2s +(\bd-1)}}\right] \,.
$$
\item {\bf [Convergence of the matrix of finite rank approximation]}
 \begin{align*}
\delta_2(\widehat\lambda^R,\lambda^{\star R})\leq  4
\delta_2(\lambda^{\star R}, \lambda^{\star }) +
\kappa_0\sqrt{{\widetilde R}\left(1+\log\left({\widetilde R}/{\alpha}\right)\right)/n}.
\end{align*}
with probability at least $1-3\alpha$ and 
\begin{align*}
\E[\delta_2^2(\widehat \lambda^{R}, \lambda^{\star R})]
&\leq \kappa_1\left\{\delta_2^2(\lambda^{\star R}, \lambda^{\star }) + \frac{\widetilde R\log n}n\right\}.
\end{align*}
\item {\bf [Convergence of the adaptation]}
For $\kappa\geq \kappa_0\sqrt{11}$, it holds that 
$$\delta_2(\widehat \lambda^{\widehat R}, \lambda^{\star})\leq
C_1\min_{R\in\R}\Big\{\delta_2(\lambda^{\star R}, \lambda^{\star })+\kappa\sqrt{\frac{\widetilde R\log n}n}\Big\},$$
with probability $1-3n^{-8}$. Furthermore, for $\kappa\geq \kappa_0\sqrt{5}$, it holds that
$$\E[\delta_2^2(\widehat \lambda^{\widehat R}, \lambda^{\star})]\leq
C_2\min_{R\in\R}\Big\{\delta_2^2(\lambda^{\star R}, \lambda^{\star })+\kappa^2\frac{\widetilde R\log n}n\Big\}.$$
\end{itemize}
\end{theorem}
A proof can be found in Appendix~\ref{app:SSCCSSversion}. Note that the same results as in Proposition~\ref{cor:variancepr} and Corollary~\ref{u} hold when $\bS$ is a compact symmetric space of rank one. Namely, adaptive estimation of the envelope function $\bp$ is possible when $\bp$ is a polynomial. 

\section{Numerical Experiments}\label{sec:numeric}
\subsection{Simulations}\label{simus}

In this section we shall assess the performances of our estimation procedure by estimating numerous envelope functions $\bp$. We consider the example of $\bS= \mathds S^{2}$, the unit sphere  in dimension $\bd=3$. 
The functions $G_\ell^\beta$ turn to be the Legendre polynomials and the dimension of the space of spherical harmonics of degree $\ell$ is $d_\ell= 2\ell +1$.  

First, we shall explain how our algorithm works in practice to compute the adaptive estimator $\widehat\bp^{\widehat R}$ of $\bp$, see~\eqref{eq:Min_Admissible} and~\eqref{pestimator}. For sake of clarity, we deal with a simple example.  Suppose we are given an adjacency matrix $\bA$ of size $20\times 20$ and we set $R_{\max}=1$. Thus  $n=20$, $d_0=1$ and $d_1=3$.

\pagebreak[3]

\begin{description}
\item[Step 1]
Compute the $20$ eigenvalues of $\bA$ and sort them in decreasing order $\lambda_{(1)} \geq  \dots \geq  \lambda_{(20)}$, see Figure~\ref{graphalgo}.
\item[Step 2]
Take $0\leq R\leq R_{\max}$. Generate $\mathfrak S_{R+2}$, the set of all permutation of~$\{\mathbf 0,d_0,\ldots,d_R\}$, the set with $R+2$ elements. The factor $+1$ in $R+2=(R+1)+1$ is due to the ‘‘zeros'' (represented by the symbol $\mathbf 0$) to be placed, see Step 3 for a proper definition. For instance, for $R=1$, we have 
 $$\mathfrak S_{R+2}= \left\{
\begin{array}{ ccc }
   \sigma_1= [d_1, & d_0, & \mathbf 0]\\
\sigma_2= [d_1, & \mathbf 0, & d_0 ]\\
\sigma_3 = [ d_0, & d_1, & \mathbf 0] \\
\sigma_4= [ d_0, & \mathbf 0, & d_1] \\
\sigma_5=  [\mathbf 0, & d_0, &d_1 ] \\
\sigma_6= [ \mathbf 0, & d_1, & d_0]
 \end{array}
\right.$$
\item[Step 3]
For each permutation $\sigma_{i}$, $i \in \{ 1,\dots, 6 \}$ of $\mathfrak S_{R+2}$, compute the following $(\widetilde \bp_{\sigma_i,\ell})_{\ell \in \{ 0,1,2\}}$  
which are the ‘‘stage means'' of the $\lambda_{(i)}$'s,  $i \in \{ 1,\dots, 20 \}$ according to the order of appearance of the $d_\ell$'s in the permutation $\sigma_i$. For instance, for $\sigma_1=  [d_1, \quad   d_0,  \quad  \mathbf 0  ] $ (see Figure \ref{graphalgo}), we get 
$$  \widetilde \bp_{\sigma_1,2} = \frac 1 3 \sum_{\ell=1}^3 \lambda_{(\ell)} , \quad \widetilde \bp_{\sigma_1,1} =\lambda_{(4)},\quad   \widetilde \bp_{\sigma_1,0}=0, $$
and for 
 $\sigma_4=  [d_0, \quad   \mathbf 0,  \quad  d_1  ]$  one gets
$$  \widetilde \bp_{\sigma_2,1} = \lambda_{(1)} ,\quad  \widetilde \bp_{\sigma_2,0}= 0 ,\quad  \widetilde \bp_{\sigma_2,2} =\frac 1 3 \sum_{i=18}^{20} \lambda_{(\ell)} . $$
In Step 2, we have called ‘‘zeros'' the fact that we always set $\widetilde \bp_{\sigma_i,0}= 0$.
\item[Step 4]
For each permutation $\sigma_i$, compute the corresponding vector $\widetilde \lambda_{\sigma_i}$ of size $20$, containing the~$\widetilde \bp_{\sigma_i,\ell}$ with multiplicity $d_\ell$. Then compute the risk $\mathrm{Score}(\sigma_i)$ for each $\sigma_i$. For example for $\sigma_1=  [d_1, \quad   d_0,  \quad  \mathbf 0  ] $  (see Figure \ref{graphalgo}), one gets
$$\widetilde \lambda_{\sigma_1}= (\underbrace{\widetilde \bp_{\sigma_1,2}, \widetilde \bp_{\sigma_1,2}, \widetilde  \bp_{\sigma_1,2}}_{d_1=3}, \underbrace{\widetilde \bp_{\sigma_1,1}}_{d_0=1}, \underbrace{0,\dots,0}_{n-d_0-d_1=16})$$
and its risk is $\displaystyle \mathrm{Score}({\sigma_1}) =  \sum_{\ell=1}^{20}(\lambda_{(\ell)}-\widetilde \bp_{\sigma_1,\ell} )^2$.
\item[Step 5]
Select the permutation $\sigma_{\min}$ such that $\displaystyle \sigma_{\min}= \argmin_{\sigma_i}\;  \mathrm{Score}({\sigma_i})$. 
\item[Step 6]
Get the estimate $\widehat \lambda^{R}$ defined by
$$ \widehat \lambda^{R} =(\underbrace{\widetilde \bp_{\sigma_{\min},1}}_{d_0=1}, \underbrace{\widetilde \bp_{\sigma_{\min},2}, \widetilde \bp_{\sigma_{\min},2}, \widetilde \bp_{\sigma_{\min},2}}_{d_1=3} )= (\widehat\bp_0^{R},\widehat\bp_1^{R}, \widehat\bp_1^{R}, \widehat\bp_1^{R})\,,
 $$
 see~\eqref{eq:Min_Admissible}.
 \item[Step 7]
 Iterate Steps 2 to 6 for $R=0$ to $R_{\max}$. Compute the level  $\widehat R$ according to~\eqref{defGL2} and the adaptive estimator  $\widehat\bp^{\widehat R}(t)$ according to~\eqref{pestimator}.
  \item[Step 8]
  Troncate $\widehat\bp^{\widehat R}(t)$ so as to it belongs to $[0,1]$. 
\end{description}

\pagebreak[3]

\begin{figure}[h]
\centering
\includegraphics[width=0.7\textwidth]{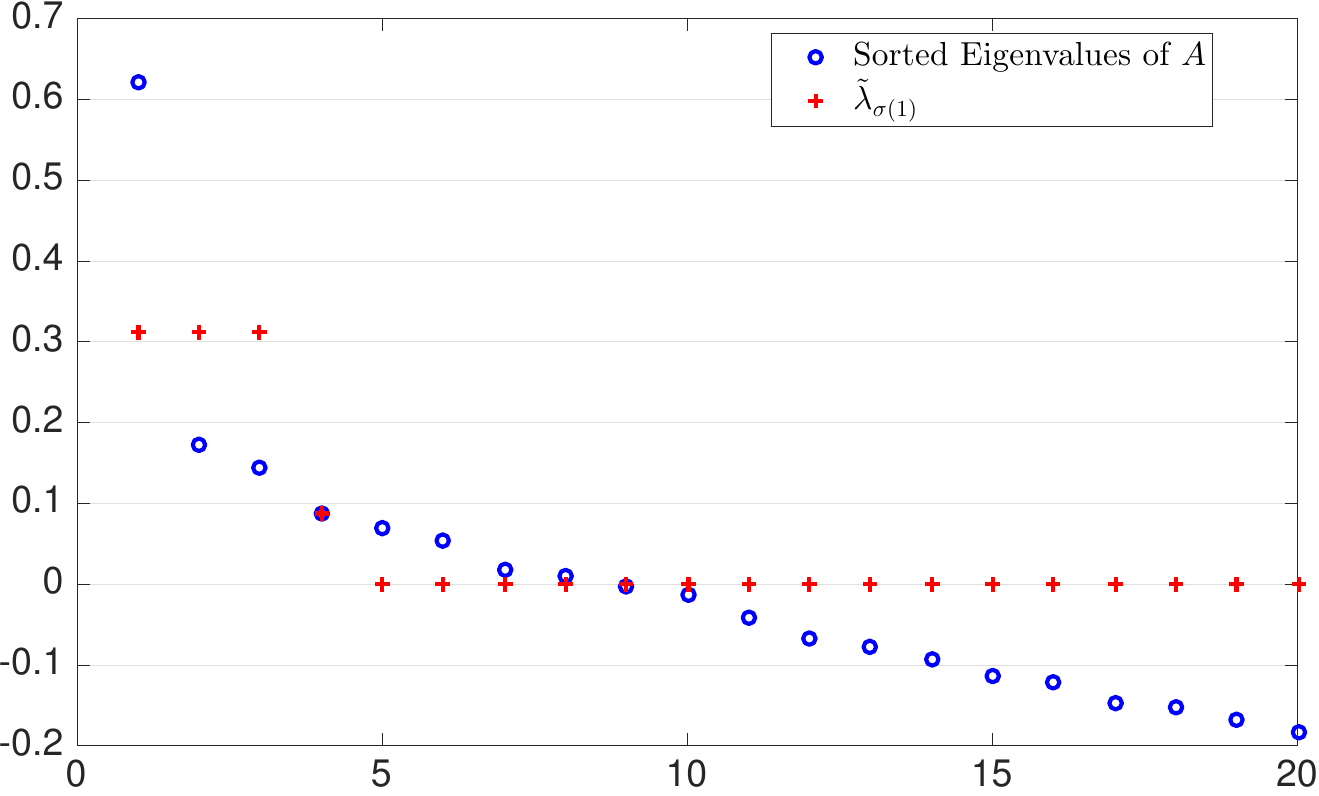}
\caption{Plot of the 20 sorted eigenvalues $\lambda_{(i)}$ of adjacency matrix $\bA$ and the values of vector~$\widetilde \lambda_{\sigma(1)}$.}
\label{graphalgo}
\end{figure}

Of course, the choice of level $R$ is crucial and the estimation is sensitive to $R$.  That is why we use our selection method, as described in Section~\ref{sec:Adaptation} (see Step~7 in the description of the algorithm above). As almost all estimators selection methods, this Goldenshluger-Lepski method uses an hyper-parameter $\kappa$. Our theoretical result ensures a good performance as soon as $\kappa$ is large enough, but it is well known that a more precise choice is better in practice. Heuristics exist to calibrate $\kappa$, but they are all based on the behavior of the estimator for very large $R$ (see for instance \cite{Baudry2012}). Hence these techniques  are not possible here, due to the computational cost of the estimation when $R$ is large (we can hardly consider larger than $R_{\max}=7$ because of the complexity in $(R_{\max}+2)!$).
Fortunately, the stability of the estimation allows us to choose here a fixed $\kappa$, namely $\kappa=0.25$, and this choice ensures good selection of $\hat R$ for a wide range of functions $\bp$.

Now, let us deal with the estimation of the six following envelope functions~$\bp$ 
\begin{align*}
\bp_1(t)&= \Big (\frac{1+t}{2} \Big)^4, \\
\bp_2(t)&= \mathds{1}_{t>0.7},\\ 
\bp_3(t)&= e^{-(t-1)^2}, \\ 
\bp_{4}(t)&=0.5+0.5 \sin \Big ({\pi t}/{2} \Big ), \\
\bp_5(t)&= \frac 1 3 +\frac{1}{12}\Big(35 t^4-30 t^2 +3\Big),\\
\bp_6(t)&= t^{10} \mathds{1}_{t>0}\,.
\end{align*}
 We consider graphs of size $n=5000$. We set $R_{\max}=4$ and $\kappa=0.25$ for the adaptive selection rule of $R$, see~\eqref{defGL2}.

 Figure \ref{graphp} presents our simulation results. For each envelope function $\bp$, we represent on the top side, the estimated coefficients $\widehat \bp_{\ell}^{ \widehat R}$  and the true coefficients~$\bp^\star_{\ell}$ with their multiplicity $2\ell +1$. On the bottom side, we represent the estimated envelope function $\widehat \bp $ and the true $\bp $. Note that our procedure is not constrained by dealing with envelope functions $\bp$ defining positive kernels~$W$. Such an example is given by  the step function $\bp_2$ as its Fourier coefficients $\bp_{2,\ell}^\star$'s can be negative, see Figure \ref{graphp}. 

The estimation of all functions are good except for the step function $\bp_2$ which is more demanding due to its discontinuity. Despite that function $\bp_6$ is not easy to be estimated because of its flatness, our estimation is satisfying. Furthermore, it is interesting to remark that except for $\bp_2$, the estimated coefficients are very close to the true ones. 



\begin{figure}[h!]
\begin{center}
\includegraphics[width=0.49\textwidth]{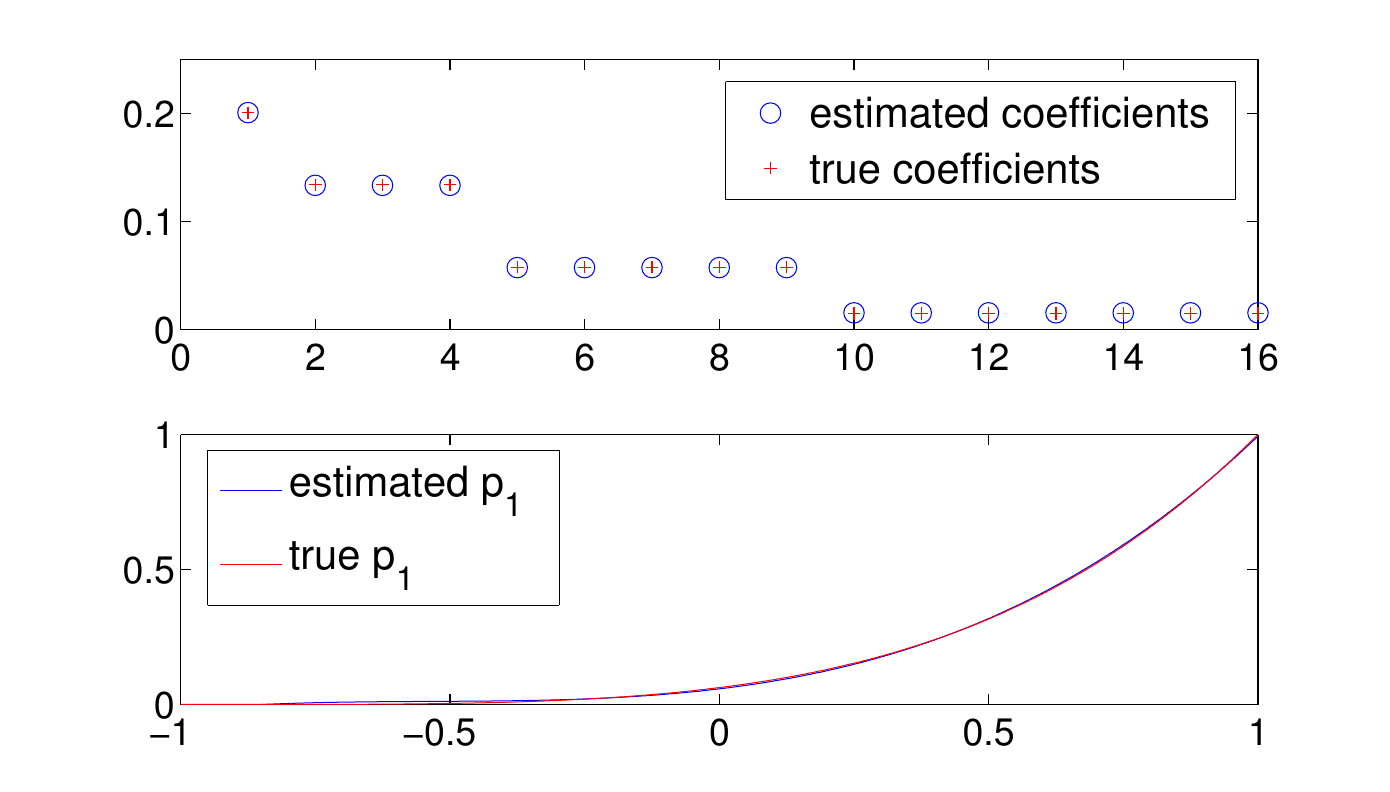}
 \includegraphics[width=0.49\textwidth]{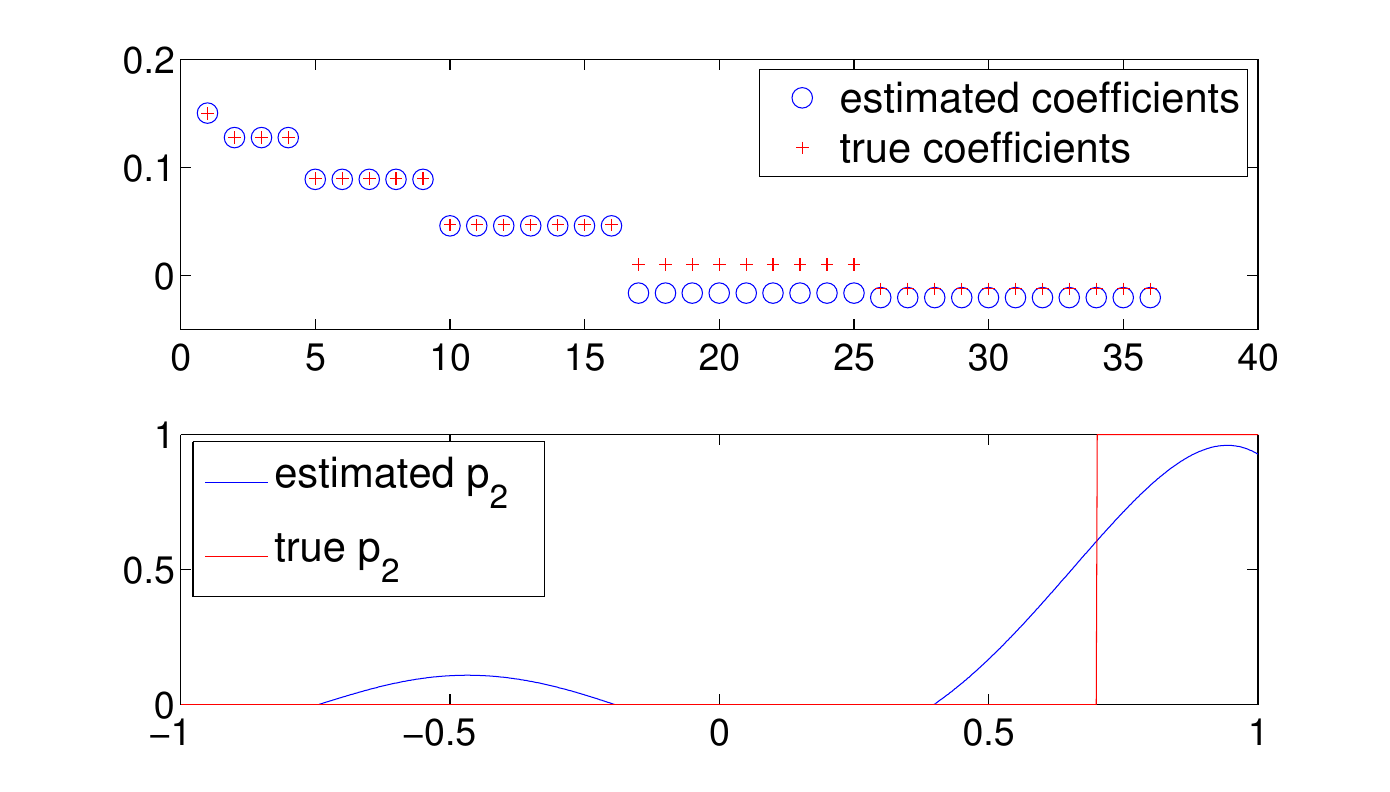}  \\
\includegraphics[width=0.49\textwidth]{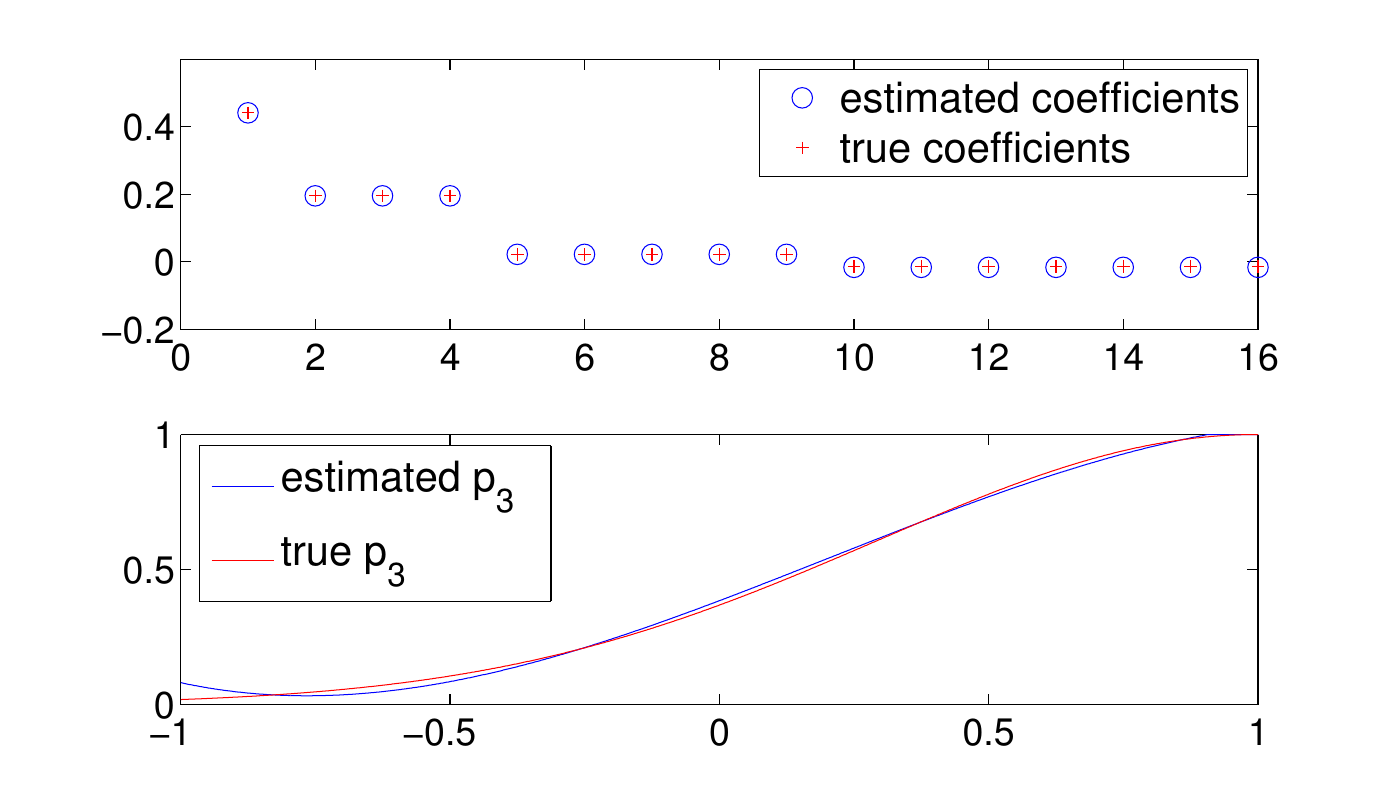} 
 \includegraphics[width=0.49\textwidth]{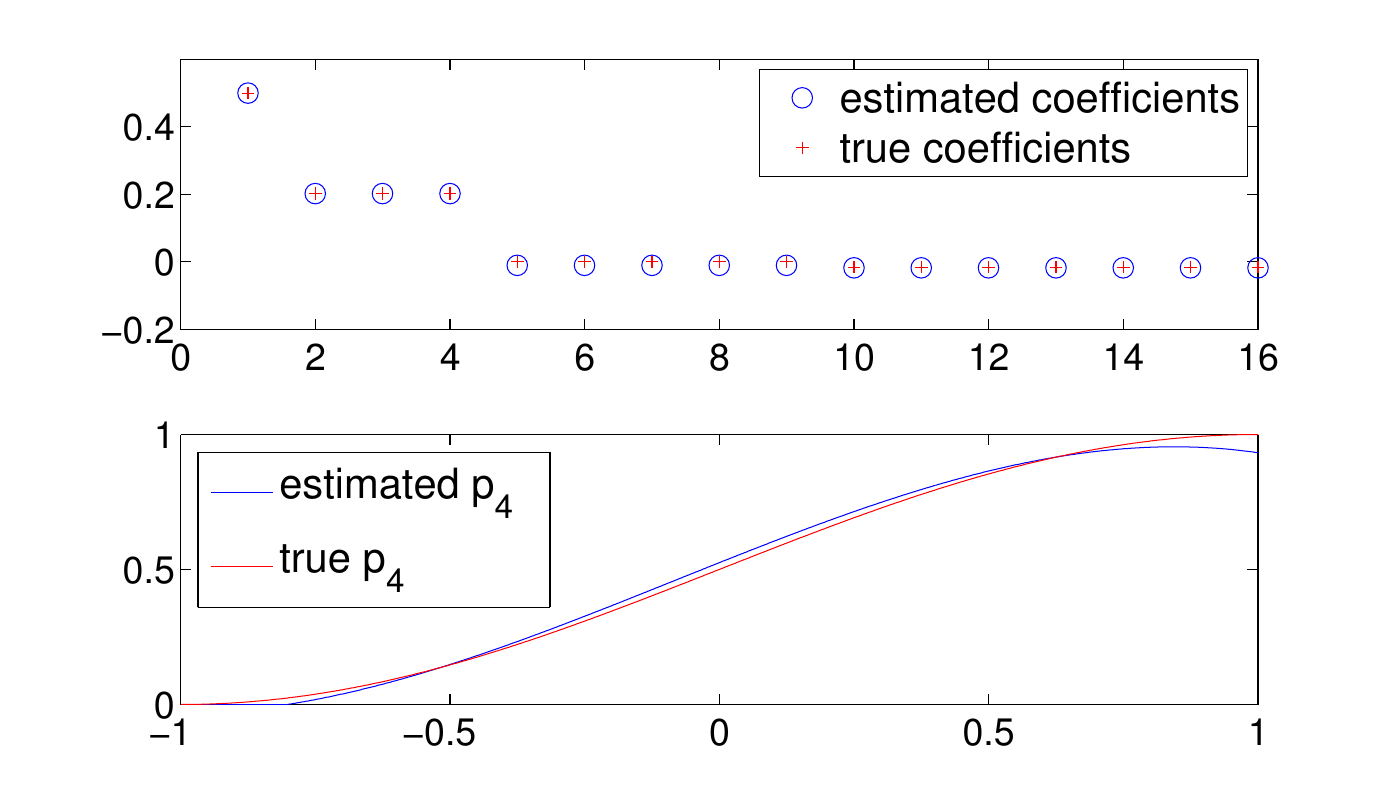} \\
\includegraphics[width=0.49\textwidth]{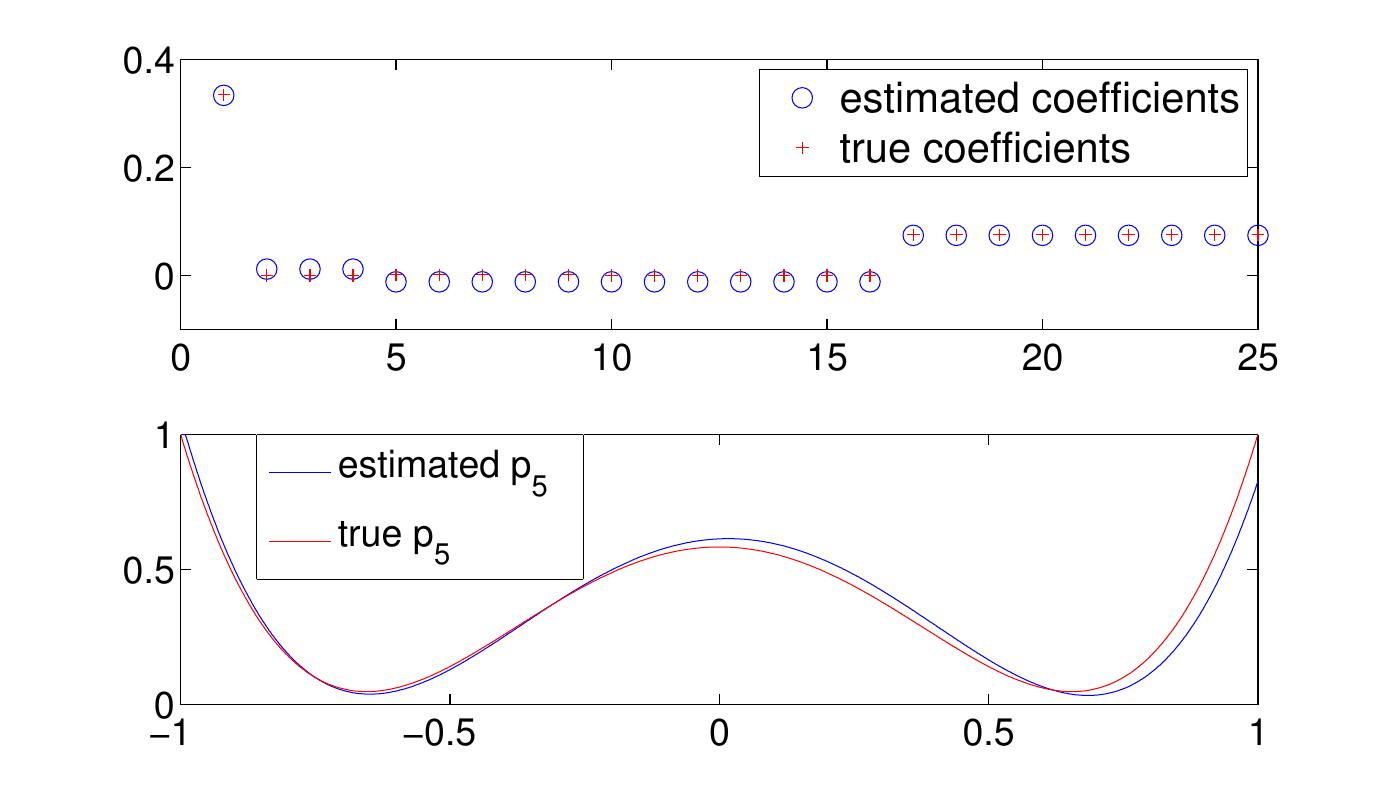} 
 \includegraphics[width=0.49\textwidth]{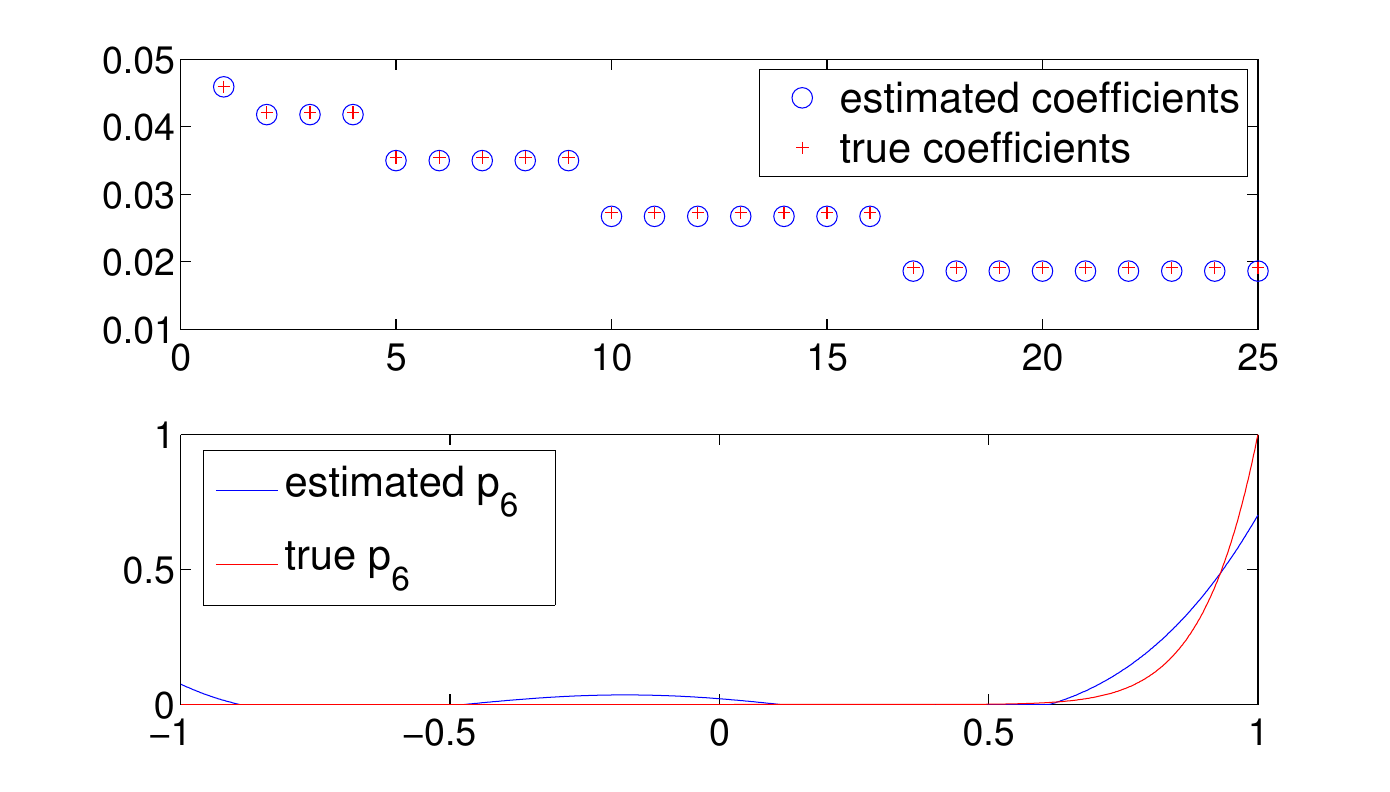}
\caption{Estimation of envelope functions $\bp_1, \dots, \bp_6$.}
\label{graphp}
\end{center}
\end{figure}

\subsection{How to speed the algorithm}\label{speed}

The aforementioned algorithm is a spectral algorithm with complexity $\mathcal O(n^3)$, which is standard. Costly steps might be {\bf Steps 2-5} which rely on evaluating the empirical risk function for all possible permutations of the ordering of the eigen-spaces. These steps are of complexity $\mathcal O(R^{\mathbf d})$, which is exponential on the number $R$ of estimated eigen-spaces. Yet, these steps are just meant to cluster empirical eigenvalues (on the real line) into groups of prescribed sizes $d_0,d_1,\ldots, d_R$ and one might be willing to use standard clustering approaches to perform this task. 

For instance a {\it Hierarchical Agglomerative Clustering} (HAC) with {\it Single Linkage} function should be able to identify {\it well separated} clusters of eigenvalues. Looking at the dimensions matching the $d_\ell$'s along the HAC tree, one might recover the right labeling of the eigenvalues. The success of these clustering approaches will rely on the event that eigenvalues are {\it well separated}. This notion has been investigated in \citep{valdivia2019latent} where the authors showed that, for $n$ sufficiently large and under mild assumptions (essentially spectral gaps between distinct eigenvalues), the empirical eigenvalues gather in clusters of size $d_\ell$. As a conclusion, for $n$ sufficiently large, one might use standard clustering approaches such as HAC, in place of the costly {\bf Steps 2-5}, to identify right clusters of empirical eigenvalues.

\bigskip

{\bf Acknowledgements:}\ 
The authors would like to thank Pierre Loïc Méliot for many useful discussions on compact symmetric spaces. 
We would like to thank the referees for their interesting and fruitful comments that helped us to improve this article.


 \bibliographystyle{imsart-nameyear}
 \bibliography{biblio}



\appendix

\begin{landscape}
{\ }\vfill
\renewcommand{\arraystretch}{2}
{   
\begin{table}[htp]
\begin{tabular}{|c|c|c|c|c|}
\hline
$\bS$ & $d_\ell$ & $\displaystyle\widetilde R=\sum_{0\leq \ell\leq R}d_\ell $ & $t=\cos(\gamma(x,e))$ &  Density $\mathbf w$ \\
\hline \hline 
$\bbS^{\mathbf d-1}$ & $\binom {\ell+\bd -1}{\ell} -\binom{\ell +\bd -3}{\ell -2}$
& $\frac{2R+\mathbf d-1}{R+\mathbf d-1}{ R+\mathbf d -1 \choose \mathbf d -1}$ & $ x_{\mathbf d}$  
& $\mathrm{Beta}(\frac{\mathbf d-1}2,\frac{\mathbf d-1}2)$ \\
\hline
$\bbR\bbP^{\mathbf d-1}$ & 
$ \frac{(6 + \bd^2 + 8\ell (2\ell-3) + \bd ( 8 \ell-5)) \Gamma(\bd + 2\ell-3)}{
\Gamma(\bd-1) \Gamma(1 + 2\ell)}$
& $\frac{4R+\mathbf d-1}{2R+\mathbf d-1}{ 2R+\mathbf d -1 \choose \mathbf d -1}$ 
& $ \frac{x_{\mathbf d}^2-x_{1}^2-\cdots-x_{\mathbf d-1}^2}{x_{1}^2+\cdots+x_{\mathbf d}^2}$     
& $\mathrm{Beta}(\frac{\mathbf d-1}2,\frac12)$\\
\hline
$\bbC\bbP^{\mathbf d-1}$ & $\frac{2\ell+\mathbf d}{\mathbf d}{ \ell+\mathbf d -1 \choose \mathbf d-1}^2-\frac{2\ell+\mathbf d-2}{\mathbf d}{ \ell+\mathbf d -2 \choose \mathbf d-1}^2$
& $\frac{2R+\mathbf d}{\mathbf d}{ R+\mathbf d -1 \choose \mathbf d-1}^2$ 
& $ \frac{|x_{\mathbf d}|^2-|x_{1}|^2-\cdots-|x_{\mathbf d-1}|^2}{|x_{1}|^2+\cdots+|x_{\mathbf d}|^2}$ 
& $\mathrm{Beta}({\mathbf d-1},1)$\\
\hline
$\bbH\bbP^{\mathbf d-1}$ 
& $ \frac{2 (\ell+1) (\bd ( 4 \bd^2-1) + 2 \bd (4 \bd-1) \ell + (4 \bd-1) \ell^2) \Gamma(2 \bd + \ell-1) \Gamma(
  2 \bd + \ell)}{\Gamma(2 \bd)\Gamma(2 + 2\bd) \Gamma(2 + \ell)^2} $
& $\frac{2R+2\mathbf d+1}{(2\mathbf d+1)(R+1)} {{R+2\mathbf d}\choose 2\mathbf d} {{R+2\mathbf d-1}\choose{2\mathbf d-1}}$
& $ \frac{|x_{\mathbf d}|^2-|x_{1}|^2-\cdots-|x_{\mathbf d-1}|^2}{|x_{1}|^2+\cdots+|x_{\mathbf d}|^2}$  
& $\mathrm{Beta}({2\mathbf d-2},2)$ \\
\hline
$\bbO\bbP^{2}$& 
$ \frac{(4 + \ell) (5 + \ell)^2 (6 + \ell)}{924}$ 
&  $\frac{2R+11}{385}{{R+7}\choose 4}{{R+10}\choose 10}$
& $ \frac{|x_{3}|^2-|x_{1}|^2-|x_{2}|^2}{|x_{1}|^2+|x_{2}|^2+|x_{3}|^2}$    
& $\mathrm{Beta}(8,4)$ \\
\hline
\end{tabular}
\medskip
\caption{Review of the dimensions $d_\ell$ of the spherical representations, the distance $\cos(\gamma(x,e))$ to the identity $e$, the weight function $\bw(t)$ 
of the compact symmetric spaces $\bS$ of rank~$1$. These latter are respectively the density $\mathbf w$ and the orthonormal polynomials of the beta law on~$[-1,1]$ with shape parameters $(\alpha,\beta)$, see~\eqref{eq:Beta}.}
\label{t:CSR1}
\end{table}%
}
\vfill{\ }
\end{landscape}

\section{Proofs}
\label{sec:proofs}
\subsection{Proof of Proposition~\ref{prop:BvH}}
\label{proof:BvH}
This result is a consequence of \cite[Corollary 3.12]{bandeira2014sharp} and \cite[Remark 3.13]{bandeira2014sharp} 
with $X_{ij}=\bA_{ij}-(\bTheta_0)_{ij}$ a centered but not symmetric random variable, $\varepsilon=1/2$ say, $\widetilde\sigma^2=\bD_0$ by definition, and observing that $ \widetilde\sigma^2_*=\max_{ij}((\bTheta_0)_{ij}\vee(1-(\bTheta_0)_{ij}))\leq1$. It gives
\[
\forall t>0,\quad
\pr{\|\bA-\bTheta_0\|\geq3\sqrt{2\bD_0}+Ct}\leq n\exp(-t^2)\,,
\]
for some universal constant $C>0$.

\subsection{Proof of Theorem~\ref{thm:KG_revisited}}
\label{proof:KG_revisited}
Let $R\geq1$ and define
\begin{align*}
\Phi_i^n		&:=(1/\sqrt n)(\phi_i(X_1),\ldots,\phi_i(X_n))\in\bbR^n,\\
K_R			&:=\mathrm{Diag}(\lambda_1(\bbT_W)\ldots,\lambda_R(\bbT_W))\in\bbR^{R\times R},\\
E_{R,n}		&:=(\langle\Phi_i^n,\Phi_j^n\rangle-\delta_{ij})_{i,j\in[R]}\in\bbR^{R\times R},\\
X_{R,n}		&:=\Big[\Phi_1^n\cdots \Phi_R^n\Big]\in\bbR^{n\times R},\\
A_{R,n}		&:=\Big(X_{R,n}^\top X_{R,n}\Big)^{\frac12}\in\bbR^{R\times R} \mbox{ and note that }A_{R,n}^2=\mathrm{Id}_R+E_{R,n},\\
\bT_{R,n}		&:=\sum_{r=1}^R\lambda_r(\bbT_W)\Phi_r^n(\Phi_r^n)^\top=X_{R,n}K_RX_{R,n}^\top\in\bbR^{n\times n},\\
\widetilde{\bT}_{R,n}&:=((1-\delta_{ij}) \bT_{R,n})_{i,j\in[n]}\in\bbR^{n\times n},\\
\bT^\star_{R,n}	&:=A_{R,n}K_RA_{R,n}\in\bbR^{R\times R},\\
\mathrm{and\ }W_R(x,y)&:=\sum_{i=1}^R\lambda_i(\bbT_W)\phi_i(x)\phi_i(y),
\end{align*}
where the last identity holds point-wise. Observe that $A_{R,n}^2=\mathrm{Id}_R+E_{R,n}$. It holds
\beq
\label{eq:proof_delta2_0}
\delta_2(\lambda(\bbT_W),\lambda(\bbT_{W_R}))=\Big(\sum_{r>R}\lambda_r^2(\bbT_W)\Big)^{\frac12}\,.
\eeq
Note the equalities between spectra $\lambda(\bbT_{W_R})=\lambda(K_R)$ and $\lambda(\bT_{R,n})=\lambda(\bT_{R,n}^\star)$ where the last one follows by using a SVD of $X_{R,n}$. Hence, we deduce that
\beq
\notag
\delta_2(\lambda(\bbT_{W_R}),\lambda(\bT_{R,n}))=\delta_2(\lambda(K_R),\lambda(\bT_{R,n}^\star))\leq\|\bT_{R,n}^\star-K_R\|_F=\|A_{R,n}K_RA_{R,n}-K_R\|_F\,,
\eeq
by Hoffman-Wielandt inequality, see \cite[page 118]{koltchinskii2000random} for instance. Equation~(4.8) at \cite[page 127]{koltchinskii2000random} gives that
\beq
\label{eq:proof_delta2}
\delta_2(\lambda(\bbT_{W_R}),\lambda(\bT_{R,n}))\leq\sqrt2\|K_R\|_F\|E_{R,n}\|=\sqrt2\|W_R\|_2\|E_{R,n}\|\,,
\eeq
Actually, one can remove the constant $\sqrt 2$ using Ostrowski's theorem, see \cite[Theorem A.2]{braun2006accurate} for instance. Also, by  Hoffman-Wielandt inequality, we have
\beq
\label{eq:proof_delta2_2}
\delta_2(\lambda(\bT_{R,n}),\lambda(\widetilde{\bT}_{R,n}))\leq\|\widetilde{\bT}_{R,n}-{\bT}_{R,n}\|_F
=\Big[\frac1{n^2}\sum_{i=1}^nW_R^2(X_i,X_i)\Big]^{\frac12}\,,
\eeq
and
\beq
\label{eq:proof_delta2_3}
\delta_2(\lambda(\widetilde{\bT}_{R,n}),\lambda(\bT_{n}))\leq\|\widetilde{\bT}_{R,n}-{\bT}_{n}\|_F=\Big[\frac1{n^2}\sum_{i\neq j}(W-W_R)^2(X_i,X_j)\Big]^{\frac12}\,.
\eeq


Invoke Lemma~\ref{lem:Lemma1_KG_revisited} to bound~\eqref{eq:proof_delta2}, Lemma~\ref{lem:Lemma1b_KG_revisited} to bound~\eqref{eq:proof_delta2_2} and Lemma~\ref{lem:Lemma2_KG_revisited} to bound~\eqref{eq:proof_delta2_3}.
\begin{lemma}
\label{lem:Lemma1_KG_revisited}
Let $R\geq1$ and denote by $\rho(R):=\max(1,\|\sum_{r=1}^R\phi_r^2\|_\infty-1)$ then it holds
\[
\forall t>0,\quad
\pr{\|E_{R,n}\|\geq t}
\leq 2R\exp\Big[-\frac n{2\rho(R)}\frac{t^2}{1+t/(3n)}\Big]\,.
\]
In particular, for all $\alpha\in(0,1)$ and for $n^3\geq \rho(R)\log(2R/\alpha)$, it holds
\[
\pr{\|E_{R,n}\|\geq \sqrt{\frac{\rho(R)\log(2R/\alpha)}n}}
\leq \alpha\,.
\]
\end{lemma}

\begin{lemma}
\label{lem:Lemma1b_KG_revisited}
Let $R\geq1$ and $\alpha\in(0,1)$ then, with probability at least $1-\alpha$, it holds
\[
\frac1{n^2}\sum_{i=1}^nW_R^2(X_i,X_i)\leq
\Bigg[1+\max_{1\leq r\leq R}||\phi^2_r||_\infty \sqrt{\frac{\log(R/\alpha)}{{2n}}}\Bigg]\frac{2\rho(R)||W_R||^2}{n}\,.
\]
\end{lemma}

\begin{lemma}
\label{lem:Lemma2_KG_revisited}
It holds, for all $\alpha\in(0,1)$,
\[
\pr{\frac1{n(n-1)}\sum_{i\neq j}(W-W_R)^2(X_i,X_j)\geq\sum_{r>R}\lambda_r^2(\bbT_W)+
\|W-W_R\|_\infty^2
\sqrt{\frac{\log(2/\alpha)}{n-1}}}\leq\alpha\,.
\]
\end{lemma}
\noindent
These lemmas are proven in Appendix~\ref{proof:Lemma1KG}, Appendix~\ref{proof:Lemma1bKG} and Appendix~\ref{proof:Lemma2KG}. Collecting~\eqref{eq:proof_delta2_0},~\eqref{eq:proof_delta2},~\eqref{eq:proof_delta2_2} and~\eqref{eq:proof_delta2_3}, the triangular inequality gives the result.

\subsection{Proof of Lemma~\ref{lem:Lemma1_KG_revisited}}
\label{proof:Lemma1KG}
Observe that $nE_{R,n}=\sum_{i=1}^n(Z_R(X_i)Z_R^\top(X_i)-\mathrm{Id}_R)$ is a sum of independent centered symmetric matrices where we denote by $Z_R(x):=(\phi_1(x),\ldots,\phi_R(x))$. In particular, $Z_R(X_i)Z_R^\top(X_i)$ are rank one matrices so that it holds
\begin{align*}
\| Z_R(X_i)Z_R^\top(X_i)-\mathrm{Id}_R\|&=1\vee\big(\|Z_R(X_i)\|_2^2-1\big)\\
&=1\vee\big((\sum_{r=1}^R\phi_r^2(X_i))-1\big)\\
&\leq
1\vee\big(\|\sum_{r=1}^R\phi_r^2\|_\infty-1\big)=: \rho(R)\,.
\end{align*}
Moreover, one has
\begin{align*}
\sigma_{R,n}^2:&=n\|\expect{(Z_R(X_1)Z_R^\top(X_1)-\mathrm{Id}_R)^2}\|\\
&=n\|\expect{\|Z_R(X_1)\|_2^2Z_R(X_1)Z_R^\top(X_1)-2Z_R(X_1)Z_R^\top(X_1)-\mathrm{Id}_R}\|\\
&=n\|\expect{\|Z_R(X_1)\|_2^2Z_R(X_1)Z_R^\top(X_1)}-\mathrm{Id}_R\|\\
&\leq n\Big\|\|\sum_{r=1}^R\phi_r^2\|_\infty\expect{Z_R(X_1)Z_R^\top(X_1)}-\mathrm{Id}_R\Big\|\\
&= n\Big\|\|\sum_{r=1}^R\phi_r^2\|_\infty\mathrm{Id}_R-\mathrm{Id}_R\Big\|\\
&= n\Big|\|\sum_{r=1}^R\phi_r^2\|_\infty-1\Big|\leq n  \rho(R)
\end{align*}
where we invoke that a.s. $\|Z_R(X_1)\|_2^2Z_R(X_1)Z_R^\top(X_1)\preccurlyeq
(\|\sum_{r=1}^R\phi_r^2\|_\infty)Z_R(X_1)Z_R^\top(X_1) $. It follows from the matrix Bernstein inequality given in \cite[Theorem 6.1.1]{tropp2012user}
\[
\forall t>0,\quad
\pr{\|E_{R,n}\|\geq t}
\leq 2R\exp\Big[-\frac n{2 \rho(R)}\frac{t^2}{1+t/(3n)}\Big]\,.
\]
Indeed, we have used \cite[Theorem 6.1.1]{tropp2012user} with 
\begin{align*}
X_k&\leftarrow Z_R(X_i)Z_R^\top(X_i)-\mathrm{Id}_R,\\
R&\leftarrow \rho(R),\\
Y&\leftarrow nE_{R,n},\\
\sigma^2\leq n||\expect{X_1^2}||&\leftarrow \sigma_{R,n}^2,\\
t&\leftarrow nt,
\end{align*}
according to the notation of \cite{tropp2012user} on the left hand side and our notation on the right hand side.
It proves the lemma. 

\subsection{Proof of Lemma~\ref{lem:Lemma1b_KG_revisited}}
\label{proof:Lemma1bKG}
Observe that
\begin{align*}
\frac1{n^2}\sum_{i=1}^nW_R^2(X_i,X_i)
&=
\frac1{n^2}\sum_{i=1}^n\big(\sum_{r=1}^R\lambda_r(\bbT_W)\phi_r^2(X_i)\big)^2\notag\\
&=\frac1{n^2}
\sum_{r,s\in[R]}\lambda_r(\bbT_W)\lambda_s(\bbT_W)
\big(\sum_{i=1}^n\phi_r^2(X_i)\phi_s^2(X_i)
\big)\\
&=x^\top \bA x\leq||\bA||||x||_2^2
\end{align*}
with $x=(\lambda_1(\bbT_W)/\sqrt n,\ldots,\lambda_R(\bbT_W)/\sqrt n)$ and $\bA=((1/n)\sum_{i=1}^n\phi_r^2(X_i)\phi_s^2(X_i))_{r,s}$. Note that $\bA$ is an irreducible and aperiodic matrix since its coefficients are positive. It follows by Perron-Frobenius theorem that
\[
||\bA||\leq\frac1n\max_{1\leq r\leq R}(\sum_{s=1}^R\sum_{i=1}^n\phi_r^2(X_i)\phi_s^2(X_i))
\]
Now, this last quantity can be upper bounded as follows
\begin{align*}
\frac1n\sum_{s=1}^R\sum_{i=1}^n\phi_r^2(X_i)\phi_s^2(X_i)&=\frac1n\sum_{i=1}^n\phi_r^2(X_i)(\sum_{s=1}^R\phi_s^2(X_i)),\\
&\leq(\frac1n\sum_{i=1}^n\phi_r^2(X_i))(1+\rho(R)).
\end{align*}
Using the bound 
\[
\phi_r^2(X_1)\leq \max_{1\leq r\leq R}{||\phi_r^2||_\infty}=:a_R\,,
\]
and Hoeffding inequality \cite[page 34]{boucheron2013concentration}, we deduce that

\[
\forall t>0,\quad
\pr{\frac1{n}\sum_{i=1}^n\phi_r^2(X_i)>\expect{\phi_r^2(X_1)}+t}\leq\exp\big(-\frac{2nt^2}{a_R^2}\big)\,.
\]
Observe that $\expect{\phi_r^2(X_1)}=1$. Let $\alpha\in(0,1)$, choosing $t^2=a_R^2\log(R/\alpha)/(2n)$ and taking an union bound, it holds that
\begin{align*}
\pr{\forall r\in[R],\quad \frac1{n}\sum_{i=1}^n\phi_r^2(X_i)\leq1+ 
\frac{a_R\log^{\frac12}(R/\alpha)}{\sqrt{2n}}}
\geq 1-\alpha
\end{align*}
It results in 
\[
\pr{||\bA||\leq\Big(1+ \frac{(1+\rho(R))\log^{\frac12}(R/\alpha)}{\sqrt{2n}}\Big)(1+\rho(R))}
\geq 1-\alpha
\]
On this event, we deduce that
\begin{align*}
\frac1{n^2}\sum_{i=1}^nW_R^2(X_i,X_i)
&\leq||\bA||||x||_2^2,\\
&\leq\Big(1+ \frac{a_R\log^{\frac12}(R/\alpha)}{\sqrt{2n}}\Big)\frac{(1+\rho(R))||W_R||^2}{n}\,,
\end{align*}
which gives the result.

\subsection{Proof of Lemma~\ref{lem:Lemma2_KG_revisited}}
\label{proof:Lemma2KG}
By a standard inequality of Hoeffding \citep{hoeffding1963probability}, 
for a bounded kernel $h$, for all $\alpha\in(0,1)$,
\[
\pr{\Big|\frac1{n(n-1)}\sum_{i\neq j}h(X_i,X_j)-\expect{h(X_1,X_2)}\Big|>\|h\|_\infty\sqrt{\frac{\log(2/\alpha)}{n-1}}}\leq\alpha
\]
Applying this result for $h=(W-W_R)^2$ and noticing that
\begin{itemize}
\item $\expect{h(X_1,X_2)}=\|W-W_R\|_2^2=\sum_{r>R}\lambda_r^2(\bbT_W)$,
\item $\|h\|_\infty=\|W-W_R\|_\infty^2$,
\end{itemize}
the result follows.

\subsection{Proof of Corollary~\ref{cor:Canonical_Kernels}}
\label{proof:CanonicalKernels}
The symmetric kernel $h:=(W-W_{ R})^2-\expect{(W-W_{R})^2}$ is $\sigma$-canonical, see \cite[Definition 3.5.1]{de2012decoupling} for a definition. The following important improvement of Hoeffding's inequalities for canonical kernels was proved by  \cite{arcones1993limit}, it holds that there exists two universal constants $C_1>0$ and $C_2>0$ such that for all $\alpha\in(0,1)$,
\[
\pr{\Big|\frac1{n(n-1)}\sum_{i\neq j}h(X_i,X_j)\Big|>C_1\|h\|_\infty{\frac{\log(C_2/\alpha)}{n}}}\leq\alpha\,.
\]
We deduce that it holds, for all $\alpha\in(0,1)$,
\[
\pr{\frac1{n(n-1)}\sum_{i\neq j}(W-W_{R})^2(X_i,X_j)\geq
\|W-W_{R}\|^2_{2}
+
C_1
\|W-W_{R}\|_\infty^2
\frac{\log(C_2/\alpha)}{n}}\leq\alpha\,,
\]
which proves the corollary substituting Lemma~\ref{lem:Lemma2_KG_revisited} by the aforementioned inequality.
\subsection{Proof of Proposition~\ref{thm:KG_Sphere}}
\label{proof:KG_Sphere}
Define
\[
\forall t\in[-1,1],\quad\bp^R(t):=\sum_{\ell=0}^R\bp_\ell^\star c_\ell G_\ell^\beta(t)\,.
\]
We apply  \corref{Canonical_Kernels} to the kernel as follows.
\[
\forall x,y\in \mathds S^{\bd-1},\quad W_{\widetilde R}(x,y):=\sum_{\ell=0}^R\bp^\star_\ell c_\ell G_\ell^\beta(\langle x,y\rangle)=\bp^R(\langle x,y\rangle)\,,
\]
First, note that
\beq\label{biais}
\|W-W_{\widetilde R}\|_2=\|\bp-\bp^R\|_2=\Big[\sum_{\ell>R}d_\ell(\bp^\star_\ell)^2\Big]^{\frac12}\,.
\eeq
Next, invoke \cite[Corollary 1.2.7]{dai2013approximation} to get that
\beq
\notag
\forall\ell\geq0,\quad\sum_{j=1}^{d_\ell}Y_{\ell j}^2=d_\ell\,.
\eeq
It follows that the quantity $\rho(\widetilde R)$ of Theorem~\ref{thm:KG_revisited} simplifies to  $\rho(\widetilde R)\leq\widetilde R$.
Furthermore, it holds
\beq
\label{eq:GengenatOne}
\forall x\in \mathds S^{\bd-1},\quad W_{\widetilde R}(x,x)
=\sum_{\ell=0}^R\bp^\star_\ell c_\ell G_\ell^\beta(1)
=\sum_{\ell=0}^R d_\ell\bp^\star_\ell\,,
\eeq
since $G_\ell^\lambda(1)=d_\ell/c_\ell$. Then by Hoffman-Wielandt inequality, we have
\beq
\notag
\delta_2(\lambda(\bT_{{\widetilde R},n}),\lambda(\widetilde{\bT}_{{\widetilde R},n}))
\leq\|\widetilde{\bT}_{{\widetilde R},n}-{\bT}_{{\widetilde R},n}\|_F
=\Big[\frac1{n^2}\sum_{i=1}^nW_{\widetilde R}^2(X_i,X_i)\Big]^{\frac12}
=\frac1{\sqrt n}\Big|\sum_{\ell=0}^R d_\ell\bp^\star_\ell\Big|\,,
\eeq
almost surely. And we use this bound instead of the one of \lemref{Lemma1b_KG_revisited}.
The following result follows:
\beq\label{dromadaire}
\begin{split}
\delta_2(\lambda(\bbT_{ W_{\widetilde R}}),\lambda({\bT}_n))\leq
&
\Big[\sum_{\ell=0}^{R}d_\ell(\bp^\star_\ell)^2\Big]^{\frac12}\Big[\frac{\widetilde R\log(2\widetilde R/\alpha)}{n}\Big]^{\frac12}\\
&+\frac1{\sqrt n}\Big|\sum_{\ell=0}^R d_\ell\bp^\star_\ell\Big|
+\|\bp-\bp^R\|_{2}+\|\bp-\bp^{R}\|_\infty\Big[\frac{C_1\log(C_2/\alpha)}{n}\Big]^{\frac12}
\end{split}
\eeq
with probability at least $1-3\alpha$.

Let us study the various terms appearing in~\eqref{dromadaire}.
First, by orthonormality
$$\sum_{\ell=0}^{R} {d_\ell}|\bp_\ell^\star|^2 =\|\bp_R\|_2^2\leq \|\bp\|_2^2\leq 2$$
since $\bp_R$ is the orthogonal projection of $\bp$, and $|\bp|\leq 1$.
Next, using Cauchy-Schwarz inequality
$$\left|\sum_{\ell=0}^R d_\ell\bp^\star_\ell\right|\leq
\left(\sum_{\ell=0}^R d_\ell\right)^{1/2}\left(\sum_{\ell=0}^R d_\ell|\bp^\star_\ell|^2\right)^{1/2}
\leq\sqrt{2\widetilde R}.$$
Now $\|\bp-\bp_R\|_\infty\leq 1 +\|\bp_R\|_\infty$, with
$\|\bp_R\|_\infty\leq \sum_{\ell=0}^R |\bp^\star_\ell c_\ell|
\| G_\ell^{\beta}\|_\infty$.
But $\| G_\ell^{\beta}\|_\infty=G_\ell^{\beta}(1)$
by Formula~(4.7.1) and Theorems 7.32.1 and 7.33.1 of \cite{Szego} so 
$$\|\bp_R\|_\infty
\leq \sum_{\ell=0}^R |\bp^\star_\ell c_\ell| G_\ell^{\lambda}(1)
= \sum_{\ell=0}^R |\bp^\star_\ell |d_{\ell}\leq \sqrt{2\widetilde R}
\,.
$$
Finally,~\eqref{dromadaire} becomes
\beq \notag
\begin{split}
\delta_2(\lambda(\bbT_{ W_{\widetilde R}}),\lambda({\bT}_n))\leq &\sqrt{2}\Big[\frac{\widetilde R\log(2\widetilde R/\alpha)}{n}\Big]^{\frac12}+\frac{\sqrt{2\widetilde R}}{\sqrt n}
\\
&+\Big[\sum_{\ell>R}{d_\ell}(\bp_\ell^\star)^2 \Big]^{\frac12}
+\Big(1+\sqrt{2\widetilde R}\Big)\Big[\frac{C_1\log(C_2/\alpha)}{n}\Big]^{\frac12}
\end{split}
\eeq
Hence, since $\widetilde R\geq 1$ and $\log n\leq n$, there exists a numerical constant $C>0$ such that, with probability at least $1-3\alpha$
\beq
\label{TrTn}
\delta_2(\lambda(\bbT_{ W_{\widetilde R}}),\lambda({\bT}_n))\leq  
\Big[\sum_{\ell>R}{d_\ell}(\bp_\ell^\star)^2 \Big]^{\frac12}+
C\sqrt{\frac{\widetilde R(1+{\log(\widetilde R/\alpha)})}{n}}.
\eeq

 Adding~\eqref{biais} gives the first statement of Proposition~\ref{thm:KG_Sphere}.
 
Now let us denote by $\Omega$ the set with probability larger than $1-3\alpha$ such that the previous inequality 
is true. One has
\begin{align*}
\delta_2(\lambda({\bT}_n),\lambda^{\star})&= \delta_2(\lambda({\bT}_n),\lambda^{\star})\1_{\Omega}+\delta_2(\lambda({\bT}_n),\lambda^{\star})\1_{\Omega^c}.
\end{align*}
Observe that each $|\lambda_k(\bT _n)|$ is bounded by $\rho({\bT}_n)$ the spectral radius of ${\bT}_n$. Since $\bT_n:=(1/n)\bTheta_0$, it holds that $\rho({\bT}_n)\leq\|\bTheta_0/n\|_F\leq 1$. Then 
\beq\label{coarsebound}
\delta_2(\lambda({\bT}_n),\lambda^{\star})\leq \delta_2(\lambda({\bT}_n),0)+\delta_2(0,\lambda^{\star })\leq
\sqrt{n}+\|\bp\|_2
\eeq
which entails $\delta_2^2(\lambda({\bT}_n),\lambda^{\star})\leq 
(1+\sqrt{2})^2 n$. 
Hence, using this bound and previous inequality,
\begin{align*}
\E\delta_2^2(\lambda({\bT}_n),\lambda^{\star})&= \E\left(\delta_2^2(\lambda({\bT}_n),\lambda^{\star})\1_{\Omega}\right)+(1+\sqrt{2})^2 n\P({\Omega^c})\\
&\leq 
8\Big[\sum_{\ell>R}{d_\ell}(\bp_\ell^\star)^2 \Big]+
2C^2\frac{\widetilde R(1+{\log(\widetilde R/\alpha)})}{n}
+3\alpha (1+\sqrt{2})^2 n
\end{align*}
as soon as  $n^3\geq\widetilde R\log(2\widetilde R/\alpha)$.
We choose $\alpha={n^{-2}}$, and assume $n\geq2\widetilde R$. Then
$$\widetilde R\log(2\widetilde R/\alpha)
=\widetilde R\log(2\widetilde Rn^2)\leq \frac{n}2\log(n^3)\leq
n^3\,,$$
and
\begin{align*}
\E\left(\delta_2^2(\lambda({\bT}_n),\lambda^{\star})\right)
&\leq 8\Big[\sum_{\ell>R}{d_\ell}(\bp_\ell^\star)^2 \Big]+
2C^2\frac{\widetilde R(1+{\log(\widetilde R n^2)})}{n}
+3 (1+\sqrt{2})^2 n^{-1}\\
&\leq  8\Big[\sum_{\ell>R}{d_\ell}(\bp_\ell^\star)^2 \Big]+
C'\frac{ R^{\bd -1}\log n}{n}\\
\end{align*}
since $\widetilde R= \mathcal{O}(R^{\bd-1})$. Now we assume that $\bp$ belongs  to the Weighted Sobolev $ Z^s_{\bw_\beta}((-1,1))$. Then, using~\eqref{eq:l2decreasingSobolov}, 
for all $R$ such that $n\geq2\widetilde R$, it holds
$$\E\left(\delta_2^2(\lambda({\bT}_n),\lambda^{\star})\right)
\leq  8C(\bp,s,\bd) R^{-2s }+
C'\frac{ R^{\bd -1}\log n}{n}$$

To conclude it is sufficient to choose $R=\lfloor ({n}/{\log n })^{\frac{1}{2s+\bd-1}} \rfloor$.

\subsection{Proof of Theorem~\ref{thm:New_Main_Sphere}}
\label{proof:New_Main_Sphere}

We use the notation of the previous proofs and, in particular, the notation of Appendix~\ref{proof:KG_Sphere}. 
The heart of the proof lies in the following proposition, proved in Appendix~\ref{proof:ancienth4}. 

\pagebreak[3]

\begin{proposition}\label{prop:ancienth4}
Let $R\geq0$ such that $2\widetilde R\leq n$. It holds 
\[
\delta_2(\widehat\lambda^R,\lambda^{\star R})\leq4\delta_2(\lambda(\bbT_{W_{\widetilde R}}),\lambda({\bT}_n))
+\sqrt{2\widetilde R}\,\|\widehat{\bT}_{n}-{\bT}_{n}\|\,.
\]
\end{proposition}

Now, using inequality~\eqref{eq:BvH}, we know that 
\[
\|\widehat{\bT}_{n}-{\bT}_{n}\|\leq \frac3{\sqrt{2n}}+C_0\frac{\sqrt{\log(n/\alpha)}}n\,,
\]
 with probability at least $1-\alpha$.
\begin{remark}
In the relatively sparse model~\eqref{eq:relativelysparse}, using~\eqref{eq:BvH2}, recall that
\[
\|\widehat{\bT}_{n}-{\bT}_{n}\|\leq3\sqrt{\frac{2\zeta_n}n}+C_0\frac{\sqrt{\log(n/\alpha)}}n
\]
with probability at least $1-\alpha$. It shows that $\|\widehat{\bT}_{n}-{\bT}_{n}\|=\mathcal{O}_{\mathds P}(\sqrt{\zeta_n/n})$ under~\eqref{eq:relativelysparse}.
\end{remark}

Moreover, by~\eqref{TrTn} in proof of Proposition~\ref{thm:KG_Sphere}, for all $n^3\geq\widetilde R\log(2\widetilde R/\alpha)$,
$$\delta_2(\lambda(\bbT_{ W_{\widetilde R}}),\lambda({\bT}_n))\leq  
\Big[\sum_{\ell>R}{d_\ell}(\bp_\ell^\star)^2 \Big]^{\frac12}+
C\sqrt{\frac{\widetilde R(1+{\log(\widetilde R/\alpha)})}n}.
$$

\begin{remark}
In the relatively sparse model~\eqref{eq:relativelysparse}, it reads
\[
\delta_2(\lambda(\bbT_{ W_{\widetilde R}}),\lambda({\bT}_n))\leq  
\zeta_n\Big[\sum_{\ell>R}{d_\ell}(\bp_\ell^\star)^2 \Big]^{\frac12}+
C\zeta_n\sqrt{\frac{\widetilde R(1+{\log(\widetilde R/\alpha)})}n}.\]
with probability at least $1-\alpha$. We recall that $\bp_\ell^\star$ are the eigenvalues of $\bbT_W$.
\end{remark}

Thus there exists a numerical constant $\kappa_0>0$ such that, with probability at least $1-3\alpha$ 

$$\delta_2(\widehat\lambda^R,\lambda^{\star R})\leq4\Big[\sum_{\ell>R}{d_\ell}(\bp_\ell^\star)^2 \Big]^{\frac12}+
\kappa_0\sqrt{\frac{\widetilde R(1+{\log(\widetilde R/\alpha)})}n}.
$$
if  $n^3\geq (2\widetilde R)^3\vee\widetilde R\log(2\widetilde R/\alpha)$,
that gives the first statement of Theorem~\ref{thm:New_Main_Sphere}. The remark (Remark~\ref{rem:Thm6}) following Theorem~\ref{thm:New_Main_Sphere} can be deduced from the previous remarks of this proof.

Now let us denote by $\Omega$ the set with probability larger than $1-3\alpha$ such that the previous inequality is true. One has
\begin{align*}
\delta_2(\widehat\lambda^R,\lambda^{\star R})&= \delta_2(\widehat\lambda^R,\lambda^{\star R})\1_{\Omega}+\delta_2(\widehat\lambda^R,\lambda^{\star R})\1_{\Omega^c}
\end{align*}
As for~\eqref{coarsebound}, we can prove the coarse bound 
$$\delta_2(\widehat\lambda^R,\lambda^{\star R})\leq 
\sqrt{\widetilde R} +\|\bp\|_2\leq (1 +\sqrt2)\sqrt{\widetilde R}.$$ 
Hence, using this bound and previous inequality,
\begin{align*}
\E\left(\delta_2^2(\widehat\lambda^R,\lambda^{\star R})\right)&= \E\left(\delta_2^2(\widehat\lambda^R,\lambda^{\star R})\1_{\Omega}\right)+(1 +\sqrt2)^2{\widetilde R}\P({\Omega^c})\\
&\leq 32
\left[\sum_{\ell>R}{d_\ell}(\bp_\ell^\star)^2 \right]+
2\kappa_0^2\frac{\widetilde R\left(1+\log(\widetilde R/\alpha)\right)}{n}+3\alpha (1 +\sqrt2)^2{\widetilde R} \,,
\end{align*}
as soon as  $n^3\geq(2\widetilde R)^3\vee\widetilde R\log(2\widetilde R/\alpha)$.
We choose $\alpha={n^{-1}}$, and assume $n\geq2\widetilde  R$. Then
$\widetilde R\log(2\widetilde R/\alpha)\leq n^3$
and
\begin{align*}
\E\left(\delta_2^2(\widehat\lambda^R,\lambda^{\star R})\right)
&\leq 32
\left[\sum_{\ell>R}{d_\ell}(\bp_\ell^\star)^2 \right]+
2\kappa_0^2\frac{\widetilde R\left(1+\log(\widetilde Rn)\right)}{n}+3(1 +\sqrt2)^2\frac{\widetilde R}{n} \\
&\leq 32
\left[\sum_{\ell>R}{d_\ell}(\bp_\ell^\star)^2 \right]+
(6\kappa_0^2+18)\frac{\widetilde R\log n}{n}.
\end{align*}
This completes the proof. The same reasoning gives the second statement of Remark~\ref{rem:Thm6}.


\subsection{Proof of Proposition~\ref{prop:ancienth4}}
\label{proof:ancienth4}

$\circ\quad$Define $\Delta_R$ as follows
 \[
 \forall x,y\in\bbR^{2\widetilde R},\quad\Delta^2_R(x,y):=\min_{\sigma\in\mathfrak S_{2\widetilde R}}
 \Big\{
 \sum_{k=1}^{2\widetilde R}(x_k-y_{\sigma(k)})^2\,
 \Big\},
 \]
 where $\mathfrak S_{2\widetilde R}$ denotes the set of permutations on $[2\widetilde R]$. 
 
 \pagebreak[3]
 
 Once again, using Hardy-Littlewood rearrangement inequality \cite[Theorem 368]{hardy1952inequalities}, it holds that
  \[
 \forall x,y\in\bbR^{2\widetilde R}\ \mathrm{s.t.}\ x_1\geq\ldots\geq x_{2\widetilde R}\ \mathrm{and}\ y_1\geq\ldots\geq y_{2\widetilde R},\quad\Delta^2_R(x,y):=
  \sum_{k=1}^{2\widetilde R}(x_k-y_k)^2\,. 
  \]
 Completing with $\widetilde R$ zeros, we denote also
 \begin{align*}
  \widehat\Lambda^R&:=(\underbrace{\widehat\bp_0}_{d_0},\underbrace{\widehat\bp_1,\ldots, \widehat\bp_1}_{d_1}, \ldots,\underbrace{\widehat\bp_R,\ldots, \widehat\bp_R}_{d_R},\underbrace{0,\ldots,0}_{\widetilde R})\in\bbR^{2\widetilde R}\,,\\
  \mathrm{and}\quad
 \Lambda^{\star R}&:=(\bp^\star_0,\bp^\star_1,\ldots, \bp^\star_1, \ldots,\bp^\star_R,\ldots, \bp^\star_R,0,\ldots,0)\in\bbR^{2\widetilde R}\,.
 \end{align*}
 Since $R$ does not vary in this proof, we have denoted $\widehat\bp_\ell:=\widehat\bp_\ell^R$.
Observe that $\delta_2(\widehat\lambda^R,\lambda^{\star R})=\Delta_R(\widehat\Lambda^R,\Lambda^{\star R})$ using the property described in~\eqref{eq:Hardy_Littlewood_L2} and Hardy-Littlewood rearrangement inequality \cite[Theorem 368]{hardy1952inequalities} again.

$\circ\quad$Recall that it holds $\lambda(\bbT_{W_{\widetilde R}})=\{0,\bp_0^\star,\bp_1^\star,\ldots, \bp_1^\star, \ldots,\bp_R^\star,\ldots, \bp_R^\star\}$ where zero is the only eigenvalue with infinite multiplicity. In particular, remark that the vector $(\bp_0^\star,\bp_1^\star,\ldots, \bp_1^\star, \ldots,\bp_R^\star,\ldots, \bp_R^\star)$ belongs to $\mathcal M_R$
. We begin by defining
\beq
\label{eq:Def_P_overline}
(\overline\bp_0,\ldots,\overline\bp_R,\ldots, \overline\bp_R)\in\arg\min_{u\in\mathcal M_R}\min_{\sigma\in\mathfrak S_n}
\Big\{
\sum_{k=1}^{\widetilde R}(u_k-\lambda_{\sigma(k)}({\bT}_{n}))^2+\sum_{k=\widetilde R+1}^n\lambda_{\sigma(k)}({\bT}_{n})^2
\Big\}\,,
\eeq
where $\mathfrak S_n$ denotes the set of permutation on $[n]$. Also, define
\[
\forall x,y\in \mathds S^{\bd-1},\quad \overline{W}_{\widetilde R}(x,y)=\sum_{\ell=0}^R\overline\bp_\ell c_\ell G_\ell^\beta(\langle x,y\rangle)\,,
\]
and observe that $\lambda(\bbT_{\overline W_{\widetilde R}})=\{0,\overline \bp_0,\overline \bp_1,\ldots, \overline \bp_1, \ldots,\overline \bp_R,\ldots, \overline \bp_R\}$ where zero is the only eigenvalue with infinite multiplicity. Denote $\overline \sigma\in\mathfrak S_n$ the permutation that achieves the minimum in~\eqref{eq:Def_P_overline}. We have the following intermediate result.
\begin{lemma}
\label{lem:distance_barW_Tn}
It holds
\beq
\notag
\delta_2^2(\lambda(\bbT_{\overline W_{\widetilde R}}),\lambda({\bT}_n))=
\sum_{k=1}^{\widetilde R}(\overline \bp_k-\lambda_{\overline\sigma(k)}({\bT}_{n}))^2+\sum_{k=\widetilde R+1}^n\lambda_{\overline\sigma(k)}({\bT}_{n})^2\leq\delta_2^2(\lambda(\bbT_{ W_{\widetilde R}}),\lambda({\bT}_n))\,,
\eeq
where $(\overline \bp_\ell)_\ell$ is defined by~\eqref{eq:Def_P_overline}.
\end{lemma}
\begin{proof}
Observe that $\lambda(\bbT_{\overline W_{\widetilde R}})$ has at most $\widetilde R$ nonzero eigenvalues. Using again Hardy-Littlewood rearrangement inequality \cite[Theorem 368]{hardy1952inequalities} and~\eqref{eq:Hardy_Littlewood_L2}, one may deduce that $\delta_2^2(\lambda(\bbT_{\overline W_{\widetilde R}}),\lambda({\bT}_n))$ reads $\sum_{k=1}^{\widetilde R}(\overline \bp_k-\lambda_{\sigma(k)}({\bT}_{n}))^2+\sum_{k=\widetilde R+1}^n\lambda_{\sigma(k)}({\bT}_{n})^2$ for some permutation $\sigma\in\mathfrak S_n$. Taking the infimum leads to the left hand side equality. 

Then, observe that $\lambda(\bbT_{ W_{\widetilde R}})$ has at most $\widetilde R$ nonzero eigenvalues. Using again Hardy-Littlewood rearrangement inequality \cite[Theorem 368]{hardy1952inequalities} and~\eqref{eq:Hardy_Littlewood_L2}, one may deduce again that $\delta_2^2(\lambda(\bbT_{\overline W_{\widetilde R}}),\lambda({\bT}_n))$ reads $\sum_{k=1}^{\widetilde R}(\bp_k^\star-\lambda_{\sigma(k)}({\bT}_{n}))^2+\sum_{k=\widetilde R+1}^n\lambda_{\sigma(k)}({\bT}_{n})^2$ for some permutation $\sigma\in\mathfrak S_n$. Furthermore, recall that $(\bp_0^\star,\bp_1^\star,\ldots, \bp_1^\star, \ldots,\bp_R^\star,\ldots, \bp_R^\star)$ belongs to $\mathcal M_R$ and, hence, it is admissible to Program~\eqref{eq:Def_P_overline}. In particular, the value of the objective at this point is always greater than the minimal value. This gives the right hand side inequality.
\end{proof}

$\circ\quad$
Similarly, denote $((\widehat\bp_\ell),\widehat \sigma)$ a point that achieves the minimum in~\eqref{eq:Min_Admissible}. Now, consider \beq\notag
S:=\overline\sigma([\widetilde R])\cup\widehat\sigma([\widetilde R])\,,\eeq and $S^c:=[n]\setminus S$, and define $s:=\#S\leq2\widetilde R\leq n$. 

\pagebreak[3]

On can check that
\begin{align*}
\widehat\bp_\ell=\frac1{d_\ell}\sum_{k=\widetilde{\ell-1}}^{\widetilde \ell}\lambda_{\widehat\sigma(k)}\quad
\mathrm{and\ }\quad\overline\bp_\ell=\frac1{d_\ell}\sum_{k=\widetilde{\ell-1}}^{\widetilde \ell}\lambda_{\overline\sigma(k)}
\end{align*}
with the convention $\widetilde{-1}=1$.

$\circ\quad$Denote by $\mathfrak S_{S,n}$ the set of permutation $\sigma\in\mathfrak S_{n}$ such that $\sigma([s])=S$, $\mathfrak S_{S}$ the set of bijections from~$[s]$ onto $S$ and $\mathfrak S_{s}$ the set of permutations of $[s]$. It is clear that $\mathfrak S_{S}\simeq\mathfrak S_{s} $. Observe that
\begin{align}
\notag
(\widehat\bp_\ell)&=
\arg\min_{u\in\mathcal M_R}\min_{\sigma\in\mathfrak S_n}\Big\{
\sum_{k=1}^{\widetilde R}(u_k-\lambda_{\sigma(k)})^2+\sum_{k=\widetilde R+1}^n\lambda_{\sigma(k)}^2
\Big\}\\\notag
&=\arg\min_{u\in\mathcal M_R}\min_{\sigma\in\mathfrak S_{S,n}}\Big\{
\sum_{k=1}^{\widetilde R}(u_k-\lambda_{\sigma(k)})^2+\sum_{k=\widetilde R+1}^n\lambda_{\sigma(k)}^2
\Big\}
\end{align}
since one of the permutation $\sigma\in\mathfrak S_n$ that achieves the minimum in the first row satisfies $\sigma\in\mathfrak S_{S,n}$  and it follows that $(\widehat\bp_\ell)$ is the arg minimum of the second program. Now, separating the terms~$\lambda_{\sigma(k)}^2$ for $k>\widetilde R$, we obtain
\begin{align}
(\widehat\bp_\ell)&
\notag
=\arg\min_{u\in\mathcal M_R}\min_{\sigma\in\mathfrak S_{S,n}}\Big\{
\sum_{k=1}^{\widetilde R}(u_k-\lambda_{\sigma(k)})^2+\sum_{k=\widetilde R+1}^s\lambda_{\sigma(k)}^2+\sum_{t\in S^c}\lambda_{t}^2
\Big\}\\
&
\notag
=\arg\min_{u\in\mathcal M_R}\min_{\sigma\in\mathfrak S_{S}}\Big\{
\sum_{k=1}^{\widetilde R}(u_k-\lambda_{\sigma(k)})^2+\sum_{k=\widetilde R+1}^s\lambda_{\sigma(k)}^2+\sum_{t\in S^c}\lambda_{t}^2
\Big\}\\
&
\label{eq:DeltaR_hatP}
=\arg\min_{u\in\mathcal M_R}\min_{\sigma\in\mathfrak S_{S}}\Big\{
\sum_{k=1}^{\widetilde R}(u_k-\lambda_{\sigma(k)})^2+\sum_{k=\widetilde R+1}^s\lambda_{\sigma(k)}^2
\Big\}\,.
\end{align}
Similarly, one can check that
\[
(\overline\bp_\ell)=\arg\min_{u\in\mathcal M_R}\min_{\sigma\in\mathfrak S_{S}}\Big\{
\sum_{k=1}^{\widetilde R}(u_k-\lambda_{\sigma(k)}({\bT}_{n}))^2+\sum_{k=\widetilde R+1}^s\lambda_{\sigma(k)}({\bT}_{n})^2
\Big\}\,.
\]

$\circ\quad$Consider the restriction $\dot\Delta_R$ of $\Delta_R$ to $\bbR^s$ defined as follows
 \[
 \forall x,y\in\bbR^{s},\quad\dot\Delta^2_R(x,y):=\min_{\sigma\in\mathfrak S_{s}}
 \Big\{
 \sum_{k=1}^{s}(x_k-y_{\sigma(k)})^2
 \Big\}\,.
 \]
Using~\eqref{eq:Weyl} and Weyl's inequality \cite[page 63]{bhatia2013matrix} and by abuse of notation, note that
\begin{align*}
\dot\Delta_R((\lambda_k({\bT}_{n}))_{k\in S},(\lambda_k)_{k\in S})
&\leq\Big[\sum_{k\in S}(\lambda_{k}-\lambda_{k}({\bT}_{n}))^2\Big]^{\frac12}\leq \sqrt s\|\widehat{\bT}_{n}-{\bT}_{n}\|\,.
\end{align*}

\noindent
Moreover, using~\eqref{eq:DeltaR_hatP} and by abuse of notation, remark that
\begin{align*}
\dot\Delta^2_R((\widehat\bp_\ell),(\lambda_k)_{k\in S})
&=\min_{\sigma\in\mathfrak S_{S}}\Big\{\min_{u\in\mathcal M_R}
\sum_{k=1}^{\widetilde R}(u_k-\lambda_{\sigma(k)})^2+\sum_{k=\widetilde R+1}^s\lambda_{\sigma(k)}^2
\Big\}\,.
\\
&
\leq
\min_{\sigma\in\mathfrak S_{S}}\Big\{\sum_{k=1}^{\widetilde R}(\overline\bp_k-\lambda_{\sigma(k)})^2+\sum_{k=\widetilde R+1}^s\lambda_{\sigma(k)}^2
\Big\}\,.
\\
&=\dot\Delta^2_R((\overline\bp_\ell),(\lambda_k)_{k\in S})
\end{align*}
where
$(\widehat\bp_\ell)=(\widehat\bp_0,\widehat\bp_1,\ldots, \widehat\bp_1, \ldots,\widehat\bp_R,\ldots, \widehat\bp_R,0,\ldots,0)\in\bbR^s $ completing with $s-\widetilde R$ zeros.

\pagebreak[3]

$\circ\quad$Using~\eqref{eq:Def_P_overline}, Lemma~\ref{lem:distance_barW_Tn} and by abuse of notation, observe that
\begin{align*}
\dot\Delta^2_R((\overline\bp_\ell),(\lambda_k({\bT}_{n}))_{k\in S})
&=\min_{\sigma\in\mathfrak S_{S}}\Big\{
\sum_{k=1}^{\widetilde R}(\overline\bp_k-\lambda_{\sigma(k)}({\bT}_{n}))^2+\sum_{k=\widetilde R+1}^s\lambda_{\sigma(k)}({\bT}_{n})^2
\Big\}\,.
\\
&\leq\min_{\sigma\in\mathfrak S_{S}}\Big\{
\sum_{k=1}^{\widetilde R}(\overline\bp_k-\lambda_{\sigma(k)}({\bT}_{n}))^2+\sum_{k=\widetilde R+1}^s\lambda_{\sigma(k)}({\bT}_{n})^2
+\sum_{t\in S^c}\lambda_{t}({\bT}_n)^2
\Big\}\,.
\\
&=\min_{\sigma\in\mathfrak S_n}
\Big\{
\sum_{k=1}^{\widetilde R}(\overline\bp_k-\lambda_{\sigma(k)}({\bT}_{n}))^2+\sum_{k=\widetilde R+1}^n\lambda_{\sigma(k)}({\bT}_{n})^2
\Big\}\,,
\\
&=
\sum_{k=1}^{\widetilde R}(\overline \bp_k-\lambda_{\overline\sigma(k)}({\bT}_{n}))^2+\sum_{k=\widetilde R+1}^n\lambda_{\overline\sigma(k)}({\bT}_{n})^2\\
&=\delta^2_2(\lambda(\bbT_{\overline W_{\widetilde R}}),\lambda({\bT}_n))\\
&
\leq\delta^2_2(\lambda(\bbT_{ W_{\widetilde R}}),\lambda({\bT}_n))
\end{align*}
where we denote by $(\overline\bp_\ell)=(\overline\bp_0,\overline\bp_1,\ldots, \overline\bp_1, \ldots,\overline\bp_R,\ldots, \overline\bp_R,0,\ldots,0)\in\bbR^s$ completing with $s-\widetilde R$ zeros.

$\circ\quad$Using that $\dot\Delta_R$ is a semi-distance\textemdash in particular the triangular inequality holds, one deduces
\begin{align*}
\dot\Delta_R((\widehat\bp_\ell),(\overline\bp_\ell))
&\leq\dot\Delta_R((\widehat\bp_\ell),(\lambda_k)_{k\in S})+\dot\Delta_R((\lambda_k)_{k\in S},(\lambda_k({\bT}_{n}))_{k\in S})+\dot\Delta_R((\lambda_k({\bT}_{n}))_{k\in S},(\overline\bp_\ell))\\
&\leq2\delta_2(\lambda(\bbT_{ W_{\widetilde R}}),\lambda({\bT}_n))+\sqrt s\|\widehat{\bT}_{n}-{\bT}_{n}\|\,,
\end{align*}
combining the aforementioned inequalities.

Define $\overline\Lambda^R:=(\overline\bp_0,\overline\bp_1,\ldots, \overline\bp_1, \ldots,\overline\bp_R,\ldots, \overline\bp_R,0,\ldots,0)\in\bbR^{2\widetilde R}$ completing with $\widetilde R$ zeros, and remark that
\[
\Delta_R(\widehat\Lambda^R,\overline\Lambda^R)\leq\dot\Delta_R((\widehat\bp_\ell),(\overline\bp_\ell))
\leq2\delta_2(\lambda(\bbT_{ W_{\widetilde R}}),\lambda({\bT}_n))+\sqrt{2\widetilde R}\|\widehat{\bT}_{n}-{\bT}_{n}\|\,.
\]

$\circ\quad$It remains to bound $\Delta_R(\Lambda^{\star R},\overline\Lambda^R)$. Note that $\Delta_R(\Lambda^{\star R},\overline\Lambda^R)=\delta_2(\lambda(\bbT_{W_{\widetilde R}}),\lambda(\bbT_{\overline W_{\widetilde R}}))$.
Then, invoke Lemma~\ref{lem:distance_barW_Tn} to get that
\[
\Delta_R(\Lambda^{\star R},\overline\Lambda^R)
\leq\delta_2(\lambda(\bbT_{W_{\widetilde R}}),\lambda({\bT}_n))+\delta_2(\lambda({\bT}_n),\lambda(\bbT_{\overline W_{\widetilde R}}))\leq2\delta_2(\lambda(\bbT_{W_{\widetilde R}}),\lambda({\bT}_n))\,.
\]

Finally we obtain the following bound: 
\beq\label{ancienth4}
\delta_2(\widehat\lambda^R,\lambda^{\star R})\leq4\delta_2(\lambda(\bbT_{W_{\widetilde R}}),\lambda({\bT}_n))+\sqrt{2\widetilde R}\|\widehat{\bT}_{n}-{\bT}_{n}\|\eeq
 for all sample size $n\geq 2\widetilde R$.

\subsection{Proof of Theorem~\ref{thm:adap}}
\label{proof:adap}
In this proof we denote $D(R)=\sqrt{{\widetilde R\log n}/n}$, so that 
$B(R)=\max_{R'\in\R} \left\{\delta_2(\widehat \lambda^{R'},\widehat \lambda^{R'\wedge R})-\kappa D(R')\right\}.$
Fix some $R\in \R$. First decompose
$$\delta_2(\widehat \lambda^{\widehat R}, \lambda^{\star})\leq
\delta_2(\widehat \lambda^{\widehat R},\widehat \lambda^{\widehat R\wedge R})+
\delta_2(\widehat \lambda^{\widehat R\wedge R},\widehat \lambda^{ R})+
\delta_2(\widehat \lambda^{R}, \lambda^{\star}).$$
Using the definition of $B(R)$ and $B(\widehat R)$ it holds that
$$\delta_2(\widehat \lambda^{\widehat R}, \lambda^{\star})\leq
B( R)+\kappa D(\widehat R)+
B(\widehat R)+\kappa D( R)+
\delta_2(\widehat \lambda^{R}, \lambda^{\star}).$$
We now use the definition of $\widehat R$ to write $\delta_2(\widehat \lambda^{\widehat R}, \lambda^{\star})\leq
2B( R)+2\kappa D(R)+\delta_2(\widehat \lambda^{R}, \lambda^{\star})$.
The last term can be split in $\delta_2(\widehat \lambda^{R}, \lambda^{\star})\leq \delta_2(\widehat \lambda^{R}, \lambda^{\star R})+\delta_2(\lambda^{\star R}, \lambda^{\star })$.
Thus, 
\beq\label{bli}
\delta_2(\widehat \lambda^{\widehat R}, \lambda^{\star })\leq
2B( R)+2\kappa D(R)+\delta_2(\widehat \lambda^{R}, \lambda^{\star R})+\delta_2(\lambda^{\star R}, \lambda^{\star }).
\eeq
We shall now control the term $B(R)$.
Denote $a_+=\max(a,0)$ the positive part of any real $a$. 
Let us write
\begin{align*}
B(R)&=\max_{R'\in\R} \{\delta_2(\widehat \lambda^{R'},\widehat \lambda^{R'\wedge R})-\kappa D(R')\}\\
&\leq  \max_{R'\in \R, R'\geq R} \{\delta_2(\widehat \lambda^{R'},\widehat \lambda^{ R})-\kappa D(R')\}_+\\
&\leq \max_{R'\in \R, R'\geq R} \{\delta_2(\widehat \lambda^{R'}, \lambda^{\star R'})
+\delta_2( \lambda^{\star R'}, \lambda^{\star  R})
+\delta_2(\lambda^{\star R},\widehat \lambda^{R})
-\kappa D(R')\}_+
\end{align*}
Now $\displaystyle \delta_2^2( \lambda^{\star R'}, \lambda^{\star   R})=\sum_{k= \widetilde R+1}^{\widetilde R'}|\lambda_k^{\star }|^2\leq \delta_2^2(\lambda^{\star R}, \lambda^{\star })$.
 Then
\[
B(R)\leq 
\max_{R'\in \R, R'\geq R} \{\delta_2(\widehat \lambda^{R'}, \lambda^{\star R'})-\kappa D(R')\}_+
+\delta_2( \lambda^{\star R}, \lambda^{\star })
+\delta_2(\lambda^{\star R},\widehat \lambda^{R})\,.
\]
Finally, combining this with~\eqref{bli},
 \begin{align*}
\delta_2(\widehat \lambda^{\widehat R}, \lambda^{\star })&\leq 
 2\max_{R'\in \R, R'\geq R} \{\delta_2(\widehat \lambda^{R'}, \lambda^{\star R'})-\kappa D(R')\}_+
+2\kappa D(R)+3\delta_2(\widehat \lambda^{R}, \lambda^{\star R})+3\delta_2(\lambda^{\star R}, \lambda^{\star })\,,\\
&\leq  5\max_{R'\in \R, R'\geq R} \{\delta_2(\widehat \lambda^{R'}, \lambda^{\star R'})-\kappa D(R')\}_+
+3\delta_2( \lambda^{\star R}, \lambda^{\star })
+5\kappa D(R)\,.
\end{align*}
Now, we invoke Theorem~\ref{thm:New_Main_Sphere} and a union bound to insure that, if $n^3\geq(2\widetilde R_{\max})^3\vee\widetilde R_{\max}\log(2\widetilde R_{\max}/\alpha)$ then, with probability greater that $1-3|\R|\alpha$, it holds
 \begin{align*}
\forall R'\in \R\,, \qquad\delta_2(\widehat\lambda^{R'},\lambda^{\star R'})\leq  4\delta_2( \lambda^{\star R'}, \lambda^{\star })
+
\kappa_0\sqrt{\frac{\widetilde{R'}}n\left(1+\log\left(\frac{\widetilde{R'}}{\alpha}\right)\right)}
\end{align*}
We choose $\alpha={n^{-1-q}}$, then $\widetilde R_{\max}\log(2\widetilde R_{\max}/\alpha)
=\widetilde R_{\max}\log(2\widetilde R_{\max}n^{q+1})\leq {n}\log(n^{q+2})/2< 0.1(q+2)n^3$ since $x^{-2}\log x\leq 0.09$ and
also note that $1+\log\left({\widetilde R}/{\alpha}\right)\leq (q+3)\log n$. If $q+2\leq 10$ then it holds that $n^3>\widetilde R_{\max}\log(2\widetilde R_{\max}/\alpha)$ and with probability $1-3n^{-q}$ 
\begin{align*}
\forall R'\in \R\,, \qquad\delta_2(\widehat\lambda^{R'},\lambda^{\star R'})\leq  4\delta_2( \lambda^{\star R'}, \lambda^{\star })
+
\kappa_0\sqrt{q+3}D(R')
\end{align*}
Then, with probability $1-3n^{-q}$, provided that $\kappa\geq \kappa_0 \sqrt{q+3}$
 \begin{align*}
\delta_2(\widehat \lambda^{\widehat R}, \lambda^{\star })&\leq 
 5\max_{R'\in \R, R'\geq R} \{4\delta_2( \lambda^{\star R'}, \lambda^{\star })\}_+
+3\delta_2( \lambda^{\star R}, \lambda^{\star })
+5\kappa D(R)\\
&\leq  23\delta_2( \lambda^{\star R}, \lambda^{\star })
+5\kappa D(R)
\end{align*}
Since it holds for any $R$, the first inequality of Theorem~\ref{thm:adap} is proved by choosing $q=8$.

The second statement will follow by the same roadmap as in the end of proof~\ref{proof:New_Main_Sphere}. 
Let us denote by $\Omega$ the set with probability larger than $1-3n^{-q}$ such that the previous inequality is true, and let us find 
a coarse bound on $\delta_2^2(\widehat\lambda^R,\lambda^{\star})$.
Remind that $\delta_2(\widehat\lambda^R,\lambda^{\star R})\leq 
(1 +\sqrt2)\sqrt{\widetilde R}$ for all $R$, see~\eqref{coarsebound}. Furthermore 
$$\delta_2(\lambda^{\star R},\lambda^{\star })=\Big[\sum_{\ell>R}{d_\ell}(\bp_\ell^\star)^2 \Big]^{\frac12}\leq\|\bp\|_2
\leq\sqrt2\,.$$
Hence, using this bound and previous inequality, for all $R\in\R$,
\begin{align*}
\E\left(\delta_2^2(\widehat \lambda^{\widehat R}, \lambda^{\star })\right)&\leq\E\left(\delta_2^2(\widehat \lambda^{\widehat R}, \lambda^{\star })\1_{\Omega}\right)+(1 +2\sqrt2)^2{\widetilde R}_{\max}\P({\Omega^c})\\
&\leq 2(23)^2\delta_2^2( \lambda^{\star R}, \lambda^{\star })
+2(5)^2\kappa^2 D^2(R)+(1 +2\sqrt2)^2{\widetilde R}_{\max}3n^{-q}\\
&\leq 2(23)^2\left(\delta_2^2( \lambda^{\star R}, \lambda^{\star })
+\kappa^2 D^2(R)+n^{1-q}\right)
\end{align*}
provided that $\kappa\geq \kappa_0 \sqrt{q+3}$.
The conclusion follows, choosing for instance $q=2$.

\subsection{Proof of Proposition~\ref{cor:variancepr} }
\label{proof:variancepr}

Note that $\delta_2^2(\lambda^{\star R}, \lambda^{\star})=\sum_{k>\widetilde R}|\lambda_k^{\star}|^2=\sum_{\ell> R}d_{\ell}|\bp_\ell^{\star}|^2$ and this quantity vanishes when $R\geq D$. 
From Theorem~\ref{thm:New_Main_Sphere} and assuming that $n^3\geq(2\widetilde R)^3\vee\widetilde R\log(2\widetilde R/\alpha)$, we derive that, for $R\geq D$, it holds 
$$\delta_2(\widehat \lambda^{R}, \lambda^{\star R})\leq\kappa_0\sqrt{{\widetilde R}\left(1+\log\left({\widetilde R}/{\alpha}\right)\right)/n}$$
with probability at least $1-3\alpha$. Remark also that $\bp^R=\bp$ as soon as $R\geq D$, where $\bp^R(t):=\sum_{\ell=0}^R\bp_\ell^\star c_\ell G_\ell^\beta(t).$
We now work on the set with probability $1-3\alpha$ given by Theorem~\ref{thm:New_Main_Sphere}.

We denote 
$$\delta:=\min_{0\leq i\neq j\leq  D;\ \bp_i^\star\neq0}|\bp_i^\star-\bp_j^\star|\wedge|\bp_i^\star|>0\,,$$
and note that, for $n$ large enough, it holds $\delta_2(\widehat \lambda^{R}, \lambda^{\star R})<\delta/2$. Then there exists a permutation $\sigma^\star\in\mathfrak S_n$ such that for all $k\in[n]$,  $|\widehat\lambda^R_{\sigma^\star(k)}-\lambda_k^{\star R}|<\delta/2$. Now, observe that 
\begin{align*}
\widehat\lambda^R
&=(\underbrace{\widehat\bp_0^R}_{d_0},\underbrace{\widehat\bp_1^R,\ldots, \widehat\bp_1^R}_{d_1}, \ldots,\underbrace{\widehat\bp_D^R,\ldots, \widehat\bp_D^R}_{d_D},\ldots,\underbrace{\widehat\bp_R^R,\ldots, \widehat\bp_R^R}_{d_R},0,\dots),\\
\lambda^{\star R}
&=(\underbrace{\bp_0^\star}_{d_0},\underbrace{\bp_1^\star,\ldots, \bp_1^\star}_{d_1}, \ldots,\underbrace{\bp_D^\star,\ldots, \bp_D^\star}_{d_D},0,\dots).
\end{align*}

We deduce that if $\delta_2(\widehat \lambda^{R}, \lambda^{\star R})<\delta/2$ then for all $h,i,j,k,\ell$ such that $\bp_h^\star\neq0$ it holds
\begin{align}
\notag
\mathrm{If}\quad&
|\widehat\bp_k^R-\bp_h^\star|\vee|\widehat\bp_\ell^R-\bp_h^\star|\leq\delta/2\ (\mathrm{resp.\ }|\widehat\bp_k^R-\bp_i^\star|\vee|\widehat\bp_k^R-\bp_j^\star|\leq\delta/2)\\
\label{eq:stagestarhat}\mathrm{then}\quad&
\widehat\bp_k^R=\widehat\bp_\ell^R\ (\mathrm{resp.\ } \bp_i^\star=\bp_j^\star)\,.
\end{align}
Indeed, one $\widehat\bp_\ell^R$ cannot be at the same time at a distance less than $\delta/2$ to some $\bp_i^\star\neq0$ and at a distance less that $\delta/2$ to some $\bp_j^\star$ since these latter are both at a distance of~$\delta$. Necessarily the permutation~$\sigma^\star$ is such that the group of eigenvalues $\bp_i^\star\neq0$ of multiplicity~$d_i$ matches with the group of eigenvalues $\widehat\bp_i^R$ with the same multiplicity\textemdash recall that the multiplicities~$d_\ell$ are pairwise different since the sequence $d_\ell$ is increasing. 
Thanks to~\eqref{eq:stagestarhat} it holds
$$
\delta_2^2(\widehat \lambda^{R}, \lambda^{\star R})
=\sum_{h;\ \bp_h^\star\neq0}d_h(\widehat \bp_h^R -\bp_h^\star)^2+\sum_{\ell;\ \bp_\ell^\star=0}d_\ell(\widehat \bp_\ell^R)^2=||\widehat \bp^R -\bp||^2_2\,,
$$
noticing that $\|\widehat \bp^{R}-\bp^{ R}\|_2^2=\sum_{\ell=0}^R{d_\ell}(\widehat \bp_\ell^R -\bp_\ell^\star)^2$. It follows that if
\[
n^3\geq(2\widetilde R)^3\vee\widetilde R\log(2\widetilde R/\alpha)\quad\mathrm{and}\quad2\kappa_0\sqrt{{\widetilde R}\left(1+\log\left({\widetilde R}/{\alpha}\right)\right)/n}<\min_{0\leq i\neq j\leq  D;\ \bp_i^\star\neq0}|\bp_i^\star-\bp_j^\star|\wedge|\bp_i^\star|
\]
then $||\widehat \bp^R -\bp||_2\leq\kappa_0\sqrt{{\widetilde R}\left(1+\log\left({\widetilde R}/{\alpha}\right)\right)/n}$ with probability at least $1-3\alpha$. Again, we choose $\alpha={n^{-1-q}}$, then $\widetilde R\log(2\widetilde R/\alpha)
=\widetilde R\log(2\widetilde Rn^{q+1})\leq {n}\log(n^{q+2})/2< 0.1(q+2)n^3$ since $x^{-2}\log x\leq 0.09$ and
also note that $1+\log({\widetilde R}/{\alpha})\leq (q+3)\log n$. With $q=8$, it holds that $n^3>\widetilde R\log(2\widetilde R/\alpha)$. Now,  if
\[
n\geq 2\widetilde R\quad\mathrm{and}
\quad2\kappa_0\sqrt{11{\widetilde R}\log(n)/n}<\min_{0\leq i\neq j\leq  D;\ \bp_i^\star\neq0}|\bp_i^\star-\bp_j^\star|\wedge|\bp_i^\star|
\]
then  $||\widehat \bp^R -\bp||_2\leq\kappa_0\sqrt{11{\widetilde R}\log(n)/n}$ with probability $1-3n^{-8}$, as claimed. 

For the second statement, let us denote by $\Omega$ the set with probability larger than $1-3n^{-q}$ such that the previous inequality is true, and let us find a coarse bound on $||\widehat \bp^R -\bp||_2^2$, for instance $||\widehat \bp^R -\bp||_2\leq (1+\sqrt{2})\sqrt{\widetilde R}$. Hence, using this bound and previous inequality, it holds 
\begin{align*}
\E\left(||\widehat \bp^R -\bp||_2^2\right)&\leq\E\left(||\widehat \bp^R -\bp||_2^2\1_{\Omega}\right)+(1+\sqrt{2})^2{\widetilde R}\P({\Omega^c})\leq \frac{\kappa_0^2\,(q+3)\,{\widetilde R}\log n}{n}
+3(1+\sqrt{2})^2{\widetilde R}n^{-q}
\end{align*}
recalling that $1+\log({\widetilde R}/{\alpha})\leq (q+3)\log n$. The conclusion follows, choosing $q= 1$.

\subsection{ Proof of Corollary~\ref{u}} 
\label{proof:adaptpolynomial}

From Theorem~\ref{thm:adap}, with probability $1-3n^{-8}$
\begin{align*}
\delta_2(\widehat \lambda^{\widehat R}, \lambda^{\star})&\leq
C\min\left(\min_{R< D}\left(\delta_2(\lambda^{\star R}, \lambda^{\star })+\kappa\sqrt{\frac{\widetilde R\log n}n}\right),
\min_{R\geq D}\left(\kappa\sqrt{\frac{\widetilde R\log n}n}\right)\right)\\
&\leq
C\min\left(\min_{R< D}\left(\delta_2(\lambda^{\star R}, \lambda^{\star })+\kappa\sqrt{\frac{\widetilde R\log n}n}\right),
\kappa\sqrt{\frac{\widetilde D\log n}n}\right) 
\end{align*}
Then, with probability $1-3n^{-8}$, 
$$\delta_2(\widehat \lambda^{\widehat R}, \lambda^{\star})\leq
C\kappa\sqrt{\frac{\widetilde D\log n}n} \underset{n\to \infty}\longrightarrow 0.$$
Thus, reasonning as in proof \ref{proof:variancepr}, it holds
$$
\delta_2^2(\widehat \lambda^{\widehat R}, \lambda^{\star})
=||\widehat \bp^{\widehat R} -\bp||^2_2
=\sum_{\ell}d_\ell(\widehat \bp_\ell^{\widehat R}-\bp^\star)^2
\,,
$$
If (by contradiction) $\widehat R<D$, then
$\delta_2( \lambda^{\star \widehat R},\lambda^\star)\geq d_D|\bp_D^\star|^2>0$ and then 
$\delta_2(\widehat \lambda^{\widehat R}, \lambda^{\star})$ cannot tend to 0. Thus necessarily $\widehat R\geq D$. 
Moreover, since 
$
\delta_2^2(\widehat \lambda^{\hat R}, \lambda^{\star})
=||\widehat \bp^{\hat R} -\bp||^2_2\,
$, 
 with probability $1-3n^{-8}$
$$||\widehat \bp^{\hat R} -\bp||^2_2\leq
C^2\kappa^2{\frac{\widetilde D\log n}n} .$$

Finally we can write
\begin{align*}
\E\left(||\widehat \bp^{\widehat R} -\bp||_2^2\right)&\leq \E\left(||\widehat \bp^{\widehat R}  -\bp||_2^2\1_{\Omega}\right)+(1+\sqrt{2})^2{\widetilde R_{\max}}\P({\Omega^c})\\
&\leq C^2\kappa^2\frac{{\widetilde D}\log n}{n}
+3(1+\sqrt{2})^2{\widetilde R_{\max}}n^{-8}\leq 
(C^2\kappa^2+9)\frac{{\widetilde D}\log n}{n}.
\end{align*}

\subsection{Proof of Theorem~\ref{thm:SSCCSSversion}}
\label{app:SSCCSSversion}
The proof follows the same guidelines as in the sphere example. The only difference is that we do not have Gegenbauer polynomials but normalized Jacobi polynomials $Z_\ell$ now. In particular, we have previously used the fact that Gegenbauer polynomials are bounded. Here, the same result holds in virtue of~\eqref{eq:BorneZon}. 

To be specific, when $\bS$ is a compact symmetric space, note that
\begin{itemize}
\item $\displaystyle\sum_{r=1}^R\phi_r^2=\sum_{\mathbf r=0}^{R-1}\sqrt{d_{\mathbf r}}\,\mathrm{zon}^{\mathbf r}(e_{\bS})=\sum_{\mathbf r=0}^{R-1}{d_{\mathbf r}}=\widetilde{R-1}$ and we get that $\rho(R)\leq\widetilde{R}$ when invoking Lemma~\ref{lem:Lemma1_KG_revisited} or Theorem~\ref{thm:KG_revisited};
\item we define $\displaystyle\bp^R(t):=\sum_{\ell=0}^{R}\sqrt{d_{\ell}}\bp_\ell^\star Z_\ell(t)$ and we get that 
\begin{align*}
W_{\widetilde R}(x,y)&=\bp^R(\cos(\gamma(x,y)))\\
W_{\widetilde R}(x,x)&=\sum_{\ell=0}^{R}\sqrt{d_{\ell}}\bp_\ell^\star Z_\ell(1)=\sum_{\ell=0}^{R}{d_{\ell}}\bp_\ell^\star\,,
\end{align*}
by~\eqref{eq:BorneZon}. This identity can be used in place of~\eqref{eq:GengenatOne}.
\end{itemize}
Using these inequalities and following the same guidelines as in the sphere example, one can prove the result.

\section{Computational Considerations}

\subsection{Proof of Theorem~\ref{thm:TrueComplexity}}
\label{proof:TrueComplexity}
Without loss of generality, assume that $\lambda_1\geq\lambda_2\geq\ldots
 \geq\lambda_n$. 
Similarly, let $u\in\mathcal M_R$ 
and remember that we can group the coordinates of $u$ in groups of sizes $d_\ell$ for $\ell=0,\ldots,R$. Reordering by decreasing order, there exists $\tau\in\mathfrak S_{R+1}$ such that 
\begin{align*}
\underbrace{u_{\widetilde{\tau(1)-1}+1}=\ldots=u_{\widetilde{\tau(1)}}}_{d_{\tau(1)}}\geq\ldots\geq& \underbrace{u_{\widetilde{\tau(q)-1}+1}=\ldots=u_{{\widetilde{\tau(q)}}}}_{d_{\tau(q)}}\geq0>\\
&\underbrace{u_{{\widetilde{\tau(q+1)-1}+1}}=\ldots=u_{{\widetilde{\tau(q+1)}}}}_{d_{\tau(q+1)}}\geq\ldots\geq \underbrace{u_{{\widetilde{\tau(R)-1}+1}}=\ldots=u_{{\widetilde{\tau({R+1})}}}}_{d_{\tau({R+1})}}\,,
\end{align*}
for some $q\in\bbN$. We may consider that $q=0$ and respectively $q={R+1}$ in degenerate cases when all the coefficients are negative and respectively non negative. Remember that $u\in\bbR^{\widetilde R}$ and set $u_k=0$ for $k>\widetilde R$ such that, completing with zeros, consider that $u\in\bbR^{n}$. One has 
\begin{align*}
\underbrace{u_{\widetilde{\tau(1)-1}+1}=\ldots=u_{\widetilde{\tau(1)}}}_{d_{\tau(1)}}\geq\ldots\geq& \underbrace{u_{\widetilde{\tau(q)-1}+1}=\ldots=u_{{\widetilde{\tau(q)}}}}_{d_{\tau(q)}}\geq \underbrace{u_{\widetilde R+1}=\ldots=u_n}_{n-\widetilde R}=0>\\
&\underbrace{u_{{\widetilde{\tau(q+1)-1}+1}}=\ldots=u_{{\widetilde{\tau(q+1)}}}}_{d_{\tau(q+1)}}\geq\ldots\geq \underbrace{u_{{\widetilde{\tau({R+1})-1}+1}}=\ldots=u_{{\widetilde{\tau({R+1})}}}}_{d_{\tau({R+1})}}\,.
\end{align*}
Note that 
\beq
\label{eq:HL_Proof}
\min_{\sigma\in\mathfrak S_n}
\Big\{
\sum_{k=1}^{\widetilde R}(u_k-\lambda_{\sigma(k)})^2+\sum_{k=\widetilde R+1}^n\lambda_{\sigma(k)}^2
\Big\}
=\delta_2^2((\lambda_k)_{k=1}^n,(u_k)_{k=1}^n)
=\min_{\sigma'\in\mathfrak S_n}
\Big\{
\sum_{k=1}^{n}(u_{\sigma'(k)}-\lambda_{k})^2
\Big\}\,,
\eeq
taking $\sigma'=\sigma^{-1}$. Using Hardy-Littlewood rearrangement inequality \cite[Theorem 368]{hardy1952inequalities}, it is standard to observe that
\begin{align*}
\eqref{eq:HL_Proof}=
&\underbrace{(u_{\widetilde{\tau(1)-1}+1}-\lambda_{1})^2+\cdots+(u_{\widetilde{\tau(1)}}-\lambda_{d_{\tau(1)}})^2}_{d_{\tau(1)}}+\cdots\\
&+ \underbrace{(u_{\widetilde{\tau(q)-1}+1}-\lambda_{d_{\tau(1)}+\cdots+d_{\tau(q-1)}+1})^2+\cdots+(u_{{\widetilde{\tau(q)}}}-\lambda_{d_{\tau(1)}+\cdots+d_{\tau(q)}})^2}_{d_{\tau(q)}}\\
&+\underbrace{\lambda^2_{d_{\tau(1)}+\cdots+d_{\tau(q)}+1}+\cdots
+\lambda^2_{d_{\tau(1)}+\cdots+d_{\tau(q)}+n-\widetilde R}}_{n-\widetilde R}\\
&+\underbrace{(u_{{\widetilde{\tau(q+1)-1}+1}}-\lambda_{d_{\tau(1)}+\cdots+d_{\tau(q)}+n-\widetilde R+1})^2+\cdots
+(u_{{\widetilde{\tau(q+1)}}}-\lambda_{d_{\tau(1)}+\cdots+d_{\tau(q+1)}+n-\widetilde R})^2}_{d_{\tau(q+1)}}\\
&+\cdots\\
&+ \underbrace{(u_{{\widetilde{\tau({R+1})-1}+1}}-\lambda_{d_{\tau(1)}+\cdots+d_{\tau({R})}+n-\widetilde R})^2+\cdots
+(u_{{\widetilde{\tau({R+1})}}}-\lambda_n)^2}_{d_{\tau({R+1})}}\,.
\end{align*}
Hence a permutation $\sigma'$ achieving the minimum in~\eqref{eq:HL_Proof} is given by
\[
\sigma^{-1}=\sigma'=
\left(\begin{array}{cc}
k & \sigma'(k)\\
1 & \widetilde{\tau(1)-1}+1 \\
\vdots & \vdots \\
d_{\tau(1)} & \widetilde{\tau(1)} \\
\vdots & \vdots \\
d_{\tau(1)}+\cdots+d_{\tau(q-1)}+1 & \widetilde{\tau(q)-1}+1 \\
\vdots & \vdots \\
d_{\tau(1)}+\cdots+d_{\tau(q)} & \widetilde{\tau(q)} \\
d_{\tau(1)}+\cdots+d_{\tau(q)} +1 & \widetilde R+1 \\
\vdots & \vdots \\
d_{\tau(1)}+\cdots+d_{\tau(q)} +n-\widetilde R & n \\
\vdots & \vdots \\
d_{\tau(1)}+\cdots+d_{\tau(R)}+n-\widetilde R & \widetilde{\tau({R+1})-1}+1\\
\vdots & \vdots\\
n & \widetilde{\tau({R+1})}\\
\end{array}\right)
\]
Remark that this permutation can be explicitly written given $\tau\in\mathfrak S_{R+1}$ and $q\in[0,R]$. It follows that the set of permutations $\sigma'$ achieving the minimum in the right hand side of~\eqref{eq:HL_Proof} is in one to one correspondence with a subset of $\mathfrak S_{R+2}$. Since $\sigma=\sigma'^{-1}$ in~\eqref{eq:HL_Proof} the same result holds true for the permutation $\sigma$ achieving the minimum of the left hand side of~\eqref{eq:HL_Proof}, proving the result. We define $\mathfrak H_R$ has the set of permutation $\sigma$ achieving the minimum of the left hand side of~\eqref{eq:HL_Proof}. The proof given here is constructive and it gives an explicit expression of~$\mathfrak H_R$.

\end{document}